\documentclass[11pt, a4paper]{book}
\makeatletter
\@twosidefalse
\@mparswitchfalse
\makeatother

\usepackage[utf8]{inputenc}
\usepackage{latexsym, amsfonts, amsmath, amsthm, amssymb}
\usepackage{mathrsfs}	
\usepackage{bbm}		
\usepackage{hyperref}
\usepackage{graphicx}
\usepackage[english]{babel}
\usepackage{enumitem}	
\usepackage{longtable}	
\usepackage{tikz}		
\usetikzlibrary{arrows, calc}	
\usepackage{Asettings}		
\usepackage{MScTitel}		
\usepackage{page}			

\allowdisplaybreaks

\VorNachName{Jonas Reiner Brehmer}
\GebDatum{}
\GebOrt{}
\MatNr{}
\TitelzeileI{Elicitability and its Application in Risk Management}
\Betreuer{Prof. Dr. Martin Schlather, Dr. Kirstin Strokorb}
\AbgDatum{April 3, 2017}

\begin{document}
\frontmatter
\TitelSeite
\pagestyle{empty}
\chapter*{Declaration of Authorship}
\thispagestyle{empty}
I declare that I have authored this thesis independently, that I have not used other than the declared sources and resources and that I have explicitly marked all material which has been quoted either literally or by content from the used sources. The thesis in this form has not been submitted to an examination body and has not been published.
\vspace{.5\baselineskip}

\vspace{4\baselineskip}
\begin{center}
\parbox{.8\textwidth}{Mannheim, 03.04.2017 \hfill Jonas Reiner Brehmer}
\end{center}

\cleardoublepage
\chapter{Acknowledgement}
\thispagestyle{empty}
At first, I would like to express my sincere gratitude to my supervisors for their guidance. The extensive feedback provided by Kirstin Strokorb and the fruitful discussions I had with her significantly improved this thesis. I am grateful to Martin Schlather for his valuable advice as well as his encouragement to focus on statistics during my master program.
Moreover, I would like to thank Ria Grindel, Todor Germanov and Irina Toncheva for helpful comments and discussions and Torsten Pook and Nicholas Schreck for proofreading. Furthermore, I thank Vadim Alekseev for contributing the idea for Example~\ref{Thm:ExCounterBoundedness}.
Finally, I am obliged to my family, who supported me during my studies.

\vspace{1.5\baselineskip}
\begin{center}
\parbox{.9\textwidth}{Mannheim, 03.04.2017}
\parbox{.9\textwidth}{Jonas Reiner Brehmer}
\end{center}

\pagestyle{headings}
\setcounter{tocdepth}{1}
\tableofcontents

\chapter{Introduction}

Many decision makers encounter situations in which they face uncertain future events, for instance rainfall on the next day, value of a stock in one week, or customers during the next month. This uncertainty is usually modelled by a random variable $Y$ with unknown distribution function. The decisions are then guided by forecasts of a real-valued property of $Y$ and since there are usually different possibilities to obtain such forecasts, it is necessary to assess their accuracy. One approach to do this consists of comparing the forecasts to realizations of $Y$ using a scoring function $S$. One of the most important theoretical aspects which has to be considered when performing such a comparison is the concept of elicitability. Elicitability is a property of real-valued (or $\mathbb{R}^k$-valued) mappings defined on a set of distribution functions $\mathcal{F}$. These mappings are called functionals in this thesis and represent statistical properties of a distribution, e.g. mean, variance, or median. A functional $T$ is called elicitable if there exists a scoring function $S$ such that for any $F \in \mathcal{F}$ the expectation $\mathbb{E}S(x,Y)$, where $Y$ is a random variable with distribution $F$, takes its unique minimum at $x=T(F)$. Such a scoring function is then called strictly consistent for $T$. Hence, when a statistical property is modelled by a functional $T$ and different forecasters issue reports of this functional, each report $x$ is assessed by calculating the score $S(x,y)$, where $y$ is a realization of $Y$. It was argued by Gneiting~\cite{GneitingPoints} that in such a setting $T$ should be elicitable and $S$ should be strictly consistent for $T$. This is because for such $S$, and if the reward for lower scores is higher, the forecasters will minimize the expected score in order to maximize their expected payoff. Therefore, the forecasters have an incentive to report the value of $T$ for their subjective distribution function. 

Although many functionals like for example expectations or quantiles are elicitable, other functionals fail to have this property, with the most prominent example being the variance. Hence, when comparing the accuracy of forecasts using scoring functions, it is necessary to check if the related functional is elicitable. Since it is an important task in quantitative risk management to perform comparative backtests in order to compare different risk estimation procedures, this insight led to the opinion that elicitability is a desirable property for a law-invariant risk measure.

The popular risk measure Value at Risk is elicitable, but has the major drawback that it is not coherent in the sense of Artzner et al.~\cite{ArtznerCoherent}. Conversely, the coherent risk measure Expected Shortfall, which was proposed as a replacement for Value at Risk, fails to be elicitable as shown by~\cite{GneitingPoints}. One solution to this dilemma is the use of Expectile Value at Risk, a risk measure which is both coherent and elicitable. Another solution is offered by considering multi-dimensional functionals. For instance, the vector-valued functional consisting of the mean and the variance is elicitable, although the variance alone fails to have this property. Consequently, such non-elicitable functionals which are a component of an elicitable functional are called jointly elicitable. Their existence underlines the important role played by vector-valued functionals and led to the question whether Expected Shortfall is also jointly elicitable.

An affirmative answer was given by Fissler and Ziegel~\cite{FissZieg}, who proved that Expected Shortfall is elicitable if combined with Value at Risk. In fact, they even showed that all spectral risk measures with finite spectrum, i.e. finite convex combinations of Expected Shortfall at different levels, are jointly elicitable. Moreover, they are one of the first to establish important results for vector-valued functionals. Their article constitutes the main motivation for this thesis together with other relevant developments related to elicitability, including, but not limited to, Gneiting~\cite{GneitingPoints}, Lambert~\cite{Lambert}, and Steinwart et al.~\cite{SteinwartPasinWilliam}.
This thesis reviews the most important results, examples, and applications which are found in the recent literature. Moreover, we also contribute our own examples and findings in order to give the reader a well-founded overview of the topic as well as of the most used tools and techniques. The covered material includes necessary and sufficient conditions for strictly consistent scoring functions, several elicitable as well as non-elicitable functionals and the use of elicitability in forecast comparison, regression, and estimation. Since it is beyond the scope of this thesis to elaborate on all aspects of the theory on elicitability, we give references to the literature wherever we do not go into details or omit material. 

The thesis is structured as follows. In the first chapter we introduce notation and present classical results in the theory of elicitability as well as detailed examples of functionals and (strictly) consistent scoring functions. Moreover, we consider two versions of Osband's principle, a result which connects strictly consistent scoring functions to identification functions. In Chapter~\ref{Chapter2} we present key aspects of probabilistic forecasting in order to illustrate its differences to point forecasting. Moreover, we discuss applications of elicitability in regression and estimation and introduce quantiles and expectiles, two types of elicitable functionals which are used to define risk measures in Chapter~\ref{Chapter3}. Moreover, the third chapter introduces the notions of coherent and convex measures of risk and continues by discussing the properties of the three risk measures Value at Risk, Expected Shortfall and Expectile Value at Risk. A proof that spectral risk measures with finite spectrum are jointly elicitable is presented. A discussion of comparative backtesting for risk measures concludes the third chapter. Finally, Chapter~\ref{Chapter4} characterizes strictly consistent scoring functions for vector-valued functionals having elicitable components. Furthermore, two different perspectives on non-elicitable functionals are discussed. A discussion about the presented results and open problems ends the thesis. Auxiliary results are moved to the appendix.

\mainmatter



\chapter{Elicitability and identifiability}
\label{Chapter1}

This chapter is meant to give a short introduction to the concepts of elicitability and identifiability. It starts with the most important definitions needed throughout the thesis and proceeds with some standard results and examples to make the ideas clearer. Finally, a multi-dimensional version of \textit{Osband's principle}, which builds a bridge between elicitability and identifiability, is presented. Most results, and especially the used notation, can be found in Fissler and Ziegel~\cite{FissZieg}. Supplementary material to the article \cite{FissZieg} is presented in~\cite{FissZiegArxiv}.

\section{Definitions and framework} 
\label{Sec:DefinitionsFramework}

We start with the underlying probability framework. Fix a dimension $d$ and a set $\mathsf{O} \subseteq \mathbb{R}^d$ and equip it with the Borel $\sigma$-algebra $\mathcal{O}$, which is the trace of the Borel $\sigma$-algebra of $\mathbb{R}^d$ (see for example Klenke~\cite[Ch. 1.1]{KlenkeProb}). A Borel probability measure on the measurable space $(\mathsf{O}, \mathcal{O})$ is denoted by $P$, and we also refer to it by its cumulative distribution function $F_P$. The function $F_P: \mathbb{R}^d \rightarrow \left[0,1\right]$ is defined via $F_P(x) = P( (-\infty, x] \cap \mathsf{O})$, where $(-\infty, x] = (-\infty, x_1] \times \ldots \times (-\infty, x_d] \subset \mathbb{R}^d$. A class of distribution functions over $(\mathsf{O}, \mathcal{O})$ is denoted by $\mathcal{F}$. Finally, fix another dimension $k$ and a set $\mathsf{A} \subseteq \mathbb{R}^k$. We represent a property of a distribution function $F \in \mathcal{F}$ by using a functional $T : \mathcal{F} \rightarrow \mathsf{A}$ and its image is denoted via $T(\mathcal{F}) := \left\lbrace x \in \mathsf{A} \mid \exists F \in \mathcal{F} : T(F) = x \right\rbrace$. We use the symbols $\mathsf{A}$ and $\mathsf{O}$ because these sets are also called \textit{action domain} and \textit{observation domain}, respectively. An interpretation which justifies these terms is presented in Section~\ref{Sec:ForecastAndBacktest}. \\
\par
Following the probability framework, we now turn to integrability. A function $h: \mathsf{O} \rightarrow \mathbb{R}$ is called $\mathcal{F}$-integrable if it is integrable with respect to all $F \in \mathcal{F}$. Similarly, a function $g: \mathsf{A} \times \mathsf{O} \rightarrow \mathbb{R}$ is said to be $\mathcal{F}$-integrable if $g(x, \cdot)$ is integrable for all $F \in \mathcal{F}$ and all $x \in \mathsf{A}$. For $\mathbb{R}^d$-valued functions, the integral is defined componentwise and such functions are called $\mathcal{F}$-integrable if each component is $\mathcal{F}$-integrable. In order to simplify notation and make clear with respect to which probability measure (or distribution function) a function is integrated, we introduce the mapping 
\begin{equation*}
\bar{g}: \mathsf{A} \times \mathcal{F} \rightarrow \mathbb{R}, \quad (x,F) \mapsto \bar{g}(x,F) = \int_\mathsf{O} g(x,y) \, \mathrm{d} F(y)
\end{equation*}
for an $\mathcal{F}$-integrable function $g$, where the integral $\mathrm{d}F(y)$ denotes the Lebesgue-Stieltjes integral. Since there is a probability measure $P$, of which $F$ is the distribution function, $\mathrm{d}F(y)$ is the Lebesgue integral with respect to $P$. If the set over which the integral is calculated is not specified, the whole space on which the measure is defined is meant. \\
\par
If we take $n \in \mathbb{N}$ and consider $g: \mathsf{A} \times \mathsf{O} \rightarrow \mathbb{R}^n$, we can fix $y \in \mathsf{O}$ or $F \in \mathcal{F}$ to obtain the mappings $g(\cdot, y)$ or $\bar{g}( \cdot , F)$. If $g(\cdot, y)$ is sufficiently smooth, we denote its $i$-th partial derivative via $\partial_i g(\cdot, y)$ for $i = 1, \ldots, k$. Moreover, we define the gradient 
of $g(\cdot, y)$ via $\nabla g(\cdot, y) := ( \partial_1 g(\cdot, y), \ldots, \partial_k g(\cdot, y) )^\top$. The same notation is used for $\bar{g}(\cdot, F)$. A function with domain $\mathsf{A}$ is called (partially) differentiable if its partial derivatives exist for all interior points of $\mathsf{A}$, denoted by $\intr(\mathsf{A})$. It is called continuously differentiable if its partial derivatives exist and are continuous functions for all points in $\intr(\mathsf{A})$. \\
\par
Using the previous notation, we define the main concepts of this thesis, namely elicitability and identifiability. We use the common definitions, which were established in the articles of Gneiting~\cite{GneitingPoints} and Steinwart et al.~\cite{SteinwartPasinWilliam} among others. In contrast to Fissler and Ziegel~\cite{FissZieg}, we suppress the dimensionality of the range of $T$, that is, we call a $k$-dimensional functional elicitable instead of $k$-elicitable and identifiable instead of $k$-identifiable.

\begin{definition}  \label{Def:ScoringFunction}
A \textit{scoring function} is an $\mathcal{F}$-integrable function $S: \mathsf{A} \times \mathsf{O} \rightarrow \mathbb{R}$. It is $\mathcal{F}$-\textit{consistent} for a functional $T: \mathcal{F} \rightarrow \mathsf{A}$ if for all $x \in \mathsf{A}$ and $F \in \mathcal{F}$ we have $\bar{S}(x, F) \geq \bar{S}(T(F), F)$. It is \textit{strictly} $\mathcal{F}$-\textit{consistent} for $T$ if it is $\mathcal{F}$-consistent for $T$ and for all $x \in \mathsf{A}$ and $F \in \mathcal{F}$ the equality $\bar{S}(x,F) = \bar{S}(T(F),F)$ implies $x = T(F)$.
\end{definition}

\begin{definition}[Elicitability]   \label{Def:Elicitability}
A functional $T: \mathcal{F} \rightarrow \mathsf{A} \subseteq \mathbb{R}^k$ is called \textit{elicitable} if there exists a strictly $\mathcal{F}$-consistent scoring function for $T$.
\end{definition}

\begin{definition}
An \textit{identification function} is an $\mathcal{F}$-integrable function $V: \mathsf{A} \times \mathsf{O} \rightarrow \mathbb{R}^k$. It is an $\mathcal{F}$-identification function for a functional $T: \mathcal{F} \rightarrow \mathsf{A} \subseteq \mathbb{R}^k$ if for all $F \in \mathcal{F}$ it holds that $\bar{V}(T(F),F) = 0$. It is a \textit{strict} $\mathcal{F}$-\textit{identification} function for $T$ if for all $x \in \mathsf{A}$ and $F \in \mathcal{F}$ we have $\bar{V}(x,F) = 0$ if and only if  $x = T(F)$.
\end{definition}

\begin{definition}[Identifiability]
A functional $T: \mathcal{F} \rightarrow \mathsf{A} \subseteq	\mathbb{R}^k$ is called \textit{identifiable} if there exists a strict $\mathcal{F}$-identification function for $T$.
\end{definition}

\begin{remark}
Note that $\mathcal{F}$-consistent scoring functions as well as $\mathcal{F}$-identification functions always exist. All constant functions are scoring functions and they are consistent for any functional $T: \mathcal{F} \rightarrow \mathsf{A}$. Similarly, the function which is zero everywhere and defined on $\mathsf{A}$ is a trivial $\mathcal{F}$-identification function for any $T$.
\end{remark}

Wherever necessary, we follow \cite{FissZieg} and impose the following assumption on the functions $S$ and $V$.
\begin{assumption}  \label{Thm:AssumptionLocSV}
Let $V$ be an identification function and $S$ a scoring function for a functional $T: \mathcal{F} \rightarrow \mathsf{A}$. For all $x \in \mathsf{A}$ the functions $V(x,\cdot)$ and $S(x,\cdot)$ are locally bounded. Moreover, for all $y \in \mathsf{O}$  the function $V(\cdot, y)$ is locally Lebesgue-integrable. We say that a property, e.g. boundedness, is satisfied locally if it is satisfied on all compact sets.
\end{assumption}

For an identification function we add the notion of orientation, which is quite intuitive in the one-dimensional case. Let $T: \mathcal{F} \rightarrow \mathsf{A} \subseteq \mathbb{R}$ be a functional and $V: \mathsf{A} \times \mathsf{O} \rightarrow \mathbb{R}$ an $\mathcal{F}$-identification function for $T$. Then $V$ is called \textit{oriented} if for all $x \in \mathsf{A}$ and $F \in \mathcal{F}$
\begin{equation*}
\bar{V}(x, F) > 0 \quad \Longleftrightarrow \quad x > T(F).
\end{equation*}
Generalizing this notion for $\mathsf{A} \subseteq \mathbb{R}^k$ with $k\geq 1$ leads to the following definition which is introduced in Fissler and Ziegel~\cite{FissZiegArxiv}.
\begin{definition}[Orientation] \label{Def:Orientation}
Let $T: \mathcal{F} \rightarrow \mathsf{A} \subseteq \mathbb{R}^k$ be a functional with a strict $\mathcal{F}$-identification function $V: \mathsf{A} \times \mathsf{O} \rightarrow \mathbb{R}^k$. Then $V$ is called \textit{oriented} strict $\mathcal{F}$-identification function if for all $F \in \mathcal{F}$, $v \in \mathbb{S}^{k-1}$ and $s \in \mathbb{R}$ such that $T(F) + sv \in \mathsf{A}$ it holds that
\begin{equation*}
v^\top \bar{V}(T(F) + sv, F) > 0 \quad \Longleftrightarrow \quad s > 0 ,
\end{equation*}
where $\mathbb{S}^{k-1} := \lbrace x \in \mathbb{R}^k \mid \Vert x \Vert = 1 \rbrace$ is the $(k-1)$-sphere in $\mathbb{R}^k$.
\end{definition}
Looking at this definition, we immediately obtain the one-dimensional version of orientation by observing that $\mathbb{S}^0 = \lbrace -1, 1 \rbrace$ and setting $s = x - T(F)$.

\section{Basic results and examples} 

This section presents central properties of $\mathcal{F}$-consistent scoring functions and $\mathcal{F}$-identification functions. Moreover, it contains some examples of functionals which are elicitable and some which fail to have this property. We begin with one of the simplest functionals known to be elicitable and state our own arguments to prove this fact.

\begin{example}[Elicitability of the mean]  \label{Thm:ExMeanIsElicitable}
It is well known that the mean of a random variable $Y$ with finite second moments minimizes the function $f(x) = \mathbb{E}(Y-x)^2$. Therefore, we would suggest that the mean is an elicitable functional. To make this rigorous, let $(\Omega, \mathscr{A}, \mathbb{P})$ be a probability space and $Y : \Omega \rightarrow \mathsf{O} \subseteq \mathbb{R}^d$ a random vector. Define $\mathcal{F}$ to be a class of distribution functions on $ (\mathsf{O},\mathcal{O} )$ for which all marginal distributions have finite second moments and set $\mathsf{A} := \mathbb{R}^d$. Let the mean functional be defined via
\begin{equation} \label{Eqn:MeanFunctionalDef}
T: \mathcal{F} \rightarrow \mathsf{A}, \quad F \mapsto T(F) := \int_{\mathsf{O}} y \, \mathrm{d} F(y)
\end{equation}
and the scoring function via $S: \mathsf{A} \times \mathsf{O} \rightarrow \mathbb{R}$, $(x,y) \mapsto \Vert y-x \Vert^2$. Now fix $x \in \mathsf{A}$, $F \in \mathcal{F}$ and set $t := T(F)$. Then for $Y =^d F$ under $\mathbb{P}$ we obtain
\begin{align*}
\bar{S}(x,F) - \bar{S}(t,F) &= \mathbb{E} \Vert Y-x \Vert^2 - 2 \mathbb{E}(Y-x)^\top (Y-t) + \mathbb{E}\Vert Y-t \Vert^2 \\
&\quad + 2 \mathbb{E}(Y-x)^\top (Y-t) - 2\mathbb{E} \Vert Y-t \Vert^2 \\
&= \Vert x-t \Vert^2 + 2 (x-t)^\top \mathbb{E}(Y-t) = \Vert x-t \Vert^2
\end{align*}
because $t = T(F) = \mathbb{E}Y$. Hence, $\bar{S}(x,F) \geq \bar{S}(T(F),F)$ and equality implies that $x = T(F)$. Consequently, $S$ is a strictly $\mathcal{F}$-consistent scoring function for $T$ and the mean is elicitable for the class $\mathcal{F}$. \\
One question which now naturally arises is: Can we get rid of the assumption of finite second moments? Indeed, this is possible and only the existence of first moments is required if the scoring function $S$ is modified. To do so, recall that a strictly convex and differentiable function $f$ has the property $f(r) - f(s) > \nabla f(s)^\top (r-s)$ for $r \neq s$, a fact which is extensively used in \cite{FissZieg}. So let $\mathcal{F}'$ be a class of distribution functions for which the marginal distributions have finite first moments. Letting $f: \mathsf{A} \rightarrow \mathbb{R}$ be a strictly convex and differentiable function, we redefine $S(x,y) := -f(x)  - \nabla f(x)^\top (y-x)$ and extend the definition of $T$ in (\ref{Eqn:MeanFunctionalDef}) to the new class $\mathcal{F}'$. Again, let $x \in \mathsf{A}$ and $F \in \mathcal{F}'$ be arbitrary. For $Y =^d F$ under $\mathbb{P}$ and $t:= T(F)$ we calculate
\begin{align*}
\bar{S}(x,F) - \bar{S}(t,F) &= f(t) - f(x) - \nabla f(x)^\top \mathbb{E}(Y - x) + \nabla f(t)^\top \mathbb{E}(Y-t) \\
&= f(t) - f(x) - \nabla f(x)^\top \mathbb{E}(Y - x)\\
&= f(t) - f(x) - \nabla f(x)^\top (t-x) \geq 0,
\end{align*}
where $t=\mathbb{E}Y$ is used. Consequently, $\bar{S}(x,F) \geq \bar{S}(T(F),F)$ and due to the strict convexity of $f$, equality holds if and only if $x = T(F)$. Therefore, $S$ is a strictly $\mathcal{F}'$-consistent scoring function for $T$ and the mean is elicitable for the wider class $\mathcal{F}'$, too.
\end{example}

From the previous example it is immediately clear that expectations of transformations of a random variable $Y$, for example $\mathbb{E}\exp (Y)$ or moments of $Y$, are elicitable. Moreover, the next theorem states that this can be extended to ratios of expectations having the same denominator. Similar results can be found in Gneiting~\cite[Thm. 8]{GneitingPoints} for the one-dimensional case and in Frongillo and Kash~\cite[Thm. 13]{FrongilloKash} for the case $ \mathsf{A} \subseteq \mathbb{R}^k$.

\begin{theorem}[Elicitability of ratios of expectations]  \label{Thm:GenExpectElicit}
For $k \in \mathbb{N}$ and a class of distribution functions $\mathcal{F}$ let $h : \mathsf{O} \rightarrow \mathbb{R}^k$ and $q : \mathsf{O} \rightarrow (0, \infty)$ be $\mathcal{F}$-integrable functions. Then the functional $T$ defined via
\begin{equation*}
T : \mathcal{F} \rightarrow \mathsf{A} \subseteq \mathbb{R}^k, \quad F \mapsto T(F) := \left( \frac{ \bar{h}_1 (F)}{ \bar{q}(F) } , \ldots, \frac{ \bar{h}_k (F)}{ \bar{q}(F) } \right)^\top
\end{equation*}
is elicitable. Moreover, $\mathcal{F}$-consistent scoring functions $S : \mathsf{A} \times \mathsf{O} \rightarrow \mathbb{R}$ are given by
\begin{equation*}
S(x,y) = - f(x)q(y) - \nabla f(x)^\top ( h(y) - q(y) x) ,
\end{equation*}
where $f: \mathsf{A} \rightarrow \mathbb{R}$ is a convex function. The function $S$ is strictly consistent if $f$ is strictly convex.
\end{theorem}

\begin{proof}
At first, we note that $T$ is well-defined since $\bar{q}(F)$ is strictly positive for any $F \in \mathcal{F}$. Moreover, $S$ is $\mathcal{F}$-integrable because $h$ and $q$ have this property. To show consistency, fix $F \in \mathcal{F}$, $x \in \mathsf{A}$ and set $t := T(F) = \bar{h}(F) / \bar{q}(F)$. Similar to Example~\ref{Thm:ExMeanIsElicitable} we calculate
\begin{align*}
\bar{S}(x, F) - \bar{S}(t, F) &= \bar{q}(F) f(t) - \bar{q}(F) f(x) - \nabla f(x)^\top ( \bar{h}(F) - \bar{q}(F) x) \\
&\quad+ \nabla f(t)^\top ( \bar{h}(F) - \bar{q}(F) t) \\
&= \bar{q}(F) (f(t) - f(x)) - \bar{q}(F) \nabla f(x)^\top ( t- x ) \\
&=\bar{q}(F) \big[ f(t) - f(x) - \nabla f(x)^\top (  t- x ) \big] \geq 0 ,
\end{align*}
where the inequality follows from the convexity of $f$. Therefore, $S$ is an $\mathcal{F}$-consistent scoring function for $T$. If $f$ is strictly convex and $x \neq t$ holds, the inequality is strict and $S$ is strictly consistent. This implies that $T$ is elicitable and finishes the proof.
\end{proof}

\begin{remark}  \label{Rem:MeanNecessaryCond}
Under additional regularity assumptions it can be shown that every strictly consistent scoring function for the mean is of the form $S(x,y) = -f(x)  - \nabla f(x)^\top (y-x)$ for some strictly convex function $f$. The precise result for the case $k=1$ can be found in \cite[Thm. 7]{GneitingPoints}, together with a list of references to similar results. This list includes Osband and Reichelstein~\cite[Sec. 3]{OsbandReichel} who also treat the case $k > 1$. A similar characterization can be done for ratios of expectations having the same denominator. In this situation, all strictly consistent scoring functions are of the form presented in Theorem~\ref{Thm:GenExpectElicit}. For the case $k=1$ we refer to \cite{GneitingPoints} again and the case $k\geq 1$ is treated in \cite{FissZieg}, using techniques presented in Section~\ref{Sec:FunctionalsComponents}.
\end{remark}

\begin{example}  \label{Thm:ExInducedFunctional}
An easy way to construct an elicitable functional $T : \mathcal{F} \rightarrow \mathsf{A}$ is possible if an $\mathcal{F}$-integrable function $S : \mathsf{A} \times \mathsf{O} \rightarrow \mathbb{R}$ is given such that for any $F \in \mathcal{F}$ the function $x \mapsto \bar{S}(x,F)$ has a unique minimum. It is then possible to define $T(F) := \argmin_{x \in \mathsf{A}} \bar{S}(x,F)$ and $T$ is by definition an elicitable functional with strictly $\mathcal{F}$-consistent scoring function $S$.
\end{example}

The following lemma shows how strictly $\mathcal{F}$-consistent scoring functions can be manipulated, without destroying their consistency. In particular, we can scale them with positive constants, add certain functions and form convex combinations of scoring functions. Moreover, it is possible to force scoring functions to be positive under some circumstances. The following lemma is a summary of properties which are stated without proof in~\cite{FissZiegArxiv,FissZieg}.

\begin{lemma}  \label{Thm:TransformOfScoring}
Let $T: \mathcal{F} \rightarrow \mathsf{A}$ be a functional and $S, S_1 : \mathsf{A} \times \mathsf{O} \rightarrow \mathbb{R}$ be $\mathcal{F}$-consistent scoring functions for $T$, with $S$ being even strictly $\mathcal{F}$-consistent. Then the following hold true:
\begin{enumerate}[label=(\roman*)]
	\item For any $\lambda > 0$ and any $\mathcal{F}$-integrable function $h: \mathsf{O} \rightarrow \mathbb{R}$ the scoring function
	\begin{equation*}
	S_*(x,y) := \lambda S(x,y) + h(y)
	\end{equation*}
	is again strictly $\mathcal{F}$-consistent for $T$.
	\item The sum
	\begin{equation*}
	S_*(x,y) := S(x,y) + S_1(x,y)
	\end{equation*}
	is a strictly $\mathcal{F}$-consistent scoring function for $T$.
	\item If for all $y \in \mathsf{O}$, $\delta_y \in \mathcal{F}$ and the mapping $h: \mathsf{O} \rightarrow \mathbb{R}$, $y \mapsto S(T(\delta_y), y)$ is $\mathcal{F}$-integrable, then
	\begin{equation*}
	S_*(x,y) := S(x,y) - h(y)
	\end{equation*}
	is a nonnegative strictly $\mathcal{F}$-consistent scoring function for $T$.
	\item Let $(Z, \mathcal{Z}, \nu)$ be a measure space, where $\nu \neq 0$ is a $\sigma$-finite measure. Moreover, let $\lbrace S_z \mid z \in \mathcal{Z} \rbrace$ be a family of strictly $\mathcal{F}$-consistent scoring functions for $T$. If for all $x \in \mathsf{A}$ and $F \in \mathcal{F}$ the mapping $Z \times \mathsf{O} \rightarrow \mathbb{R}$, $(z,y) \mapsto S_z(x,y)$ is $\nu \otimes F$-integrable, then
	\begin{equation*}
	S_*(x,y) := \int_Z S_z(x,y) \, \mathrm{d}\nu(z)
	\end{equation*}
	is again a strictly $\mathcal{F}$-consistent scoring function for $T$.
\end{enumerate}
\end{lemma}

\begin{proof}
(i): For $x \in \mathsf{A}$ and $F \in \mathcal{F}$ we obtain
\begin{equation*}
\bar{S}_*(x,F) = \lambda \bar{S}(x,F) + \bar{h}(F) \geq \lambda \bar{S}(T(F),F) + \bar{h}(F) = \bar{S}_*(T(F),F)
\end{equation*}
and equality implies $x = T(F)$ due to the strict $\mathcal{F}$-consistency of $S$. \\
(ii): Again take $x \in \mathsf{A}$ and $F \in \mathcal{F}$ to get
\begin{equation*}
\bar{S}_*(x,F) = \bar{S}(x,F) + \bar{S}_1(x,F) \geq \bar{S}(T(F),F) + \bar{S}_1(T(F),F) = \bar{S}_*(T(F),F)
\end{equation*}
with equality implying
\begin{equation*}
\bar{S}(x,F) - \bar{S}(T(F),F) = \bar{S}_1(T(F),F) - \bar{S}_1(x,F).
\end{equation*}
The left-hand side of this equation is nonnegative while the right-hand side is nonpositive. Therefore, the terms on both sides vanish and the strict $\mathcal{F}$-consistency of $S$ implies $x = T(F)$. \\
(iii): The strict $\mathcal{F}$-consistency of $S_*$ follows from part (i). To show that $S_*$ is positive, we pick $x \in \mathsf{A}$ and $y \in \mathsf{O}$ and calculate
\begin{equation*}
S_*(x,y) = S(x,y) - S(T(\delta_y), y) = S(x,y) - \bar{S}(T(\delta_y), \delta_y) \geq S(x,y) - \bar{S}(x, \delta_y) = 0.
\end{equation*}
(iv): Let $x \in \mathsf{A}$ and $F \in \mathcal{F}$ be arbitrary.  The mapping $(z,y) \mapsto S_z(x,y)$ is $\nu \otimes F$-integrable and both $(Z, \mathcal{Z}, \nu)$ and $(\mathsf{O}, \mathcal{O}, F)$ are $\sigma$-finite measure spaces. Therefore, we may use Fubini's theorem (see for instance \cite[Ch. 14.2]{KlenkeProb}) to change the order of integration. This gives
\begin{align*}
\bar{S}_*(x,F) &= \int_\mathsf{O} \int_Z S_z(x,y) \,  \mathrm{d} \nu(z) \,\mathrm{d} F(y) = \int_Z \bar{S}_z(x,F) \,  \mathrm{d} \nu(z) \\
&\geq \int_Z \bar{S}_z(T(F),F) \,  \mathrm{d} \nu(z) = \int_\mathsf{O} \int_Z S_z(T(F),y) \,  \mathrm{d} \nu(z) \,\mathrm{d} F(y) = \bar{S}_*(T(F),F)
\end{align*}
with equality implying
\begin{equation} \label{Eqn:IntZFam}
\int_Z \bar{S}_z(x,F) - \bar{S}_z(T(F),F) \,  \mathrm{d} \nu(z) = 0.
\end{equation}
For any $z \in Z$ the integrand is positive due to the $\mathcal{F}$-consistency of $S_z$. Consequently, we have $\bar{S}_z(x,F) = \bar{S}_z(T(F),F)$ for $\nu$-almost every $z$ and this implies $x = T(F)$.
\end{proof}

\begin{remark}
Note that in part (ii) of Lemma~\ref{Thm:TransformOfScoring} only one of the scoring functions needs to be strictly $\mathcal{F}$-consistent. A similar fact holds true for part (iv), which becomes apparent in Equation~(\ref{Eqn:IntZFam}). More precisely, we only need to require that there is a subset $Z' \subset Z$ with $\nu (Z') >0$ such that the scoring functions $\lbrace S_z \mid z \in Z' \rbrace$ are strictly $\mathcal{F}$-consistent. All other members of the scoring function family need only be $\mathcal{F}$-consistent.
\end{remark}


Next we consider how conditions for the elicitability of $T: \mathcal{F} \rightarrow \mathsf{A}$ react if we modify the set of distribution functions $\mathcal{F}$ or the set $\mathsf{A}$. As the following lemma shows, strict consistency of a function $S$ is preserved if we look at subsets  $\mathsf{A}' \subseteq \mathsf{A}$ or $\mathcal{F}' \subseteq \mathcal{F}$.

\begin{lemma}[Fissler and Ziegel {\cite[Lemma 2.5]{FissZieg}}]
Let $T: \mathcal{F} \rightarrow \mathsf{A} \subset \mathbb{R}^k$ be a functional and $S: \mathsf{A} \times \mathsf{O} \rightarrow \mathbb{R}$ a strictly $\mathcal{F}$-consistent scoring function for $T$. Then the following assertions are true:
\begin{enumerate}[label=(\roman*)]
	\item Let $\mathcal{F}' \subseteq \mathcal{F}$ and set $T':= T_{\vert \mathcal{F}'}$. Then $S$ is also a strictly $\mathcal{F}'$-consistent scoring function for $T'$.
	\item Let $\mathsf{A}' \subseteq \mathsf{A}$ such that $T(\mathcal{F}) \subseteq \mathsf{A}'$ and set $S' := S_{\vert \mathsf{A}' \times \mathsf{O} }$. Then $S'$ is a strictly $\mathcal{F}$-consistent scoring function for $T$.
\end{enumerate}
\end{lemma}

\begin{proof}
(i): Let $x \in \mathsf{A}$ and $F \in \mathcal{F}'$. Then we also have $F \in \mathcal{F}$ and $T'(F) = T(F)$, which immediately gives the definition of strict $\mathcal{F}'$-consistency for $T'$. \\
(ii): Let $F \in \mathcal{F}$. For any $x \in \mathsf{A}'$ we have $\bar{S}(x,F) = \bar{S}'(x,F)$, which implies that $S'$ is a strictly $\mathcal{F}$-consistent scoring function for $T$.
\end{proof}

Next we present a criterion which can be used to check if a given scoring function is strictly $\mathcal{F}$-consistent for a functional $T$. In particular, the second part shows how it is possible to establish strict $\mathcal{F}$-consistency for a scoring function if its gradient is an oriented strict $\mathcal{F}$-identification function. This already hints at Section~\ref{Sec:Osband}, where we establish a similar relation between elicitability and identifiability. In part (ii) of the following lemma, we assume that the partial derivatives are dominated by an integrable function in order to permit interchanging differentiation and integration. Apart from this assumption, the result and its proof can be found in \cite[Lemma 2.9]{FissZiegArxiv}.

\begin{lemma}  \label{Thm:SufficientCondElicit}
Let $S: \mathsf{A} \times \mathsf{O} \rightarrow \mathbb{R}$ be a scoring function and $T: \mathcal{F} \rightarrow \mathsf{A} \subseteq \mathbb{R}^k$ a functional. Then the following assertions hold:
\begin{enumerate}[label=(\roman*)]
	\item $S$ is strictly $\mathcal{F}$-consistent for $T$ if and only if for all $F \in \mathcal{F}$, $v \in \mathbb{S}^{k-1}$ and $t = T(F)$ the function
	\begin{equation} \label{Eqn:PsiFuncionDef}
	\Psi_{F,v} : \lbrace s \in \mathbb{R} \mid t + s v \in \mathsf{A} \rbrace \rightarrow \mathbb{R}, \quad s \mapsto \bar{S}(t + sv, F)
	\end{equation}
	has a global unique minimum at $s= 0$.
	\item Let $S$ be continuously differentiable in $x$ and suppose there is an $\mathcal{F}$-integrable function $h$ such that the inequality $\sup_{x \in \mathsf{A}} \vert \partial_i S(x,\cdot) \vert \leq h$ holds $F$-almost surely for all $F \in \mathcal{F}$ and $i=1, \ldots, k$. Moreover, set $\mathcal{F}' := T^{-1}(\intr(\mathsf{A})) \subseteq \mathcal{F}$. If $\nabla S : \intr(A) \times \mathsf{O} \rightarrow \mathbb{R}^k$ is an oriented strict $\mathcal{F}'$-identification function for $T_{\vert \mathcal{F}'}$, then $S_{\vert \intr(A) \times \mathsf{O}}$ is a strictly $\mathcal{F}'$-consistent scoring function for $T_{\vert \mathcal{F}'}$.
	\item For all $F \in \mathcal{F}$, let $\bar{S}(\cdot, F)$ be continuously differentiable and define $\mathcal{F}'$ as in (ii). If for all $F \in \mathcal{F}'$ with $t = T(F)$ and all $v \in \mathbb{S}^{k-1}$, $s \in \mathbb{R}$ such that $t + sv \in \intr(A)$ it holds that
	\begin{align} \label{Eqn:CaseNablaS}
	v^\top \nabla \bar{S} (t + sv, F) \, \left\lbrace
	\begin{array}{lc}
	> 0 , &  \text{if } s > 0 \\
	= 0 , &  \text{if } s = 0 \\
	< 0 , &  \text{if } s < 0 
	\end{array} \right.
	\end{align}
	then $S_{\vert \intr(A) \times \mathsf{O}}$ is a strictly $\mathcal{F}'$-consistent scoring function for $T_{\vert \mathcal{F}'}$.
\end{enumerate}
\end{lemma}

\begin{proof}
(i): At first, we define $D_{F,v} := \lbrace s \in \mathbb{R} \mid T(F) + sv \in \mathsf{A} \rbrace$ and let $S$ be a strictly $\mathcal{F}$-consistent scoring function for $T$. For any $F \in \mathcal{F}$, $v \in \mathbb{S}^{k-1}$ and $r \in D_{F,v}$ we get $\Psi_{F,v} (r) = \bar{S}(T(F) + rv, F) \geq \bar{S}(T(F),F) = \Psi_{F,v} (0)$ so $\Psi_{F,v}$ has a minimum at $r = 0$. Moreover, $\Psi_{F,v} (r) = \Psi_{F,v} (0)$ implies $T(F) + rv = T(F)$ due to the strict $\mathcal{F}$-consistency of $S$, so the minimum is unique. To show the converse implication, take any $x \in \mathsf{A}$, $F \in \mathcal{F}$ and represent $x$ via  $T(F) + rv = x$ for some $r \in \mathbb{R}$ and $v \in \mathbb{S}^{k-1}$. Now $\bar{S}(x,F) = \Psi_{F,v} (r) \geq \Psi_{F,v} (0) = \bar{S}(T(F),F)$ and equality implies $x = T(F)$, since the minimum of $\Psi_{F,v}$ is unique. \\
(ii): To show this, we apply part (i). To this end, fix $F \in \mathcal{F}'$ and $v \in \mathbb{S}^{k-1}$ and define $D_{F,v} := \lbrace s \in \mathbb{R} \mid T(F) + sv \in \intr(\mathsf{A}) \rbrace$ as above. To simplify notation, write $S$ instead of $S_{\vert \intr(\mathsf{A}) \times \mathsf{O}}$. Applying Theorem~\ref{Thm:AppendixDiff} from the appendix, we see that the function $\Psi_{F,v}$ as defined in (\ref{Eqn:PsiFuncionDef}) is continuously differentiable on $D_{F,v}$ and for $s \in D_{F,v}$ satisfies
\begin{equation*}
\Psi_{F,v}' (s) = \frac{\mathrm{d}}{\mathrm{d}s} \bar{S}(T(F) + sv, F) = \int_\mathsf{O} \frac{\mathrm{d}}{\mathrm{d}s} S(t + sv, y) \, \mathrm{d} F(y) = v^\top \overline{\nabla S}(t + sv,F) .
\end{equation*}
By assumption, $\nabla S$ is an oriented strict $\mathcal{F}'$-identification function for $T_{\vert \mathcal{F}'}$, so $\Psi_{F,v}'(0) = 0$, $\Psi_{F,v}'(s) > 0$ for $s > 0$ and $\Psi_{F,v}'(s) < 0$ for $s<0$ according to Definition~\ref{Def:Orientation}. Therefore, $s=0$ is a unique minimum of $\Psi_{F,v}$ and by (i) $S_{\vert \intr(A) \times \mathsf{O}}$ is a strictly $\mathcal{F}'$-consistent scoring function for $T_{\vert \mathcal{F}'}$. \\
(iii): Again we apply (i). Let $F \in \mathcal{F}'$ and $v \in \mathbb{S}^{k-1}$ be arbitrary and define $D_{F,v}$ as well as $\Psi_{F,v}$ as in (ii). The function $\Psi_{F,v}$ is continuously differentiable by assumption and $\Psi_{F,v}'(s) = v^\top \nabla \bar{S}(T(F) + sv, F)$. Due to Equation~(\ref{Eqn:CaseNablaS}), $\Psi_{F,v}$ has a unique minimum at $s=0$ and an application of (i) finishes the proof.
\end{proof}

The next proposition gives an important necessary condition for the strict $\mathcal{F}$-consistency of a functional $T$. The result allows us to clarify the question if, similar to the mean, the variance functional is elicitable (see also Example~\ref{Thm:ExMeanIsElicitable}). The result is found in Gneiting~\cite[Thm. 6]{GneitingPoints}, Lambert~\cite{Lambert}, and Fissler and Ziegel~\cite[Proposition 2.14]{FissZiegArxiv} and the latter credit Kent Osband for stating this result in his doctoral thesis. It is also called the `convexity of level sets'-property. 

\begin{prop}[Convexity of level sets] \label{Thm:LevelSetCond}
Let $T : \mathcal{F} \rightarrow \mathsf{A} \subseteq \mathbb{R}^k$ be an elicitable functional. Then for all $F_0, F_1 \in \mathcal{F}$ such that $t := T(F_0) = T(F_1)$ and all $\lambda \in (0,1)$ such that $F_\lambda := (1-\lambda) F_0 + \lambda F_1 \in \mathcal{F}$ we have $T(F_\lambda) = t$.
\end{prop}
\begin{proof}
Since $T$ is elicitable, we select a strictly $\mathcal{F}$-consistent scoring function and denote it by $S$. For any $x \in \mathsf{A}$ the linearity of the integral gives
\begin{align*}
\bar{S}(x, F_\lambda) &= (1-\lambda) \bar{S}(x,F_0) + \lambda \bar{S}(x, F_1) \\
&\geq (1-\lambda) \bar{S}(T(F_0),F_0) + \lambda \bar{S}(T(F_1),F_1) = \bar{S}(t, F_\lambda) \geq \bar{S}(T(F_\lambda), F_\lambda)
\end{align*}
and setting $x = T(F_\lambda)$, gives $t = T(F_\lambda)$ due to the strict $\mathcal{F}$-consistency of $S$.
\end{proof}

\begin{remark}  \label{Rem:LevelSetSufficient}
In the one-dimensional case, the necessary condition of convex level sets is also a sufficient condition under certain assumptions. For finite probability spaces this is shown by Lambert~\cite[Thm. 5]{Lambert}, who assumes that $T$ is continuous and on no open set in $\mathcal{F}$ constant. The result is extended to arbitrary probability spaces by Steinwart et al.~\cite{SteinwartPasinWilliam}. They additionally assume that all probability measures $(P_F)_{F \in \mathcal{F}}$ admit a bounded density with respect to some finite measure $\mu$ and equip this space with the total variation norm. For $k> 1$ sufficient conditions are still unknown, since a counterexample of Frongillo and Kash~\cite[Example 1]{FrongilloKash} shows that convex level sets do not suffice. That some additional assumptions on the functional as used in \cite{Lambert} and \cite{SteinwartPasinWilliam} are indeed necessary is demonstrated by Heinrich~\cite{HeinrichMode}, who shows that the mode functional fails to be elicitable, although having convex level sets.
\end{remark}

In the next example we present our own calculations to check the level set condition for the variance. 

\begin{example}[Non-elicitability of the variance]  \label{Thm:ExVariNotElicitable}
For $\mathsf{A}:= [0,\infty)$, $\mathsf{O}:=\mathbb{R}$ and some class of distribution functions $\mathcal{F}$ let the variance functional be defined in the same manner as in \cite[Corollary 2.16]{FissZiegArxiv}, by setting
\begin{equation}  \label{Eqn:VariFunctionalDef}
T:\mathcal{F} \rightarrow \mathsf{A}, \quad F \mapsto T(F) := \int_{\mathsf{O}} y^2 \, \mathrm{d} F(y) - \left( \int_{\mathsf{O}} y \, \mathrm{d} F(y) \right)^2.
\end{equation}
Naturally, we restrict $\mathcal{F}$ to all distributions having finite second moments. In order to establish the non-elicitability of $T$ and in view of Proposition~\ref{Thm:RevelationPrinciple}, it suffices to find two distributions $F_1$ and $F_2$ with $T(F_1) = T(F_2) = t$ and $\lambda \in (0,1)$ such that $T(F_\lambda) \neq t$. To this end, fix a measure space $(\Omega, \mathscr{A}, \mathbb{P})$ and let $Y_1 =^d F_1$ and $Y_2 =^d F_2$ be random variables on it such that $\Var (Y_1) = \Var (Y_2)$. Moreover, let $B: \Omega \rightarrow \lbrace 0,1 \rbrace$ be a random variable having a Bernoulli distribution with parameter $\lambda \in (0,1)$ under $\mathbb{P}$. Assuming that $B$ is independent of $Y_1,  Y_2$, we calculate for any $x \in \mathsf{O}$ the distribution function of $Z_\lambda := BY_1 + (1-B)Y_2 $, given by
\begin{align*}
\mathbb{P}( BY_1 + (1-B)Y_2 \leq x) &= \mathbb{P}(BY_1 + (1-B)Y_2 \leq x  , \, B = 1) \\
&\quad+ \mathbb{P}( BY_1 + (1-B)Y_2 \leq x , \, B = 0) \\
&=\mathbb{P}(Y_1 \leq x) \mathbb{P}(B=1) + \mathbb{P}(Y_2 \leq x) \mathbb{P}(B= 0) \\
&=\lambda F_1(x) + (1- \lambda) F_2(x).
\end{align*}
Therefore, $Z_\lambda$ has distribution function $F_\lambda := \lambda F_1 + (1-\lambda) F_2$ under $\mathbb{P}$. Hence, in order to study the elicitability of (\ref{Eqn:VariFunctionalDef}), we have to check for which random variables $Y_1$, $Y_2$, and $B$ it holds that $ \Var (Y_1) = \Var (Y_2) = t$ implies $\Var (Z_\lambda) = t$. \\
We begin by calculating the expectation and variance of $Z_\lambda$ in terms of expectation and variance of $Y_1$, $Y_2$, and $B$. We obtain the formulas
\begin{align*}
\mathbb{E} Z_\lambda &= \lambda \mathbb{E}Y_1 + (1- \lambda) \mathbb{E}Y_2 \\
\Var (Z_\lambda) &= \mathbb{E}B^2Y_1^2 + 2 \mathbb{E}B(1-B) Y_1 Y_2 + \mathbb{E}(1-B)^2 Y_2^2 - (\mathbb{E}Z_\lambda )^2 \\
&= \lambda \mathbb{E}Y_1^2 + (1-\lambda) \mathbb{E}Y_2^2 - (\lambda \mathbb{E}Y_1 + (1-\lambda) \mathbb{E}Y_2 )^2 \\
&= \lambda \Var (Y_1) + (1-\lambda) \Var (Y_2) + \lambda (1-\lambda) (\mathbb{E}Y_1 - \mathbb{E}Y_2 )^2
\end{align*}
and realize that the variance cannot be elicitable if the expectations of $Y_1$ and $Y_2$ do not coincide. In the situation of differing expectations we always have that $\Var(Z_\lambda)$ and $\Var (Y_1)$ differ by $\lambda (1-\lambda) (\mathbb{E}Y_1 - \mathbb{E}Y_2 )^2 > 0$. Consequently, if a class $\mathcal{F}$ contains distributions $F_1, F_2$ such that they have equal variance but different expectation and $\lambda F_1 + (1-\lambda)F_2 \in \mathcal{F}$ for a $\lambda \in (0,1)$, the variance functional (\ref{Eqn:VariFunctionalDef}) is not elicitable relative to $\mathcal{F}$
\end{example}

\begin{remark}  \label{Rem:FnotElicitable}
In light of the previous example we will from now on say that the variance is not elicitable. With this statement we \textit{do not} want to say that it is impossible to find a class $\mathcal{F}$ such that the functional (\ref{Eqn:VariFunctionalDef}) is elicitable. It rather means that for classes which are `reasonably' rich it is not possible to elicit the variance. Which class $\mathcal{F}$ is rich enough to say that some functional $T$ is not elicitable, when it, strictly speaking, only fails to be elicitable if defined on $\mathcal{F}$, is of course subjective. Nevertheless, if we keep this fact in mind, we can safely speak of a functional as being not elicitable in the following.
\end{remark}

\begin{example}  \label{Thm:ExVariElicitableForCenter}
To clarify the results of Example~\ref{Thm:ExVariNotElicitable} and the previous remark we look again at the variance functional defined in (\ref{Eqn:VariFunctionalDef}). If we change $\mathcal{F}$ to be the class of all \textit{centered} distribution functions with existing second moments, the variance is identical to the second moment for all $F \in \mathcal{F}$. Since Theorem~\ref{Thm:GenExpectElicit} shows that all moments are elicitable, we see that the variance functional is elicitable relative to this class of distribution functions.
\end{example}

The previous examples emphasize that the choice of the class $\mathcal{F}$ is significant and elicitability cannot be studied without making this choice. Another important aspect is that if a functional is not elicitable for some class $\mathcal{F}$, it nevertheless can be elicitable in combination with some other functionals. More precisely, it can be part of a vector of functionals which is elicitable. This fact is mathematically explained by the revelation principle, which is stated and proved in~\cite[Prop. 2.13]{FissZiegArxiv} and~\cite[Thm. 4]{GneitingPoints}. Both sources credit Kent Osband.

\begin{prop}[Revelation principle] \label{Thm:RevelationPrinciple}
Choose $\mathsf{A}, \mathsf{A}' \in \mathbb{R}^k$ and let $g: \mathsf{A} \rightarrow \mathsf{A}'$ be a bijection with inverse $g^{-1}$. Let $T: \mathcal{F} \rightarrow \mathsf{A}$ and $T_g : \mathcal{F} \rightarrow \mathsf{A}', F\mapsto T_g(F) := g(T(F))$ be functionals. Then the following is true:
\begin{enumerate}[label=(\roman*)]
	\item $T$ is identifiable if and only if $T_g$ is identifiable. A function $V: \mathsf{A} \times \mathsf{O} \rightarrow \mathbb{R}^k$ is a strict $\mathcal{F}$-identification function for $T$ if and only if $V_g : \mathsf{A}' \times \mathsf{O} \rightarrow \mathbb{R}^k$, $(x,y) \mapsto V_g(x,y) := V( g^{-1} (x), y)$ is a strict $\mathcal{F}$-identification function for $T_g$.
	\item $T$ is elicitable if and only if $T_g$ is elicitable. A function $S: \mathsf{A} \times \mathsf{O} \rightarrow \mathbb{R}$ is a strictly $\mathcal{F}$-consistent scoring function for $T$ if and only if $S_g : \mathsf{A}' \times \mathsf{O} \rightarrow \mathbb{R}$, $(x,y) \mapsto S_g(x,y) := S( g^{-1}(x), y)$ is a strictly $\mathcal{F}$-consistent scoring function for $T_g$.
\end{enumerate}
\end{prop}

\begin{proof}
(i): It suffices to show one implication since $g$ is a bijection. Moreover, by the definition of identifiability, we only need to show that $V_g$ is a strict $\mathcal{F}$-identification function if $V$ is. So let $F \in \mathcal{F}$, $x' \in \mathsf{A}'$ and set $x = g^{-1}(x')$. Then we have that $\bar{V}_g(x',F) = 0 \Leftrightarrow \bar{V}(g^{-1}(x'),F) = 0 \Leftrightarrow \bar{V}(x,F) = 0$ and the identification property yields that this is equivalent to $x= T(F) \Leftrightarrow x' = T_g(F)$. \\
(ii): As in (i) it suffices to show that $S_g$ is a strictly $\mathcal{F}$-consistent scoring function if $S$ is. So let $F \in \mathcal{F}$, $x' \in \mathsf{A}'$ and set $x = g^{-1}(x')$. This yields the inequality
\begin{align*}
\bar{S}_g(x',F) = \bar{S}(g^{-1}(x'),F) = \bar{S}(x,F) &\geq  \bar{S}(T(F),F) \\
&= \bar{S}( g^{-1} ( g (T(F) ) ),F) = \bar{S}_g(T_g(F),F),
\end{align*}
and equality implies $x = T(F)$, which gives $x' = g(x) = g(T(F)) = T_g(F)$. 
\end{proof}

\begin{remark}  \label{Rem:RevelationPrinciple}
The revelation principle states that elicitability is preserved if a functional is transformed using a bijective mapping. For example, we obtain that if $T$ is elicitable, so is $-T$. Moreover, if we consider the one-dimensional case,  $aT + c$ (with $a \neq 0$) is elicitable and if $T$ is nonnegative, also $\sqrt{T}$ or $\vert T \vert$ are elicitable. The same holds for identifiability and consistent scoring functions as well as identification functions can be calculated using Proposition~\ref{Thm:RevelationPrinciple}.
\end{remark}

In the following example we use the revelation principle to show that mean and variance are jointly elicitable. We apply the same technique as~\cite[Cor. 2.16]{FissZiegArxiv}.

\begin{example}[Joint elicitability of mean and variance]  \label{Thm:ExMAndVElicitable}
It is well known, that the representation $\Var (Y) = \mathbb{E}Y^2 - (\mathbb{E}Y)^2$ holds. A bijection between the first two moments and the pair (expectation, variance) is thus immediately seen. To be more precise, let $T_1$ be the mean functional defined in (\ref{Eqn:MeanFunctionalDef}) and $T_2$ the second moment functional defined by setting $k=1$, $h(x) = x^2$ and $q(x) = 1$ in Theorem~\ref{Thm:GenExpectElicit}. Denote the variance functional defined in (\ref{Eqn:VariFunctionalDef}) via $T_{\Var}$.  Let $\mathcal{F}$ be a class of distribution functions on $\mathsf{O} := \mathbb{R}$ such that second moments exist for all members. Define the sets $\mathsf{A} := \lbrace (x_1, x_2) \mid x_2 \geq x_1^2 \rbrace \subset \mathbb{R}^2$ and $\mathsf{A}' := \mathbb{R} \times [0, \infty) \subset \mathbb{R}^2$ with a bijection $g: \mathsf{A} \rightarrow \mathsf{A}'$ given by $(x_1, x_2) \mapsto (x_1, x_2 - x_1^2)$. The inverse of $g$ is given by $g^{-1} : \mathsf{A}' \rightarrow  \mathsf{A}$, $(x_1, x_2) \mapsto (x_1, x_2 + x_1^2)$. The functional we are interested in, namely $(T_1, T_{\Var})^\top$, can now be written as $g((T_1, T_2))$ and hence it is elicitable due to the revelation principle \ref{Thm:RevelationPrinciple}. Why $(T_1, T_2)^\top$ is elicitable is rigorously proved in Lemma~\ref{Thm:AllCompElicitable} (i). We calculate a strictly $\mathcal{F}$-consistent scoring function for $(T_1, T_{\Var})^\top$ by taking a strictly $\mathcal{F}$-consistent scoring function $S$ for $(T_1, T_2)^\top$ similar to Example~\ref{Thm:ExMeanIsElicitable}, that is
\begin{align*}
S(x_1, x_2, y) := -f_1 (x_1) - f_1'(x_1) (y - x_1) - f_2 (x_2) - f_2'(x_2) (y^2 - x_2)
\end{align*}
for two differentiable strictly convex functions $f_1, f_2$. The strict consistency of $S$ for $(T_1, T_2)^\top$ follows from Lemma~\ref{Thm:AllCompElicitable}. If we now apply the revelation principle once again, we obtain that
\begin{align*}
S_g(x_1, x_2, y) = &\,\,S( g_1^{-1} (x_1, x_2), g_2^{-1} (x_1, x_2), y) \\
= &- f_1(x_1) - f_1'(x_1) (y - x_1) - f_2(x_2 + x_1^2) \\
&- f_2'(x_2 + x_1^2) (y^2- (x_2 + x_1^2) )
\end{align*}
is a strictly $\mathcal{F}$-consistent scoring function for $(T_1, T_{\Var})^\top$. We could also have employed the second part of Proposition~\ref{Thm:RevelationPrinciple} to show that the functional $(T_1, T_{\Var})^\top$ is identifiable and calculate a strict $\mathcal{F}$-identification function.
\end{example}

\begin{remark}
One question which comes to mind regarding the revelation principle \ref{Thm:RevelationPrinciple} and Example~\ref{Thm:ExMAndVElicitable} is the following: Can we always find a vector of elicitable functionals such that we can bijectively map it onto a vector of functionals which contains (at least) one component not being elicitable? Or more generally put, is every non-elicitable functional part of an elicitable vector of functionals? The fact that skewness and kurtosis of a random variable are not elicitable, but can be part of an elicitable vector (together with the necessary moments), is an encouraging result in this direction. It can be proved using the revelation principle in the same manner as Example~\ref{Thm:ExMAndVElicitable}. Nevertheless, the examples variance, skewness and kurtosis are rather simple functionals, and a similar result fails to hold for more complex ones. The most prominent case is the functional pair Value at Risk and Expected Shortfall, where the latter is not elicitable. Details are presented in Subsections~\ref{Sec:ValueAtRisk} and~\ref{Sec:ExpectedShortfall}.
\end{remark}

This section concludes with a result showing that vectors of functionals are elicitable if all components are elicitable. That the converse implication is not true is shown by Examples~\ref{Thm:ExVariNotElicitable} and \ref{Thm:ExMAndVElicitable}. The first part of the result is proved in Fissler and Ziegel~\cite[Lemma 2.6]{FissZieg} and we prove two similar statements for identifiability.

\begin{lemma}  \label{Thm:AllCompElicitable}
Let $n \geq 1$ and choose functionals $T_i : \mathcal{F} \rightarrow \mathsf{A}_i \subseteq \mathbb{R}^{k_i}$ for $i= 1, \ldots, n$ and $k_1, \ldots, k_n \in \mathbb{N}$. For $k := \sum_{i=1}^{n} k_i$ and $\mathsf{A} := \mathsf{A}_1 \times \ldots \times \mathsf{A}_n \subseteq \mathbb{R}^k$ we define a $k$-dimensional functional $T: \mathcal{F} \rightarrow \mathsf{A}$ via $T(F) := (T_1(F), \ldots, T_n(F) )^\top$. Then the following assertions hold:
\begin{enumerate}[label=(\roman*)]
	\item $T$ is elicitable if all $(T_i)_{i=1, \ldots, n}$ are elicitable.
	\item $T$ is identifiable if all $(T_i)_{i=1, \ldots, n}$ are identifiable.
	\item $T$ has an oriented strict $\mathcal{F}$-identification function if all $(T_i)_{i=1, \ldots, n}$ have such a function.
\end{enumerate}
\end{lemma}
\begin{proof}
(i): For $i \in I_n := \lbrace 1, \ldots, n \rbrace$ let $S_i : \mathsf{A}_i \times \mathsf{O} \rightarrow \mathbb{R}$ be a strictly $\mathcal{F}$-consistent scoring function for $T_i$. Define the scoring function $S$ for $T$ via
\begin{equation}  \label{Eqn:JointScoringFunction}
S: \mathsf{A} \times \mathsf{O} \rightarrow \mathbb{R}, \quad (x,y) \mapsto S(x_1, \ldots, x_n,y) := \sum_{i=1}^{n} S_i (x_i,y). 
\end{equation}
Now for any $F \in \mathcal{F}$, $x \in \mathsf{A}$ we obtain
\begin{equation*}
\bar{S}(x_1, \ldots, x_n,F) - \bar{S}(T_1(F), \ldots, T_n(F), F) = \sum_{i=1}^{n} \bar{S}_i(x_i,F) - \bar{S}_i (T_i(F), F) \geq 0
\end{equation*}
and every summand is positive, so if equality holds, we have $x_i = T_i(F)$ for all $i \in I_n$. This shows that $S$ is a strictly $\mathcal{F}$-consistent scoring function for $T$. \\
(ii): As above, for $i \in I_n$ let $V_i : \mathsf{A}_i \times \mathsf{O} \rightarrow \mathbb{R}^{k_i}$ be a strict $\mathcal{F}$-identification function for $T_i$. Concatenate all $V_i$ to define the identification function
\begin{equation}  \label{Eqn:JointIdentificationFunction} 
V : \mathsf{A} \times \mathsf{O} \rightarrow \mathbb{R}^k, \quad  (x,y) \mapsto V(x_1, \ldots, x_n,y) := (V_1(x_1,y), \ldots, V_n(x_n,y))^\top .
\end{equation}
For any $F \in \mathcal{F}$, $x \in \mathsf{A}$ we have $\bar{V}(x_1, \ldots, x_n, F) = 0$ if and only if $\bar{V}_i(x_i,F) = 0$ for all $i \in I_n$    and this is equivalent to $x_i = T_i(F)$ for all $i \in I_n$. Hence, $V$ is a strict $\mathcal{F}$-identification function for $T$. \\
(iii): For $i \in I_n$ let $V_i : \mathsf{A}_i \times \mathsf{O} \rightarrow \mathbb{R}^{k_i}$ be an oriented strict $\mathcal{F}$-identification function for $T_i$ and define $V$ as in (\ref{Eqn:JointIdentificationFunction}). For $F \in \mathcal{F}$, $v \in \mathbb{S}^{k-1}$ and $s \in \mathbb{R}$ such that $T(F) + sv \in \mathsf{A}$ holds we calculate
\begin{equation*}
v^\top \bar{V}(T(F) + sv, F) = \sum_{i=1}^{n} v_i^\top \bar{V}_i (T_i(F) + s v_i , F) .
\end{equation*}
If $s > 0$, it follows for any $i \in I_n$ that $\bar{V_i}(T_i(F) + s v_i , F) > 0 \, \Leftrightarrow \, v_i > 0$ since $V_i$ is oriented. Similarly, $s < 0$ implies that $\bar{V_i}(T_i(F) + s v_i , F) > 0 \, \Leftrightarrow \, v_i < 0$ holds for any $i \in I_n$, so $V$ is an oriented $\mathcal{F}$-identification function for $T$.
\end{proof}

\begin{remark}
The previous lemma states that functionals which have elicitable (identifiable) components are elicitable (identifiable). Moreover, its proof yields strictly consistent scoring function in (\ref{Eqn:JointScoringFunction}) and (oriented) strict identification functions in (\ref{Eqn:JointIdentificationFunction}).
\end{remark}

\section{Osband's principle} 
\label{Sec:Osband}

In this section, two versions of Osband's principle which provide a connection between scoring and identification functions for a functional $T$ are proved. The motivation for this is as follows. If $S$ is a strictly $\mathcal{F}$-consistent scoring function, the mapping $x \mapsto \bar{S}(x, F)$  attains its minimum at $T(F)$. Simultaneously, the mapping $x \mapsto \bar{V}(x,F)$ vanishes at $T(F)$ for every strict $\mathcal{F}$-identification function $V$. Consequently, we are tempted to think of $\bar{V}(\cdot,F)$ as the derivative of $\bar{S}(\cdot,F)$ with respect to $x$, which necessarily vanishes for a local minimum. Indeed, we have 
\begin{equation}  \label{Eqn:OsbandsPrinciple}
 \nabla \bar{S}(x,F) = h(x) \bar{V}(x,F)
\end{equation}
for a matrix-valued function $h$ and the precise result is stated in Theorem~\ref{Thm:OsbandPrinciple1}. If we further assume that integration and differentiation can be interchanged, we might even think that $hV$ is the derivative of $S$ with respect to $x$. A statement of this type is presented in Theorem~\ref{Thm:OsbandPrinciple2}. Both results are important tools to investigate the structure of strictly consistent scoring functions for a functional $T$. In many cases, there is a straightforward way to define a strict identification function for $T$. Under certain regularity assumptions, Equation~(\ref{Eqn:OsbandsPrinciple}) can then be used to calculate a strictly consistent scoring function for $T$ and such an application can be found in Section~\ref{Sec:FunctionalsComponents}. For both versions of Osband's principle, we follow the proofs as presented in Fissler and Ziegel~\cite{FissZiegArxiv} and add more details. A version of Osband's principle in the one-dimensional case can also be found in Steinwart et al.~\cite{SteinwartPasinWilliam}. 

\subsection{First version of Osband's principle} 

This subsection starts with a version of Osband's principle on the level of expectations as stated in (\ref{Eqn:OsbandsPrinciple}). Following \cite{FissZieg}, we state some assumptions concerning the identification function $V$ and the functions $\bar{V}(\cdot, F) : \mathsf{A} \rightarrow \mathbb{R}^k$, $x \mapsto \bar{V}(x,F)$, which are needed below. We denote the convex hull of a set $M$ by $\conv (M)$.

\phantomsection
\begin{assumption2}[V1] \label{As:V1}
For every $x \in \intr(A) \subseteq \mathbb{R}^k$ there are $F_1, \ldots, F_{k+1} \in \mathcal{F}$ such that
\begin{equation*}
0 \in \intr ( \conv (\bar{V}(x,F_1), \ldots, \bar{V}(x, F_{k+1}) ) ) .
\end{equation*}
\end{assumption2}

\begin{assumption2}[V2] \label{As:V2}
For every $F \in \mathcal{F}$ the mapping $\bar{V}( \cdot, F)$ is continuous.
\end{assumption2}

\begin{assumption2}[V3] \label{As:V3}
For every $F \in \mathcal{F}$ the mapping $\bar{V}( \cdot, F)$ is continuously  differentiable.
\end{assumption2}

\begin{remark}
Assumption (\nameref{As:V1}) is a richness assumption which guarantees that the functional attains enough different values and is not `trapped' in a linear subspace of $\mathbb{R}^k$. The Assumptions~(\nameref{As:V2}) and (\nameref{As:V3}) provide enough regularity for the function $h$ appearing in (\ref{Eqn:OsbandsPrinciple}). If for all $y \in \mathsf{O}$ the mapping $x \mapsto V(x,y)$ is continuously differentiable, Assumption~(\nameref{As:V3}) holds, as long as the partial derivatives are dominated by an integrable function. For an illustration see Theorem~\ref{Thm:AppendixDiff} and Example~\ref{Thm:AppendixExCounterDiff} from the appendix.
\end{remark}

For the scoring function $S$ and the functions $\bar{S}(\cdot, F) : \mathsf{A} \rightarrow \mathbb{R}$,  $x \mapsto \bar{S}(x,F)$ we also require the same regularity assumptions as in \cite{FissZieg}:

\phantomsection
\begin{assumption2}[S1] \label{As:S1}
For every $F \in \mathcal{F}$ the mapping $\bar{S}(\cdot, F)$ is continuously differentiable.
\end{assumption2}

\begin{assumption2}[S2] \label{As:S2}
For every $F\in \mathcal{F}$ the mapping $\bar{S}(\cdot, F)$ is continuously differentiable and the gradient is locally Lipschitz continuous. 
Furthermore, $\bar{S}(\cdot, F)$ is twice continuously differentiable at $t=T(F) \in \intr(\mathsf{A})$.
\end{assumption2}

Using these assumptions, we prove the first version of Osband's principle, which is concerned with the connection between $\bar{S}(\cdot,F)$ and $\bar{V}(\cdot,F)$.

\begin{theorem}[Fissler and Ziegel {\cite[Thm. 3.2]{FissZieg}}] \label{Thm:OsbandPrinciple1}
Let $\mathcal{F}$ be a convex class of distribution functions on $\mathsf{O} \subseteq \mathbb{R}^d$ and $T: \mathcal{F} \rightarrow \mathsf{A} \subseteq \mathbb{R}^k$ a surjective elicitable and identifiable functional with strict $\mathcal{F}$-identification function $V$ and strictly $\mathcal{F}$-consistent scoring function $S$. Assuming (\nameref{As:V1}) and (\nameref{As:S1}), there exists a function $h: \intr (\mathsf{A}) \rightarrow \mathbb{R}^{k \times k}$ such that
\begin{equation*}
\partial_l \bar{S}(x,F) = \sum_{m=1}^{k} h_{lm} (x) \bar{V}_m (x, F)
\end{equation*}
holds  for all $x \in \intr (\mathsf{A})$ and $F \in \mathcal{F}$. If additionally (\nameref{As:V2}) holds, then $h$ is continuous and if (\nameref{As:V3}) and (\nameref{As:S2}) also hold, then $h$ is locally Lipschitz continuous.
\end{theorem}

\begin{proof}
Following the proof of \cite{FissZieg}, we need to show two things: Firstly, we prove the existence of a function $h$ which only depends on $x \in \intr (\mathsf{A})$. Secondly, we establish continuity of $h$ under Assumption~(\nameref{As:V2}) and local Lipschitz continuity under the Assumptions~(\nameref{As:V3}) and (\nameref{As:S2}). \\
To show the existence of $h$, fix $l \in \lbrace 1, \ldots, k \rbrace$ and $x \in \intr (\mathsf{A})$. Using the identifiability of $T$, the strict $\mathcal{F}$-consistency of $S$ and Assumption~(\nameref{As:S1}), we have for all $F \in \mathcal{F}$ the implications
\begin{equation} \label{Eqn:FocImpVS}
\bar{V}(x,F) = 0 \, \Rightarrow \, x = T(F) \, \Rightarrow \, \nabla \bar{S}(x, F) = 0 ,
\end{equation}
since $T(F)$ is the local and global minimum of the continuously differentiable function $\bar{S}(\cdot,F)$. Invoking (\nameref{As:V1}), there exist $F_1, \ldots, F_{k+1} \in \mathcal{F}$ such that 
\begin{equation} \label{Eqn:ZeroInConvHull}
0 \in \intr ( \conv ( \bar{V}(x,F_1), \ldots, \bar{V}(x, F_{k+1}) ) )
\end{equation}
holds and thus the matrix $( \bar{V}(x,F_1), \ldots, \bar{V}(x, F_{k+1}) ) \in \mathbb{R}^{k \times k+1}$ has maximal rank $k$. If it had rank $j <k$, then its columns would span a linear subspace of $\mathbb{R}^k$ having dimension $j$, a contradiction to (\ref{Eqn:ZeroInConvHull}). Now let $G \in \mathcal{F}$ be arbitrary and define the matrix
\begin{equation*}
\mathbb{V}_G := (\bar{V}(x,G), \bar{V}(x,F_1), \ldots, \bar{V}(x, F_{k+1}) ) \in \mathbb{R}^{k \times k+2},
\end{equation*}
which also has full rank $k$. Moreover, if we consider matrices as linear mappings, the matrix $\mathbb{V}_G$ has the same kernel as the matrix
\begin{equation*}
\mathbb{W}_G := \begin{pmatrix}
\partial_l \bar{S}(x,G) & \partial_l \bar{S}(x,F_1) & \ldots & \partial_l \bar{S}(x,F_{k+1}) \\ 
  & \mathbb{V}_G &  & 
\end{pmatrix}
\in \mathbb{R}^{k+1 \times k+2}.
\end{equation*}
To show this, observe that by the definition of $\mathbb{W}_G$ and ignoring its first row, we immediately get $\ker (\mathbb{W}_G) \subseteq \ker (\mathbb{V}_G)$. Moreover, the kernels of both matrices have to be true subspaces of $\mathbb{R}^{k+2}$ since they map into the spaces $\mathbb{R}^k$ and $\mathbb{R}^{k+1}$. Therefore, in order to show the other inclusion, take $\theta \in \ker ( \mathbb{V}_G) \backslash \lbrace 0 \rbrace$ which satisfies $\theta_i \geq 0$ for all $i = 1, \ldots, k+2$. If we define $\sigma := \sum_{i=1}^{k+2} \theta_i$, convexity of $\mathcal{F}$ gives
\begin{equation*}
\partial_l \bar{S}(x,G) \theta_1 + \sum_{i=1}^{k+1} \partial_l \bar{S}(x, F_i) \theta_{i+1} = \sigma \partial_l \bar{S} \left( x, \frac{\theta_1}{\sigma} G + \sum_{i=1}^{k+1} \frac{\theta_{i+1}}{\sigma} F_i \right) = 0 ,
\end{equation*}
where the last equality uses the Implication~(\ref{Eqn:FocImpVS}) and the fact that $\theta \in \ker ( \mathbb{V}_G)$, which implies
\begin{equation*}
 \sigma \bar{V}\left( x,  \frac{\theta_1}{\sigma} G + \sum_{i=1}^{k+1} \frac{\theta_{i+1}}{\sigma} F_i \right)= \mathbb{V}_G \theta = 0.
\end{equation*}
Consequently, we obtain that all elements of $\ker (\mathbb{V}_G)$  with nonnegative components are contained in $\ker (\mathbb{W}_G)$. Now let $\theta \in \ker (\mathbb{V}_G)$ be arbitrary. Due to (\ref{Eqn:ZeroInConvHull}), there exists a linear combination of $\bar{V}(x,G), \bar{V}(x,F_1), \ldots, \bar{V}(x, F_{k+1})$ which is zero and the coefficients of this linear combination can be chosen strictly positive (for a constructive proof, see Lemma~\ref{Thm:AppendixConvexHull}). This gives a $\theta^* \in \ker (\mathbb{V}_G)$ having strictly positive components. Scaling $\theta^*$ with a real-valued number $r> 0$ such that $\theta + r\theta^*$ has nonnegative components gives $\theta + r \theta^* \in \ker (\mathbb{W}_G)$ by the above arguments. Hence, $\mathbb{W}_G \theta = \mathbb{W}_G( \theta + r \theta^*) = 0$ and it follows that $\theta \in \ker (\mathbb{W}_G)$ and finally $\ker (\mathbb{V}_G) = \ker (\mathbb{W}_G)$. \\
The rank-nullity theorem for linear mappings (see Liesen and Mehrmann~\cite[Thm. 10.9]{LiesenMehrmann}) gives
\begin{align*}
k+2 = \dim (\mathbb{R}^{k+2}) &= \dim ( \imge (\mathbb{W}_G )) + \dim ( \ker(\mathbb{W}_G )) \\
k+2 = \dim (\mathbb{R}^{k+2}) &=  \dim ( \imge (\mathbb{V}_G )) + \dim ( \ker(\mathbb{V}_G )) 
\end{align*}
and since $\mathbb{V}_G$ has maximal rank, $ \dim ( \ker(\mathbb{V}_G )) = 2$ must hold. Consequently, $\dim ( \imge (\mathbb{W}_G) ) = k$ showing that $\mathbb{W}_G$ has rank $k$. This allows us to represent the first row of $\mathbb{W}_G$ as a linear combination of all other rows. Call the unique coefficients of this linear combination $h_{l1}(x) , \ldots, h_{lk}(x)$. They do not depend on $G$, since they must also hold for the columns $2, \ldots, k+2$, which are fixed for every choice of $G \in \mathcal{F}$. We obtain
\begin{equation*}
\partial_l \bar{S}(x,G) = \sum_{m=1}^{k} h_{lm}(x) \bar{V}_m (x,G)
\end{equation*}
and repeat the previous steps of the proof for all $l = 1, \ldots, k$, where each $l$ gives a vector $(h_{l1}(x) , \ldots, h_{lk}(x)) \in \mathbb{R}^{1 \times k}$. Concatenating all vectors to a matrix $h(x) \in \mathbb{R}^{k \times k}$ we finally have
\begin{equation} \label{Eqn:OsbandMatrix}
\nabla \bar{S}(x,G) \, = \, h(x) \bar{V}(x,G)
\end{equation}
and the first part of the theorem is proved. \\
To prove that more regularity can be imposed on $h$, let Assumption~(\nameref{As:V2}) hold and fix $x \in \intr (\mathsf{A})$. Moreover, choose $F_1, \ldots, F_k \in \mathcal{F}$ such that $\bar{V}(x,F_1), \ldots, \bar{V}(x, F_k)$ are linearly independent, which is possible due to Assumption~(\nameref{As:V1}). Define the matrix-valued mapping 
\begin{equation} \label{Eqn:MappingMathbbV}
\mathbb{V}: \intr (\mathsf{A}) \rightarrow \mathbb{R}^{k \times k} , \quad z \mapsto \mathbb{V}(z) := ( \bar{V}(z,F_1), \ldots, \bar{V}(z, F_k) ) ,
\end{equation}
which is continuous due to the continuity of $\bar{V}(\cdot, F)$. The matrix $\mathbb{V}(x)$ is invertible and due to Lemma~\ref{Thm:AppendixLinearInd} there is an open neighborhood $U_x$ of $x$ such that $\mathbb{V}(z)$ is invertible for all $z \in U_x$. Consequently, Identity~(\ref{Eqn:OsbandMatrix}) leads to a representation of $h$ for all $z \in U_x$ given by
\begin{equation} \label{Eqn:RepOfh}
h(z) = ( \nabla \bar{S}(z,F_1), \ldots , \nabla \bar{S}(z,F_k) ) \, \mathbb{V}(z)^{-1}.
\end{equation}
Since the inversion of a matrix is a continuous mapping (see Lemma~\ref{Thm:AppendixInverseCont}), we obtain that $h$ is continuous in every $z \in U_x$ due to the assumptions imposed on $\bar{S}$ and $\bar{V}$. \\
Finally, let Assumptions~(\nameref{As:V3}) and (\nameref{As:S2}) hold, fix $x \in \intr (\mathsf{A})$ and choose $F_1, \ldots, F_k \in \mathcal{F}$ such that $\bar{V}(x,F_1), \ldots, \bar{V}(x, F_k)$ are linearly independent. Defining $\mathbb{V}$ as in (\ref{Eqn:MappingMathbbV})  and performing the same steps as in the previous part of the proof, we again arrive at the Representation~(\ref{Eqn:RepOfh}) of $h$ for all points in a neighborhood $U_x$ of $x$. Now let $U_x^S$ be an open neighborhood of $x$ such that for all $i = 1, \ldots, k$ the mapping $z \mapsto \nabla \bar{S}(z, F_i)$ is Lipschitz continuous in $U_x^S$. Similarly, take an open set $U_x^V$ containing $x$ such that the mapping $z \mapsto \mathbb{V}(z)^{-1}$ is Lipschitz continuous in $U_x^V$. This is possible since $\mathbb{V}$ is continuously differentiable and inverting a matrix is a continuously differentiable operation by Lemma~\ref{Thm:AppendixInverseCont}. Consequently, we obtain that $h$ is Lipschitz continuous in the open set $U_x \cap U_x^S \cap U_x^V$, which shows that $h$ is locally Lipschitz continuous.
\end{proof}

\begin{remark}
Note that the function $h$ provided by Theorem~\ref{Thm:OsbandPrinciple1} is uniquely determined. To see this, recall that $h$ is defined using the unique coefficients $h_{l1}(x) , \ldots, h_{lk}(x)$ in the previous proof.
\end{remark}

In order to complete this subsection, we consider how the function $h$ of Theorem~\ref{Thm:OsbandPrinciple1} behaves under the revelation principle stated in Proposition~\ref{Thm:RevelationPrinciple}.  The result is the following corollary, which is mentioned in \cite[Remark 3.6]{FissZiegArxiv} without proof.

\begin{cor}
Let the assumptions of Theorem~\ref{Thm:OsbandPrinciple1} hold, and for $\mathsf{A}' \subseteq \mathbb{R}^k$ let $g: \mathsf{A} \rightarrow \mathsf{A}'$ be a continuously differentiable bijection. Moreover, define the functional $T_g := g \circ T$ and the functions $S_g$ and $V_g$ in the same way as in Proposition~\ref{Thm:RevelationPrinciple}. Then there exists a matrix-valued function $h_g$ such that for all $x' \in \intr (\mathsf{A}')$ and $F \in \mathcal{F}$ we have
\begin{equation*}
\partial_l \bar{S}_g (x',F) = \sum_{m=1}^{k} (h_g)_{lm} (x') \, (\bar{V}_g )_m (x',F),
\end{equation*}
where $h_g$ relates to the function $h$ of Theorem~\ref{Thm:OsbandPrinciple1} via
\begin{equation*}
(h_g)_{lm} (x') = \sum_{i=1}^{k} \partial_l (g^{-1})_i (x') \, h_{im} ( g^{-1} (x') ).
\end{equation*}
\end{cor}

\begin{proof}
The revelation principle as stated in Proposition~\ref{Thm:RevelationPrinciple} implies that $T_g$ is elicitable and identifiable whenever $T$ is. Furthermore, let $h$ be the function of Theorem~\ref{Thm:OsbandPrinciple1}. Differentiating $\bar{S}_g(\cdot,F)$ with respect to the component $l$ using the chain rule of calculus gives 
\begin{align*}
\partial_l \bar{S}_g (x',F) &= \sum_{i=1}^{k} \partial_i \bar{S}(g^{-1} (x') , F) \, \partial_l (g^{-1})_i (x') \\
&= \sum_{m=1}^{k} \sum_{i=1}^{k} \partial_l (g^{-1})_i (x') h_{im} (g^{-1} (x') ) \, \bar{V}_m(g^{-1} (x'), F)
\end{align*}
for any $x' \in \intr (\mathsf{A}')$ and $F \in \mathcal{F}$. From this we obtain the first formula of the corollary and since the function $h_g$ is unique, it must have the claimed representation.
\end{proof}

\subsection{Pointwise version of Osband's principle} 

This subsection shows that Equation~(\ref{Eqn:OsbandsPrinciple}) also holds pointwise, i.e. with $F$ replaced by $y$. The result is the second version of Osband's principle, which directly connects the functions $S$ and $V$. In order to prove it, two further assumptions, which are also given in Fissler and Ziegel~\cite{FissZiegArxiv}, have to be added.

\phantomsection
\begin{assumption2}[F1] \label{As:F1}
For every $y \in \mathsf{O}$ there exists a sequence $(F_n)_{n \in \mathbb{N}} \subset \mathcal{F}$ of distributions which converges weakly to the Dirac-measure $\delta_y$ and such that for all $n \in \mathbb{N}$ the support of $F_n$ is contained in a compact set $K \subseteq \mathsf{O}$.
\end{assumption2}

\begin{assumption2}[VS1] \label{As:VS1}
The complement of the set
\begin{equation*}
C := \lbrace (x,y) \in \mathsf{A} \times \mathsf{O} \mid V(x, \cdot) \text{ and } S(x, \cdot) \text{ are continuous at } y \rbrace
\end{equation*}
has $(k+d)$-dimensional Lebesgue-measure zero.
\end{assumption2}

Besides, we add another condition on the function $V$. This assumption is not mentioned in \cite{FissZiegArxiv}, but we think that it is necessary in order to prove the next theorem. We refer to Subsection~\ref{Sec:JustifyAssumption} for a detailed discussion. The assumption is the following:

\begin{assumption2}[B1]  \label{As:B1}
The function $V$ is locally bounded.
\end{assumption2}

The next theorem is a version of \cite[Prop. 3.4]{FissZieg} with two modifications. At first, we add Assumption~(\nameref{As:B1}) for reasons explained in  Subsection~\ref{Sec:JustifyAssumption}. Secondly, it is required in \cite{FissZieg} that $\intr (\mathsf{A})$ is a star domain instead of a hyperrectangle, see also Remark~\ref{Rem:WhyHyperrectangle}.

\begin{theorem}  \label{Thm:OsbandPrinciple2}
Let $\mathcal{F}$ be convex, assume that $\intr (\mathsf{A}) \subseteq \mathbb{R}^k$ is a hyperrectangle and let $T: \mathcal{F} \rightarrow \mathsf{A}$ be a surjective elicitable and identifiable functional with strict $\mathcal{F}$-identification function $V$ and strictly $\mathcal{F}$-consistent scoring function $S$. Suppose assumptions (\nameref{As:V1}), (\nameref{As:V2}), (\nameref{As:S1}), (\nameref{As:F1}), (\nameref{As:VS1}), (\nameref{As:B1}), and \ref{Thm:AssumptionLocSV} hold. Then for almost all $(x,y) \in \mathsf{A} \times \mathsf{O}$ the function $S$ is of the form
\begin{align}  \label{Eqn:OsbandsPrincipleII}
S(x,y) = \sum_{r=1}^{k} \sum_{m=1}^{k} \int_{z_r}^{x_r} &h_{rm} (x_1, \ldots, x_{r-1}, v, z_{r+1}, \ldots, z_k) \\
\times &V_m (x_1, \ldots, x_{r-1}, v, z_{r+1}, \ldots, z_k, y) \, \mathrm{d} v + a(y) ,   \nonumber
\end{align}
where $(z_1, \ldots, z_k) \in \intr (\mathsf{A})$ and $a: \mathsf{O} \rightarrow \mathbb{R}$ is some $\mathcal{F}$-integrable function and $h$ is the unique matrix-valued function provided by Theorem~\ref{Thm:OsbandPrinciple1}. On the level of the expected score $\bar{S}(x,F)$, Equation~(\ref{Eqn:OsbandsPrincipleII}) holds for all $x \in \intr (\mathsf{A})$ and all $F \in \mathcal{F}$.
\end{theorem}

\begin{remark}  \label{Rem:WhyHyperrectangle}
As mentioned above, \cite[Prop. 3.4]{FissZieg} require $\intr (\mathsf{A})$ to be a star domain and $z \in \intr(\mathsf{A})$ to be a star point. Looking at the representation of $S$, we see that all integrals evaluate $h$ and $V$ on line segments between the points $(x_1, \ldots, x_{r-1}, z_r, z_{r+1}, \ldots, z_k)$ and $(x_1, \ldots, x_{r-1}, x_r, z_{r+1}, \ldots, z_k)$ for $r \in \lbrace 1, \ldots, k \rbrace$. These lines are edges of a hyperrectangle which is defined by the two corners $x$ and $z$, see also Figure~\ref{Fig:Edges} for an illustration. Consequently, we think that it is insufficient to work with a star domain $\intr (\mathsf{A})$ since such a domain does not necessarily contain all hyperrectangles spanned by its points. Since $h$, $V$ and $S$ might not be defined outside of $\intr (\mathsf{A})$, Representation~(\ref{Eqn:OsbandsPrincipleII}) might be invalid and thus we propose to assume that $\intr (\mathsf{A})$ is a hyperrectangle. Nevertheless, we think that the result can also be proved if $\intr (\mathsf{A})$ is a star domain or more generally path-connected. In this case we need to integrate along paths instead of line segments which increases the complexity of (\ref{Eqn:OsbandsPrincipleII}) and makes it difficult to apply the theorem. Hence, we do not pursue this approach.
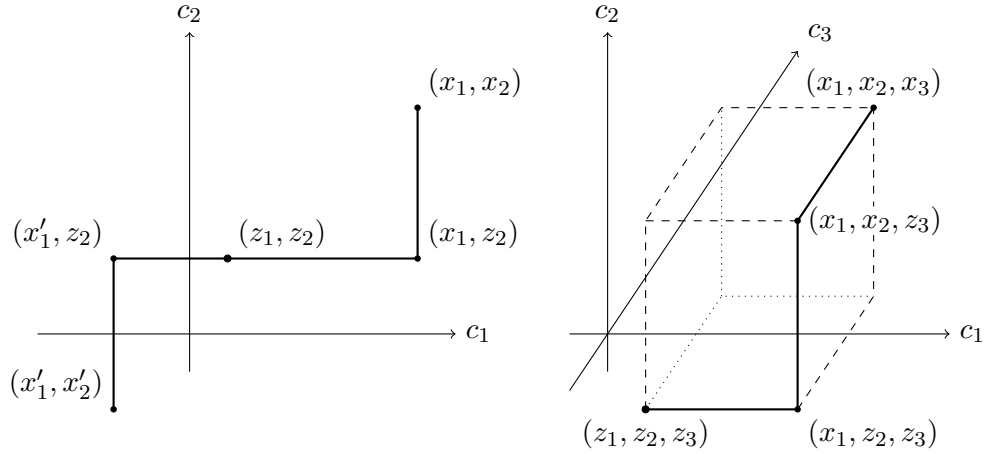
\begin{figure}[h]
\centering
\begin{tikzpicture}
\coordinate (x) at (3,3);
\coordinate (z) at (0.5,1);
\coordinate (xz) at (3,1);
\coordinate (xx) at (-1,-1);
\coordinate (xxz) at (-1,1);

\draw[->] (0,-0.5) -- (0,4) node[above] {$c_2$};
\draw[->] (-2,0) -- (3.5,0) node[right] {$c_1$};

\filldraw [black] (x) circle (1pt);
\draw (x) node[above right] {$(x_1, x_2)$};
\filldraw [black] (z) circle (1.3pt);
\draw (z) node[above right] {$(z_1, z_2)$};
\filldraw [black] (xz) circle (1pt);
\draw (xz) node[above right] {$(x_1, z_2)$};
\filldraw [black] (xx) circle (1pt);
\draw (xx) node[above left] {$(x_1', x_2')$};
\filldraw [black] (xxz) circle (1pt);
\draw (xxz) node[above left] {$(x_1', z_2)$};
\draw[thick] (xx) -- (xxz) -- (z) -- (xz) -- (x);

\begin{scope}[xshift=5.5cm]

\coordinate (x3) at (3.5,3);
\coordinate (z3) at (0.5,-1);
\coordinate (zx1) at (2.5,-1);
\coordinate (zx2) at (2.5,1.5);
\coordinate (h1) at (0.5,1.5);
\coordinate (h2) at (1.5,3);
\coordinate (h3) at (3.5,0.5);
\coordinate (hh) at (1.5,0.5);

\draw[->] (0,-0.5) -- (0,4) node[above] {$c_2$};
\draw[->] (-0.5,0) -- (4.5,0) node[right] {$c_1$};
\draw[->] ($-0.5*(1,1.5)$) -- ($ 2.5*(1,1.5)$) node[above right] {$c_3$};

\draw[dotted] (z3) -- (hh) -- (h3);
\draw[dotted] (hh) -- (h2);
\draw[dashed] (z3) -- (h1) -- (h2) -- (x3);
\draw[dashed] (h1) -- (zx2);
\draw[dashed] (zx1) -- (h3) -- (x3);

\filldraw [white] ($(zx2) + (1,0)$) circle (5.2pt);

\filldraw [black] (x3) circle (1pt);
\draw (x3) node[above] {$(x_1, x_2, x_3)$};
\filldraw [black] (z3) circle (1.3pt);
\draw (z3) node[below] {$(z_1, z_2, z_3)$};
\filldraw [black] (zx1) circle (1pt);
\draw (zx1) node[below right] {$(x_1, z_2, z_3)$};
\filldraw [black] (zx2) circle (1pt);
\draw (zx2) node[right] {$(x_1, x_2, z_3)$};

\draw[thick] (z3) -- (zx1) -- (zx2) -- (x3);

\end{scope}
\end{tikzpicture}
\caption{Illustration of Theorem~\ref{Thm:OsbandPrinciple2}. The solid lines are the edges along which the integrals in (\ref{Eqn:OsbandsPrincipleII}) are calculated, in the two-dimensional and three-dimensional situation. The label $c_i$ represents the $i$-th coordinate. The dashed or dotted lines illustrate the cuboid which is spanned by $(z_1,z_2,z_3)$ and $(x_1,x_2,x_3)$.}
\label{Fig:Edges}
\end{figure}
\end{remark}

\begin{proof}
We proceed similar to \cite{FissZiegArxiv}. At first, Assumption~\ref{Thm:AssumptionLocSV} ensures that $V(\cdot, y)$ is locally integrable for all $y \in \mathsf{O}$. Hence, select some point $z \in \intr (\mathsf{A})$ and define the function $H : \intr(\mathsf{A} ) \times \mathsf{O} \rightarrow \mathbb{R}$ via
\begin{align*}
H(x,y) := \sum_{r=1}^{k} \sum_{m=1}^{k} \int_{z_r}^{x_r} &h_{rm} (x_1, \ldots, x_{r-1}, v, z_{r+1}, \ldots, z_k) \\
\times &V_m (x_1, \ldots, x_{r-1}, v, z_{r+1}, \ldots, z_k, y) \, \mathrm{d} v ,
\end{align*}
where  $h$ is the continuous matrix-valued function from Theorem~\ref{Thm:OsbandPrinciple1}. The next step is to show continuity of $H(x, \cdot)$. Let the set $C$ be defined as in Assumption~(\nameref{As:VS1}) and note that $(x,y) \in C^c$ does not guarantee continuity of $H (x, \cdot)$ in $y$. Hence, define the set $\mathsf{A}_y := \lbrace x \in \mathsf{A} \mid S(x,\cdot) \text{ or } V(x, \cdot) \text{ not continuous in } y \rbrace$ for any $y \in \mathsf{O}$ and notice that $ x \in \mathsf{A}_y \Leftrightarrow (x,y) \in C^c$ holds true. Consequently, there is a set $N \subset \mathsf{O}$ having Lebesgue measure zero such that for all $y \in N^c$ the set $\mathsf{A}_y$ is a Lebesgue null set. We thus select $(x_0, y_0) \in C^c$ such that $y_0 \in N^c$ is satisfied and observe that the functions $S(x, \cdot)$ and $V(x, \cdot)$ are continuous in $y_0$ for a.e. $x \in \mathsf{A}$ and in particular for $x=x_0$. The values $s \in \mathbb{R}$ such that $(x_1, \ldots, x_{r-1}, s, x_{r+1}, \ldots, x_k)^\top \in \mathsf{A}_{y_0}$ is satisfied form a null set for a.e. $x \in \mathsf{A}$, hence, we have for a.e. $x \in \mathsf{A}$ that $V_m (x_1, \ldots, x_{r-1}, s, x_{r+1}, \ldots ,x_k, \cdot)$ is continuous in $y_0$ for a.e. $s \in \mathbb{R}$ and $m,r \in \lbrace 1, \ldots, k \rbrace$.

We continue by proving that $H(x, \cdot)$ is continuous in $y_0$ for a.e. $x \in \intr(\mathsf{A})$ and locally bounded for all $x \in \intr(\mathsf{A} )$. At first observe that by Assumption (\nameref{As:B1}) and the continuity of $h$, for any two compact sets $K_1 \subset \mathsf{A}$ and $K_2 \subset \mathsf{O}$ we have that 
\begin{equation}  \label{Eqn:BoundednesshV}
\vert h_{rm}(x) V_m(x, y) \vert \leq M < \infty \quad \text{for } r,m \in \lbrace 1, \ldots, k \rbrace \text{ and } (x,y) \in K_1 \times K_2. 
\end{equation}
Hence, for any $x \in \intr (\mathsf{A})$ we have that $H(x,\cdot)$ is bounded on compacts. To show continuity, we choose a sequence $(y_n)_{n \in \mathbb{N}}$ which is contained in a compact neighborhood of $y_0$ and satisfies $y_n \rightarrow y_0$ as $n \rightarrow \infty$. For a.e. $x \in \intr( \mathsf{A})$ we have $h_{rm}(x) V_m(x, y_n) \rightarrow h_{rm}(x) V_m(x, y_0)$ due to the choice of $y_0$ and the arguments following it. Using (\ref{Eqn:BoundednesshV}) again, we see that the sequence $(h_{rm}(x) V_m(x, y_n))_{n \in \mathbb{N}}$ is bounded by a constant if $x$ is in a compact set. Since all integrals belonging to $H$ are taken over compact sets,  dominated convergence implies $H(x,y_n) \rightarrow H(x,y_0)$ for a.e. $x \in \intr (\mathsf{A})$ as $n \rightarrow \infty$.

The final step of the proof is to pass from $\bar{S}(x,F)$ to $S(x,y_0)$. Using a telescoping argument, we obtain for all $x \in \intr (\mathsf{A})$, $F \in \mathcal{F}$
\begin{align}
\bar{S}(x,F) - \bar{S}(z,F) &= \sum_{r=1}^{k} \bar{S}(x_1, \smalldots, x_r, z_{r+1}, \smalldots, z_k,F) - \bar{S}(x_1, \smalldots, x_{r-1}, z_r, \smalldots, z_k,F) \nonumber \\
&= \sum_{r=1}^{k} \int_{z_r}^{x_r} \partial_r \bar{S}(x_1, \ldots, x_{r-1}, v, z_{r+1}, \ldots, z_k, F) \, \mathrm{d} v  \label{Eqn:RepSandH}  \\
&= \bar{H}(x, F)  \nonumber
\end{align}
due to Osband's principle \ref{Thm:OsbandPrinciple1} and the Fubini-Tonelli theorem. Employing Assumption~(\nameref{As:F1}) gives a sequence $(F_n)_{n \in \mathbb{N}} \subset \mathcal{F}$ which converges weakly to $\delta_{y_0}$ and the support of this sequence lies in some compact set $K$. We use a Skorohod representation (see Theorem~\ref{Thm:SkorohodRep}) to get a probability space $(\Omega, \mathscr{A}, \mathbb{P})$ such that $Y_n, Y: \Omega \rightarrow \mathsf{O}$ are random variables which satisfy $Y_n =^d F_n$ under $\mathbb{P}$, $Y \equiv y_0$ and $Y_n \rightarrow Y$ for $n \rightarrow \infty$ almost surely. Next, we apply the continuous mapping theorem for almost sure convergence (see for instance van der Vaart \cite[Thm. 2.3]{vanderVaart}). Let $\mathsf{D}_{S,H}$ be a measurable set which contains all discontinuities of $S(x,\cdot)$ and $H(x,\cdot)$, which means it cannot include $y_0$. Consequently, we have $\mathbb{P}( Y \in \mathsf{D}_{S,H}) = \delta_{y_0} (\mathsf{D}_{S,H}) = 0$, implying $S(x,Y_n) \rightarrow S(x,y_0)$ and $H(x, Y_n) \rightarrow H(x,y_0)$ almost surely and for a.e. $x \in \intr (\mathsf{A})$. Since $S(x, \cdot)$ and $H(x, \cdot)$ are locally bounded, they are bounded on the compact set $K$ in which the $Y_n$ take values a.s. and hence dominated convergence implies $\mathbb{E} S(x,Y_n)  \rightarrow S(x,y_0)$ and $\mathbb{E} H(x,Y_n) \rightarrow H(x, y_0)$ for $n \rightarrow \infty$. 

The sequence $\mathbb{E}(S(x, Y_n) - H(x, Y_n)) = \bar{S}(x,F_n) - \bar{H}(x, F_n)$ is equal to $\bar{S}(z,F_n)$ by Equation~(\ref{Eqn:RepSandH}) and thus does not depend on $x$. Therefore, the limit is also independent of $x$ and we may define the function $a$ via
\begin{equation*}
a(y_0) := \underset{n \rightarrow \infty}{\lim} \mathbb{E}(S(x, Y_n) - H(x, Y_n)) = S(x,y_0) - H(x,y_0).
\end{equation*}
Since $S(x,Y_n) = H(x,Y_n) + (S(x,Y_n) - H(x,Y_n))$  and both sides must have the same limit, the Representation~(\ref{Eqn:OsbandsPrincipleII}) holds for our choice of $y_0$. Finally, we repeat all previous steps for different $y_0$ and this gives the identity $a(y) = S(x,y) - H(x,y)$ for a.e. $y \in \mathsf{O}$, implying that $a$ is an $\mathcal{F}$-integrable function.
\end{proof}

\subsection{Examples showing the necessity of Assumption (\nameref{As:B1})} 
\label{Sec:JustifyAssumption}

This section contains two examples which are designed to show that Assumption~(\nameref{As:B1}) or related assumptions are needed to prove Theorem~\ref{Thm:OsbandPrinciple2}. The crucial point in the proof is the local boundedness of $H(x, \cdot)$. Assumption~\ref{Thm:AssumptionLocSV} ensures that $V(x, \cdot)$ is locally bounded for all $x \in \mathsf{A}$. However, $H(x,\cdot)$ contains an integral of such a function. Therefore, it remains unclear how local boundedness of $H(x,\cdot)$ can be ensured without using (\nameref{As:B1}) or a similar statement. The first example proves that integration over a family of bounded functions does not always lead to a locally bounded function. 

\begin{example} \label{Thm:ExCounterBoundedness}
Define the function $f: (0,1] \times [0,1] \rightarrow \mathbb{R}$ via
\begin{equation*}
f(x,y) = \frac{1}{x^y} \, \mathbbm{1}_{[0,1)} (y) ,
\end{equation*}
where $\mathbbm{1}$ represents the indicator function. The function $f(x, \cdot)$ is displayed in Figure~\ref{Fig:ExCounterBoundedness} for different values of $x$. For any $x \in (0,1]$ the mapping $y \mapsto f(x,y)$ is bounded and for any $y \in [0,1] $ it holds that $x \mapsto f(x,y)$ is integrable. Hence, we define the integrated function
\begin{equation*}
F: [0,1] \rightarrow \mathbb{R}, \quad y \mapsto F(y) := \int_{0}^{1} f(x,y) \, \mathrm{d}x .
\end{equation*}

\begin{figure}[ht]
\centering
\includegraphics[width= 0.8\textwidth, trim= 0mm 0mm 0mm 10mm, clip]{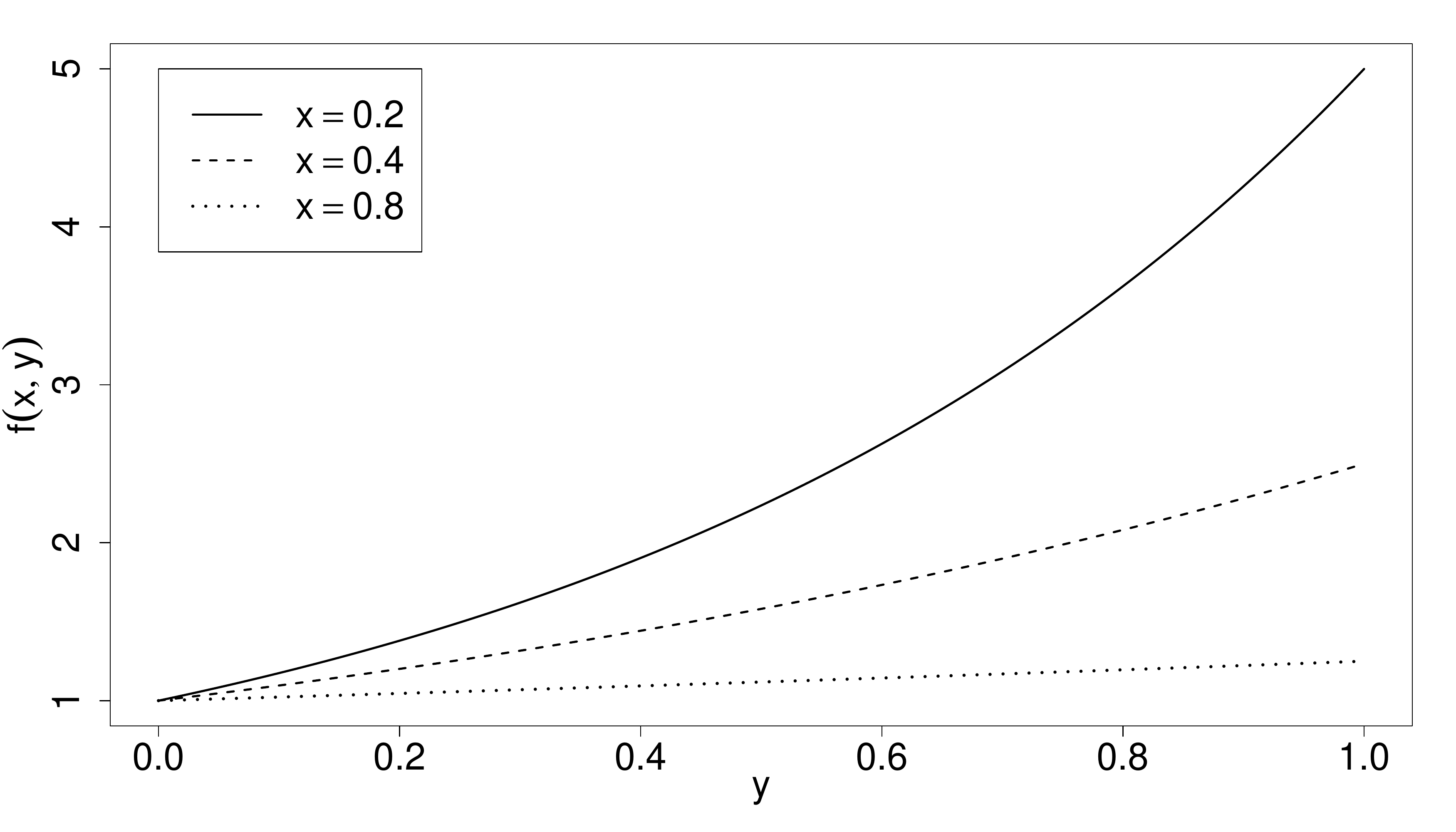}
\caption{Plot of the function $f(x,y)$ from Example~\ref{Thm:ExCounterBoundedness} for different values of $x$.}
\label{Fig:ExCounterBoundedness}
\end{figure}

$F$ fails to be locally bounded, which becomes obvious by calculating it explicitly. For $y=1$ the function $f$ vanishes and so we have $F(1) = 0$. For all other values of $y$ we obtain
\begin{align*}
F(y) = \int_{0}^{1} \frac{1}{x^y} \, \mathrm{d}x  =  \left. \frac{1}{1-y} x^{1-y} \, \right|_{x=0}^{x=1} = \frac{1}{1-y}  ,
\end{align*}
which immediately shows that $F$ fails to be bounded on all compact intervals $[a,1]$ for $0 \leq a <1$.
\end{example}

The second example shows that integration over a family of bounded continuous functions does not necessarily lead to a continuous function. Consequently, it is unclear how continuity of $H(x, \cdot)$ at $y_0$ can be secured without employing Assumption~(\nameref{As:B1}) or a similar statement. 

\begin{example} \label{Thm:ExCounterContinuity}
Define the function $f : (0,1] \times \mathbb{R} \rightarrow \mathbb{R}$ via
\begin{equation*}
f(x,y) := \frac{y}{x^2} \, \mathbbm{1}_{ [0, x/2) } (y) + \left(\frac{1}{x} - \frac{y}{x^2} \right) \mathbbm{1}_{[x/2, x]} (y) .
\end{equation*}
This is a family of hat functions, which is displayed in Figure~\ref{Fig:ExCounterContinuity} for some values of $x$. For any $x \in (0,1]$ the mapping $y \mapsto f(x,y)$ is bounded and continuous, and for any $y \in \mathbb{R}$ we see that $x \mapsto f(x,y)$ is integrable. As in the previous example, we define the integrated function $F$ via
\begin{equation*}
F: \mathbb{R} \rightarrow \mathbb{R}, \quad y \mapsto F(y) := \int_{0}^{1} f(x,y) \, \mathrm{d}x .
\end{equation*}

\begin{figure}[ht]
\centering
\includegraphics[width= 0.8\textwidth, trim= 0mm 0mm 0mm 10mm, clip]{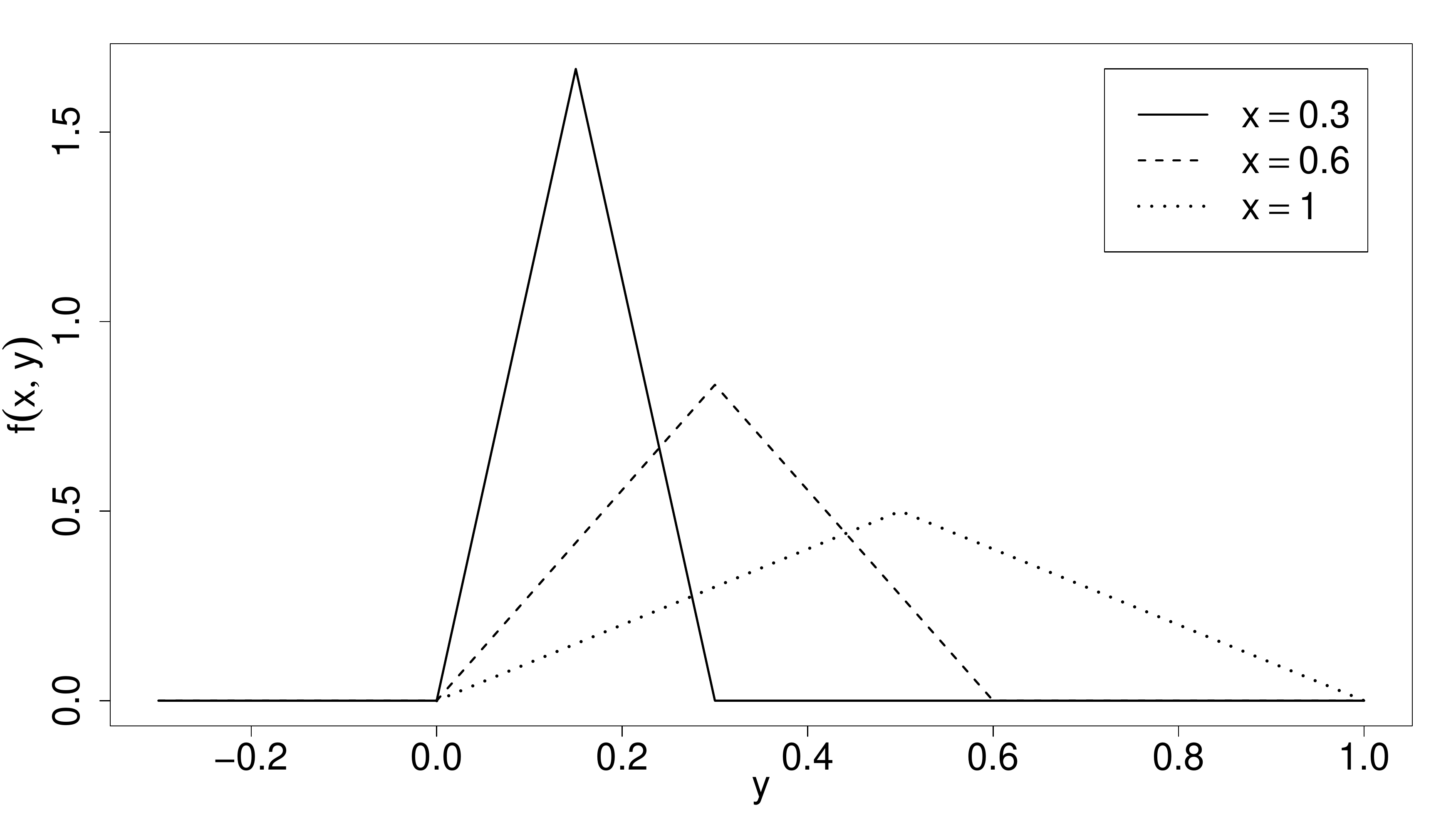}
\caption{Plot of the function $f(x,y)$ from Example~\ref{Thm:ExCounterContinuity} for different values of $x$.}
\label{Fig:ExCounterContinuity}
\end{figure}

$F$ fails to be continuous in $0$, which is shown by calculating the integral. For any $y \leq 0$ or $y \geq 1$ the function $f(x,y)$ vanishes and therefore $F$ is zero for these values of $y$. For $0 < y < 1/2$ we calculate
\begin{align*}
F(y) = \int_{2y}^{1} \frac{y}{x^2} \, \mathrm{d}x + \int_{y}^{2y} \frac{1}{x} - \frac{y}{x^2} \, \mathrm{d}x = \left. - \frac{y}{x} \, \right|_{x=2y}^{x=1} + \ln (x) + \left. \frac{y}{x} \, \right|_{x=y}^{x=2y} = -y + \ln (2),
\end{align*}
and similarly for $1/2 \leq y < 1$ we obtain
\begin{align*}
F(y) = \int_{y}^{1} \frac{1}{x} - \frac{y}{x^2} \, \mathrm{d}x = \ln (x) + \left. \frac{y}{x} \, \right|_{x=y}^{x=1} = y - \ln (y) - 1. 
\end{align*}
All in all, we have the representation
\begin{equation*}
F(y) = \left\lbrace
\begin{array}{ll}
-y + \ln(2), & \, \text{if } y \in (0,1/2] \\
y - \ln(y) - 1, & \, \text{if } y \in (1/2, 1) \\
0, & \, \text{else } 
\end{array}
\right. 
\end{equation*}
which shows that $F$ is not continuous at $y = 0$.
\end{example}

\chapter{Applications of elicitability}
\label{Chapter2}

This chapter presents three areas in the field of mathematical statistics where elicitability, and sometimes also identifiability, plays a central role. These three topics include parameter estimation, generalized regression, and forecast comparison. They are included in order to show why the concept of elicitability plays a non-trivial role in statistics and has attracted attention by researchers. However, the sections contain only short introductions to the mentioned concepts.  In the following, if not stated otherwise, all random variables are defined on a measure space $(\Omega, \mathscr{A}, \mathbb{P})$.

\section{Forecast comparison} 
\label{Sec:ForecastAndBacktest}

This section considers a decision maker who needs information concerning a real-valued (or $\mathbb{R}^d$-valued) random variable $Y$ which has some unknown distribution function $F \in \mathcal{F}$. To get this information, the decision maker asks several forecasters who can form an opinion about the distribution of $Y$ for their best forecasts. After some assessment of the quality of their forecasts, the forecasters are rewarded with money. While the forecasters' aim is to maximize the expected payoff, the decision maker wants to obtain the most accurate forecast. \\
Stated more technically, $n$ forecasters issue reports $R_1, \ldots, R_n$ which are compared to a realization $y$ of $Y$ using a scoring function $S$. The scoring function is chosen by the decision maker and determines which forecasts are the best or most accurate. If $R_j$ and $R_i$ satisfy $S(R_j, y) < S(R_i,y)$, the forecast $R_j$ is considered better than $R_i$ and the $j$-th forecaster's payoff is higher than the $i$-th forecaster's payoff. Hence, forecasters who minimize their (expected) score maximize their (expected) payoff. \\
\par
We begin by introducing probabilistic forecasting, a setting for which elicitability is irrelevant. Nevertheless, we consider it in order to understand how the necessity for elicitability emerges when moving from probabilistic forecasting to point forecasting. At the end of this section, we show the relevance of elicitability in deciding between different point forecasts.

\subsection{Probabilistic forecasting} 

This subsection is a short introduction to probabilistic forecasting based on material of Gneiting and Raftery~\cite{GneitingRaftery}. In the setting of probabilistic forecasting, each forecaster reports a probability measure $P \in \mathcal{P}$, where $\mathcal{P}$ is a convex class of probability measures on the measurable space $(\Omega, \mathscr{A})$. The forecasts are rated using a scoring rule $S: \mathcal{P} \times \Omega \rightarrow \mathbb{R}$. As introduced in Section~\ref{Sec:DefinitionsFramework}, a measurable function $G: \Omega \rightarrow \mathbb{R}$ is called $\mathcal{P}$-integrable if it is integrable with respect to all $P \in \mathcal{P}$ and we use $\bar{G}(P)$ as short notation for $\int_\Omega G(\omega) \, \mathrm{d}P(\omega)$. \\
\par
One problem which the decision maker wants to avoid is that the forecasters do not report the probability measure they assume to be correct. This could be because other probability measures maximize their expected payoff. In order to incentivize the forecasters to report truthfully, the decision maker should communicate the scoring rule which is used for payoff calculation to the forecasters. Moreover, this scoring rule should be \textit{proper}. Propriety means for the forecasters that the minimization of the expected score (which is equivalent to the maximization of the expected payoff) is achieved by reporting the probability measure they believe in. The mathematical definition according to \cite{GneitingRaftery} is the following.

\begin{definition}
A scoring rule is any real-valued function $S: \mathcal{P} \times \Omega \rightarrow \mathbb{R}$ such that for all $P \in \mathcal{P}$ the mapping $\omega \mapsto S(P,\omega)$ is  $\mathcal{P}$-integrable. The scoring rule $S$ is called $\mathcal{P}$-proper if $\bar{S}(Q,Q) \leq \bar{S}(P,Q)$ holds for all $P,Q \in \mathcal{P}$. It is strictly $\mathcal{P}$-proper if it is $\mathcal{P}$-proper and for any $P,Q \in \mathcal{P}$ the equality $\bar{S}(Q,Q) = \bar{S}(P,Q)$ implies $Q=P$.
\end{definition}

\begin{remark}
The definition of proper scoring rules is slightly modified to make it consistent with scoring functions as in Definition~\ref{Def:ScoringFunction}. In particular, the orientation of the scoring rules is changed, i.e. we look at minima instead of maxima. Moreover, \cite{GneitingRaftery} only require $S$ to be quasi-integrable, meaning that it can take values in $\bar{\mathbb{R}} := \mathbb{R} \cup \lbrace -\infty, \infty \rbrace$. Hence, a definition of regular scoring rules is added in \cite{GneitingRaftery} which is not needed here.
\end{remark}

\begin{definition}
Let $G: \mathcal{P} \rightarrow \mathbb{R}$ be a convex function. A function $G^* (P, \cdot) : \Omega \rightarrow \mathbb{R}$ is called a \textit{subtangent} of $G$ in the point $P \in \mathcal{P}$ if it is $\mathcal{P}$-integrable and satisfies
\begin{equation}  \label{Eqn:SubtangentInequality}
G(Q) \geq G(P) + \int_\Omega G^* (P, \omega) \, \mathrm{d}(Q-P)(\omega)
\end{equation}
for all $Q \in \mathcal{P}$.
\end{definition}

Using the subtangent definition, it is possible to completely characterize all proper scoring rules. The following characterization is due to Gneiting and Raftery~\cite[Thm. 1]{GneitingRaftery} and its proof is modified to fit our definition.

\begin{theorem}
A scoring rule $S : \mathcal{P} \times \Omega \rightarrow \mathbb{R}$ is (strictly) $\mathcal{P}$-proper if and only if there exists a (strictly) convex, real-valued function $G: \mathcal{P} \rightarrow \mathbb{R}$ such that
\begin{equation}  \label{Eqn:ProperScoringRep}
S(P,\omega) = \bar{G}^*(P,P) - G^*(P, \omega) - G(P)
\end{equation}
for all $P \in \mathcal{P}$, $\omega \in \Omega$ and a subtangent $G^*(P, \cdot) : \Omega \rightarrow \mathbb{R}$ of $G$ at $P$.
\end{theorem}

\begin{remark}
Representation (\ref{Eqn:ProperScoringRep}) shows that (strictly) proper scoring rules are similar to (strictly) consistent scoring functions for expectations which are treated in Example~\ref{Thm:ExMeanIsElicitable}. Both functions consist of a convex function ($f$ or $G$) and a `subtangent term'. In the scoring function case, the subtangent inequality becomes $f(r) \geq f(s) + \nabla f(s)^\top (r-s)$ and $\nabla f$ can be interpreted as the subtangent of $f$. In fact, the definition of a subtangent can also be used for real-valued functions, and scoring functions as in Example~\ref{Thm:ExMeanIsElicitable} can be defined using this concept.
\end{remark}

\begin{proof}
Let $S$ be as stated in the theorem and $G$ be convex. Then $S$ is a scoring rule, since $G^*$ is per definitionem $\mathcal{P}$-integrable. For any $Q,P \in \mathcal{P}$ we employ the subtangent Inequality~(\ref{Eqn:SubtangentInequality}) and the Representation~(\ref{Eqn:ProperScoringRep}) to obtain
\begin{align*}
\bar{S}(P,Q) &= \bar{G}^* (P,P) - \bar{G}^*(P,Q) - G(P) \\
&= - \bar{G}^*(P, Q-P) - G(P) \geq -G(Q) = \bar{S}(Q,Q) ,
\end{align*}
which shows propriety of $S$. If $\bar{S}(P,Q) = \bar{S}(Q,Q)$ holds, we have $ \bar{G}^*(P,P) + G(Q) = G(P) + \bar{G}^*(P,Q)$ and this together with the subtangent property of $G^*(P, \cdot)$ leads to
\begin{align*}
G\left( \frac{1}{2}P + \frac{1}{2} Q \right) &\geq G(P) + \bar{G}^* \left( P, \frac{1}{2}P + \frac{1}{2}Q - P \right) \\
&= G(P) - \frac{1}{2} (\bar{G}^*(P,P) - \bar{G}^*(P,Q) ) = \frac{1}{2} G(P) + \frac{1}{2} G(Q).
\end{align*}
Therefore, if $G$ is strictly convex, we must have $P = Q$ and hence $S$ is strictly $\mathcal{P}$-proper whenever $G$ is strictly convex. For the converse implication, let $S$ be a $\mathcal{P}$-proper scoring rule and define the mapping $G: \mathcal{P} \rightarrow \mathbb{R}$, $P \mapsto - \bar{S}(P,P)$. For all $P, Q \in \mathcal{P}$ and $\lambda \in (0,1)$ we have
\begin{align*}
\lambda G(P) + (1-\lambda) G(Q) &= -\lambda \bar{S}(P,P) - (1-\lambda) \bar{S}(Q,Q) \\
&\geq - \lambda \bar{S}(\lambda P + (1-\lambda) Q,P) - (1-\lambda) \bar{S}(\lambda P + (1-\lambda) Q, Q) \\
&= - \bar{S}(\lambda P + (1-\lambda) Q,\lambda P + (1-\lambda) Q) = G(\lambda P + (1-\lambda) Q),
\end{align*}
so $G$ is convex. Whenever $S$ is strictly proper and $P \neq Q$, the inequality is strict and hence $G$ is strictly convex. Moreover, $G^*(P, \omega) := -S(P, \omega)$ is a subtangent of $G$ at the point $P \in \mathcal{P}$ due to the propriety of $S$ and the fact that $G(P) = \bar{G}^*(P,P)$. Using these choices of $G$ and $G^*$ the Representation~(\ref{Eqn:ProperScoringRep}) follows.
\end{proof}

Before turning to the next subsection, we illustrate the close connection between proper scoring rules and consistent scoring functions. For this statement we interpret distribution functions in $\mathcal{F}$ as probability measures on $\mathsf{O}$. 

\begin{theorem}[Gneiting {\cite[Thm. 3]{GneitingPoints}}]
Let $S$ be an $\mathcal{F}$-consistent scoring function for a functional $T$. Then the function $R : \mathcal{F} \times \mathsf{O} \rightarrow \mathbb{R}$ defined via $R(F, y) := S(T(F), y)$ is an $\mathcal{F}$-proper scoring rule. 
\end{theorem} 

The result follows immediately from the used definitions. Note that strict propriety can only be achieved using strict consistency and an injective functional $T$.

\subsection{The need for elicitability in point forecasting}

This subsection considers point forecasting and the difference to probabilistic forecasting. In theory, there is no compelling reason why points should be used instead of whole probability distributions, since the decision maker can always extract all necessary information from a probabilistic forecast. However, point forecasts are widely used in practice and one reason for this is that they are simpler and easier to communicate.\\
\par
In our setting, which is adopted from Gneiting~\cite{GneitingPoints}, the distributions of the class $\mathcal{F}$ are all defined on some \textit{observation domain} $\mathsf{O}$. The decision maker asks the forecasters to report points lying in an \textit{action domain} $\mathsf{A}$. The scoring function $S : \mathsf{A} \times \mathsf{O} \rightarrow \mathbb{R}$ is used to assess the forecasts. Similar to probabilistic forecasting and as pointed out in \cite{GneitingPoints}, the decision maker has to communicate to the forecasters how forecast accuracy is measured, to make them report truthfully. However, since only points are allowed as forecasts, there are now two possible things the decision maker can communicate:

\begin{enumerate}
\item The scoring function $S$. If $S$ is chosen such that $\bar{S}(\cdot, F)$ has a unique minimum for every $F \in \mathcal{F}$, then the  situation is similar to probabilistic forecasting. The forecasters will maximize their expected payoff by reporting the minimum of $\bar{S}(\cdot, F)$, where $F$ is their subjective distribution function. 
\item A functional $T$. The functional tells the forecasters how to translate their opinions of the distribution $F$ into real numbers which they can report. If $T$ is elicitable and $S$ is a strictly $\mathcal{F}$-consistent scoring function for $T$, then $T(F)$ is the unique minimizer of $\bar{S}(\cdot, F)$. Consequently, the forecasters will maximize their expected payoff by reporting $T(F)$, where $F$ is their subjective distribution function.
\end{enumerate}

The second point shows that elicitability of $T$ is crucial to incentivize the forecasters to report truthfully. It guarantees that the choices of the forecasters are in line with the request of the decision maker. It should be remarked that this is also what the term elicitability suggests. An \textit{elicitable} functional enables the decision maker to \textit{elicit} truthful reports of the functional value. Finally, note that the possibilities 1. and 2. are equivalent solutions to the same problem. It is the decision maker's choice to determine which functional is needed to make good decisions or which scoring functions captures best the losses incurred by using inaccurate forecasts.

\subsection{Comparing forecasts using statistical tests} 
\label{Sec:CompBacktest}

This subsection shows how elicitability of a functional $T$ can be used to decide between different forecasts for $T$, based on a set of realizations of $Y$. In the following, we consider a situation where two forecasters issue $n$ point forecasts $(\hat{x}_t)_{t=1,\ldots,n}$ and $(\hat{z}_t)_{t=1,\ldots,n}$. For simplicity, we assume that the times for which the forecasts are made are simply $t= 1, \ldots, n$ and realizations of $Y$ denoted by $(y_t)_{t=1, \ldots, n}$ are given. Moreover, we let $S$ be a scoring function which is strictly $\mathcal{F}$-consistent for $T$. In order to decide which set of forecasts is `better' we can compute the \textit{mean score difference} given by
\begin{equation}  \label{Eqn:MeanScoreDiff}
\frac{1}{n} \sum_{t=1}^{n} S( \hat{x}_t, y_t) - \frac{1}{n} \sum_{t=1}^{n} S(\hat{z}_t , y_t)
\end{equation}
and check whether it is negative or positive. Based on a law of large numbers argument and the strict consistency of $S$, we can argue that a positive value supports $(\hat{z}_t)$ while a negative value supports $( \hat{x}_t)$. However, since  consistency of $S$ is a statement about $\bar{S}( \cdot, F)$ for $F \in \mathcal{F}$, a simple comparison of score values, which can only be an approximation to $\bar{S}$, is unpleasant. In particular, positive or negative values in (\ref{Eqn:MeanScoreDiff}) can be sheer coincidence. This issue leads directly to statistical tests based on the mean score difference which were introduced by Diebold and Mariano~\cite{DieboldMariano}. Since these Diebold-Mariano tests (DM tests in the following) are designed in \cite{DieboldMariano} to handle dependent realizations of $Y$, we present some concepts of time series analysis before we return to them. \\
\par
For the rest of this section, let $(Y_t)_{t \in \mathbb{Z}}$ be a sequence of random variables on the probability space $(\Omega, \mathscr{A}, \mathbb{P})$. For more details on the following concepts, we refer to Shumway and Stoffer~\cite{ShumwayTimeSeries} as well as Durrett~\cite[Sec. 7.C]{DurrettProb}.

\begin{definition}
A sequence of random variables $(Y_t)_{t \in \mathbb{Z}}$ is called \textit{strictly stationary}, if for any $k \in \mathbb{N}$, $h \in \mathbb{Z}$ and $t_1, \ldots, t_k \in \mathbb{Z}$ the distributions of $(Y_{t_1}, \ldots, Y_{t_k} )$ and $(Y_{t_1 + h} , \ldots, Y_{t_k + h} )$ coincide.
\end{definition}

\begin{definition}
For a sequence of random variables $(Y_t)_{t \in \mathbb{Z}}$ define the $\sigma$-algebras $\mathscr{A}_t^\infty := \sigma (Y_t, Y_{t+1}, \ldots)$ and $\mathscr{A}_{-\infty}^t := \sigma (Y_t, Y_{t-1}, \ldots)$. For $k \in \mathbb{N}$ call
\begin{equation*}
\alpha (k) := \sup \lbrace \vert \mathbb{P} (A \cap B) - \mathbb{P}(A) \mathbb{P}(B) \vert \mid t \in \mathbb{Z}, A \in \mathscr{A}_{-\infty}^t , B \in \mathscr{A}_{t+k}^\infty \rbrace
\end{equation*}
the \textit{strong mixing coefficient} of  $(Y_t)_{t \in \mathbb{Z}}$ at lag $k$. Then $(Y_t)_{t \in \mathbb{Z}}$ is called \textit{strongly mixing} if $\alpha (k) \rightarrow 0$ for $k \rightarrow \infty$.
\end{definition}

Combining strict stationarity with the strong mixing condition makes it possible to extend the classical central limit theorem from independent to dependent sequences of random variables. Several results and conditions concerning such an extension can be found in \cite[Ch. 7]{DurrettProb} together with proofs. We confine ourselves to the following theorem, which is a combination of \cite[Thm. 7.8]{DurrettProb} and the remark thereafter.

\begin{theorem}   \label{Thm:CLTforStationary}
Let $(Y_t)_{t \in \mathbb{Z}}$ be a stationary and strongly mixing sequence with mixing coefficient $\alpha$ and zero mean. Moreover, assume there is some $\delta >0$ such that $\mathbb{E} \vert Y_t \vert^{2 + \delta} < \infty$ and $\sum_{k=1}^{\infty} \alpha (k)^{\delta / (2 + \delta)} < \infty$ hold. Then we have
\begin{equation*}
\sigma^2_Y := \mathbb{E}Y_0 + 2 \sum_{t=1}^{\infty} \mathbb{E} Y_0 Y_t < \infty
\end{equation*}
and if $\sigma^2_Y > 0$ we get
\begin{equation*}
\frac{1}{\sqrt{n}} \sum_{t=1}^{n} Y_t \, \rightarrow^d \, \mathcal{N}(0, \sigma^2_Y )
\end{equation*}
for $n \rightarrow \infty$.
\end{theorem}

The conditions imposed on  $(\alpha (k)^{\delta / (2 + \delta)})_{k \in \mathbb{N}}$ and  $\mathbb{E} \vert Y_t \vert^{2 + \delta}$ imply in particular that the sequence $ (\mathbb{E}Y_0 Y_t)_{t \in \mathbb{N}}$ converges absolutely (see for instance \cite[Thm. 7.7]{DurrettProb}). Therefore, if $(Y_t)_{t \in \mathbb{Z}}$ satisfies the requirements of the theorem, $\sigma^2_Y$ is the spectral density $f_Y$ of $(Y_t)_{t \in \mathbb{Z}}$ at zero and can be estimated consistently using estimators for $f_Y(0)$. For a discussion on consistent estimation of $f_Y$, using lag windows as well as spectral windows, we refer to Shumway and Stoffer~\cite[Sec. 4.5]{ShumwayTimeSeries}. If we are in the situation of Theorem~\ref{Thm:CLTforStationary} and have a sequence of estimators $(\hat{\sigma}^2_n)_{n \in \mathbb{N}}$ which is consistent for $\sigma^2_Y$ (meaning $\hat{\sigma}^2_n \rightarrow \sigma^2_Y$ in probability), we apply Slutsky's theorem (see for instance van der Vaart~\cite[Lemma 2.8]{vanderVaart}) to obtain
\begin{equation}   \label{Eqn:DMTestAsymptotic}
\frac{1}{ \sqrt{ \hat{\sigma}^2_n n}} \sum_{t=1}^{n} Y_t \, \rightarrow^d \, \mathcal{N}(0, 1)
\end{equation}
for $n \rightarrow \infty$. This convergence allows for tests based on asymptotic normality, even for dependent sequences of random variables.\\
\par
Returning to the forecast comparison setting we let $(Y_t)_{t \in \mathbb{Z}}$ be a stationary sequence of random variables and $(\hat{x}_t)_{t \in \mathbb{Z}}$ and $(\hat{z}_t)_{t \in \mathbb{Z}}$ infinite sequences of forecasts. Again we assume that we sample at the time points $t=1, \ldots, n$. Using the time series $D_t := S(\hat{x}_t, Y_t) - S(\hat{z}_t, Y_t)$ for $t \in \mathbb{Z}$ and remembering (\ref{Eqn:MeanScoreDiff}), we define $\mathsf{K}_n := \frac{1}{n} \sum_{t=1}^{n} D_t$ in order to compare the two forecast sequences. If we think $(\hat{z}_t)$ are superior forecasts, we should test the null hypothesis $H_0:$ \textit{The forecasts $(\hat{x}_t)$ are at least as good as the forecasts $(\hat{z}_t)$}. Using scoring functions, this translates to the hypothesis $H_0 : \mathbb{E} \mathsf{K}_n \leq 0$. In order to test if $(\hat{x}_t)$ is superior we would naturally reverse the inequality and for equal accuracy we would use $H_0 : \mathbb{E} \mathsf{K}_n = 0$.

The final step made in \cite{DieboldMariano} is now to choose a sequence of estimators $( \hat{\sigma}_n^2)_{n \in \mathbb{N}}$ which consistently estimates $\sigma_D^2$ and to define the test statistic $ \sqrt{n /\hat{\sigma}_n^2 } \, \mathsf{K}_n$. Under the assumptions of Theorem~\ref{Thm:CLTforStationary}, the hypothesis $\mathbb{E}\mathsf{K}_n = 0$ and using the argument in (\ref{Eqn:DMTestAsymptotic}), this test statistic is approximately normal if $n$ is large enough. Therefore, one- or two-sided regions of rejection can be calculated using the quantiles of the standard normal distribution. Note that the assumption of stationarity as well as the condition on the strong mixing coefficient in Theorem~\ref{Thm:CLTforStationary} are only necessary for $(Y_t)_{t \in \mathbb{Z}}$ because both properties carry over to $(D_t)_{t \in \mathbb{Z}}$. The moment condition $\mathbb{E} \vert D_t \vert^{2 + \delta} < \infty$  has to be ensured by the choice of $S$.

\subsection{Remarks} 
\label{Sec:ForecastRankRemark}

The previous subsections considered forecast comparison for point and probabilistic forecasting and the role played by elicitability.  We conclude this section by also looking at problems or limitations which have to be taken into account when comparing forecasts. For brevity, we only sketch the arguments. \\
\par
\textbf{Applicability of Diebold-Mariano tests.} A DM test in the setting of the previous subsection has one important drawback: It remains unclear if Theorem~\ref{Thm:CLTforStationary} is still applicable if the forecasting sequences $(\hat{x}_t)_{t \in \mathbb{Z}}$ and $(\hat{z}_t)_{t \in \mathbb{Z}}$ are not deterministic, but random. In particular, the presented framework cannot handle a situation where the forecast for time $t$ is calculated using $Y_{t-1}, \ldots, Y_{t-m}$ for some $m \in \mathbb{N}$, since it is not clear how the distribution and dependence structure of $(D_t)_{t \in \mathbb{Z}}$ react. However, such an approach is quite common in practice. As stated in Giacomini and White~\cite{GiacominiWhite}, the classical DM test compares forecasting models, not \textit{forecasting procedures}. According to the definition of \cite{GiacominiWhite}, the latter includes the forecasting model together with an estimation method and a choice which part of the data is used. \\
An extension of the considered test to such a framework is done in \cite{GiacominiWhite}. The authors consider a setting where forecasts can be measurable functions of realizations and estimated parameters based on a finite and fixed time window in the past. They are able to show asymptotic normality of the test statistic $\mathsf{K}_n$ under similar conditions as in Theorem~\ref{Thm:CLTforStationary}, see~\cite[Thm. 4]{GiacominiWhite}. Moreover, they are able to drop the stationarity condition as long as the estimator $\hat{\sigma}_n^2$ is modified such that it is consistent for non-stationary processes. \\
\par
\textbf{Unconditional versus conditional forecast comparison.} Consider a test for equal forecast accuracy. The hypothesis in the DM test then states that the mean score differences have expected value equal to zero. This is equivalent to saying that the forecast accuracy is equal \textit{on average}. As pointed out in \cite{GiacominiWhite}, this might be inappropriate in some situations. If one forecast outperforms the other when $Y_{t-1}$ was small and vice versa when $Y_{t-1}$ was large, a classical DM test cannot detect this. In contrast, the conditional tests proposed in \cite{GiacominiWhite} take information up to time $t$ into account and are designed to detect if forecast performance is predictable. Due to this difference, the DM test considered here (which could be called unconditional test) should only guide a decision maker in selecting a forecast for an unspecified future date for which no information is available.\\
\par
\textbf{Scoring functions and ordering.} Even if the integrated score function $x \mapsto \bar{S}(x, F)$ is considered, the inequality $\bar{S}(x, F) < \bar{S}(z,F)$ does not imply that $x$ is closer (in Euclidean distance) to the true value $t=T(F)$, since consistency only requires $\bar{S}(\cdot, F)$ to take its minimum in $t$.  Apart from that, there can be local minima which cause the integrated score to be small for  arbitrary values $x \neq t$. Therefore, the concept of \textit{(strict) order sensitivity} is introduced by Lambert~\cite[Def. 2]{Lambert}, which requires $\bar{S}(\cdot, F)$ to be (strictly) decreasing for $x \leq t$ and (strictly) increasing for $x \geq t$. However, even this property does not exclude the problem mentioned at the beginning. If $\bar{S}(\cdot, F)$ is very steep for $x \leq t$ and rather flat for $x \geq t$, it is still possible that $\bar{S}(x, F) < \bar{S}(z,F)$ holds for $z < t < x$ and $\vert z- t \vert < \vert x - t\vert$. Therefore, the best solution to the ordering problem is probably to consider multiple scoring functions (see also the next paragraph). \\
\par
\textbf{Choice of scoring function.} Until now, the choice of the scoring function $S$ was not discussed, although there are often many possible choices. It is therefore a priori not clear which function one should use for forecast comparison. It could be the case, that forecaster 1 outperforms forecaster 2 for a scoring function $S_1$ but the opposite relation holds for another scoring function $S_2$. For quantiles and expectiles (see Definitions~\ref{Def:quantilefunction} and~\ref{Def:Expectiles}) a solution is proposed by Ehm et al.~\cite{EhmGneiting}. For these  functionals they show that all scoring functions which satisfy certain regularity conditions are convex combinations of a class of extremal functions. These extremal functions are simple and can be parametrized by $\theta \in \mathbb{R}$. As a consequence, they propose to plot the scores of the extremal functions for some interval in order to perform a graphical check if one forecaster dominates the other for all parameters. If one forecast is more accurate for all extremal functions, the same holds true for the whole class of scoring functions under consideration.

\section{Quantile and expectile regression} 

This section introduces quantiles and expectiles of distribution functions and shows how least squares regression can be generalized using these characteristics together with strictly consistent scoring functions. Moreover, both concepts are used to define measures of risk in the next chapter, see Subsections~\ref{Sec:ValueAtRisk} and~\ref{Sec:ExpectileVaR}.

\subsection{Quantile regression}  
\label{Sec:QuantileReg}

Quantiles are closely connected to distribution functions. For a value $\alpha \in (0,1)$, the $\alpha$-quantile of a random variable is the threshold which is exceeded with probability $1-\alpha$ or less. Hence, it can be interpreted as the inverse of the distribution function and in certain cases this is indeed true. For general distribution functions the definition is as follows.

\begin{definition}  \label{Def:quantilefunction}
For a univariate distribution function $F$ and $\alpha \in (0,1)$ define the functions
\begin{align*}
F^{\leftarrow}   : \, &(0,1) \rightarrow \mathbb{R}, \quad \alpha \mapsto \inf \lbrace x \in \mathbb{R} \mid F(x) \geq \alpha \rbrace , \\
F^{\rightarrow}  : \, &(0,1) \rightarrow \mathbb{R}, \quad \alpha \mapsto \inf \lbrace x \in \mathbb{R} \mid F(x)  > \alpha \rbrace
\end{align*}
and call the first \textit{lower} and the second \textit{upper quantile function} of $F$. Moreover, any element of the interval $[F^{\leftarrow}(\alpha), F^{\rightarrow} (\alpha) ]$ is called an $\alpha$\textit{-quantile} of $F$. If $F^{\leftarrow}(\alpha) = F^{\rightarrow} (\alpha) = x$ holds, $x$ is the \textit{unique} $\alpha$\textit{-quantile} of $F$.
\end{definition}

The main focus of this section lies on the lower quantile function. The upper quantile function is needed if it is convenient to use the set of quantiles $[F^{\leftarrow}(\alpha), F^{\rightarrow} (\alpha) ]$. If $F$ is strictly increasing, both functions coincide and every quantile is unique. The following lemma lists some well-known properties of quantile functions which are needed below.

\begin{lemma}   \label{Thm:QuantileLemma}
For $\alpha \in (0,1)$, $x \in \mathbb{R}$ and distribution functions $F$ and $G$ the following hold:
\begin{enumerate}[label=(\roman*)]
	\item $ F^{\leftarrow} (\alpha) \leq x \, \Leftrightarrow \, \alpha \leq F(x)$.
	\item $ F( F^{\leftarrow} (\alpha) ) \geq \alpha$ with equality if $F$ is continuous.
	\item $ F^{\leftarrow} (F(x)) \leq x$ with equality if $F$ is strictly increasing.
	\item $F^{\leftarrow}$ is left-continuous and $F^{\rightarrow}$ is right-continuous.
	\item If $X =^d F$ and $g(X) =^d G$ for an increasing left-continuous function $g$, then $G^{\leftarrow}(\alpha) =  g( F^{\leftarrow} (\alpha))$.
\end{enumerate}
\end{lemma}

\begin{proof}
The first three properties are checked straightforward, see for instance van der Vaart~\cite[Lemma 21.1]{vanderVaart}. For (iv) and (v) see McNeil et al.~\cite[Prop. A.3, A.5]{McNeilRisk}.
\end{proof}

\begin{remark}
Looking at the properties (ii) and (iii) of Lemma~\ref{Thm:QuantileLemma}, we see that $F^{\leftarrow}$ and $F^{\rightarrow}$ are the usual inverse function of $F$ if $F$ is continuous and strictly increasing everywhere.
\end{remark}

In Chapter~\ref{Chapter1} it is shown that expectations and ratios of expectations are elicitable. The next step is to prove the same for quantiles. In contrast to previous examples, we cannot represent the $\alpha$-quantile $q$ of $F$ via an expectation. However, if $F$ is continuous and $Y$ is a random variable such that $Y =^d F$ under $\mathbb{P}$, the equation
\begin{equation*}
\mathbb{E}  \mathbbm{1}_{\lbrace Y \leq x \rbrace} =  F(x) = \alpha
\end{equation*}
is solved by any $[F^{\leftarrow}(\alpha), F^{\rightarrow} (\alpha) ]$. For classes $\mathcal{F}$ of continuous distributions having unique $\alpha$-quantiles this fact can be used to construct the oriented strict $\mathcal{F}$-identification function $V(x,y) = \mathbbm{1}_{\lbrace y \leq x \rbrace} - \alpha$. Continuity is vital for this statement because the next example shows that for a convex class $\mathcal{F}$ which contains a continuous distribution and a specific Dirac measure there cannot exist a strict $\mathcal{F}$-identification function.

\begin{example}
Let $\alpha \in (0,1)$ and a convex $\mathcal{F}$ be given and let $V$ be a strict $\mathcal{F}$-identification function for $T(F):=F^{\leftarrow}(\alpha)$. Assume that there exists an $F \in \mathcal{F}$ which is continuous in $F^{\leftarrow}(\alpha)$. Furthermore, suppose there is an $\varepsilon >0$ such that $q:= F^{\leftarrow}(\alpha + \varepsilon)$ is well defined and $\delta_q \in \mathcal{F}$. We define $F_\lambda := \lambda F + (1-\lambda) \delta_q$ with $\lambda < \alpha /( \alpha + \varepsilon)$ and observe that 
\begin{equation*}
F_\lambda (q-) = \lambda F(q) < \frac{\alpha}{\alpha + \varepsilon} (\alpha + \varepsilon) = \alpha \quad \text{and} \quad F_\lambda(q) = \lambda (\alpha + \varepsilon) + (1-\lambda) > \alpha
\end{equation*}
hold, hence $F^{\leftarrow}_\lambda (\alpha) = q$. This leads to 
\begin{equation*}
0 = \bar{V}(q, F_\lambda) = \lambda \bar{V}(q,F) + (1-\lambda) V(q,q) = \lambda \bar{V}(q,F),
\end{equation*}
since all quantiles of $\delta_q$ are $q$, implying $V(q,q) = \bar{V}(q, \delta_q) = 0$. Because $V$ is a strict identification function, we obtain $q= F^{\leftarrow}(\alpha) <F^{\leftarrow}(\alpha + \varepsilon)$, a contradiction.
\end{example}

Although strict identification functions may not exist for certain classes $\mathcal{F}$, it is possible to construct a strictly consistent scoring function, as the following theorem shows. It is well-known and similar results are stated in Gneiting~\cite{GneitingQuantiles,GneitingPoints} without proofs, hence we add our own proof. Note that the result does not need $F$ to be continuous, but requires that the $\alpha$-quantile is unique instead.

\begin{theorem}  \label{Thm:QuantilesElicitable}
Fix $\alpha \in (0,1)$ and let $\mathcal{F}$ be a class of distribution functions having unique $\alpha$-quantiles. Then the functional $T : \mathcal{F} \rightarrow \mathbb{R}$, $F \mapsto F^{\leftarrow} (\alpha)$ is elicitable with respect to $\mathcal{F}$. An $\mathcal{F}$-consistent scoring function for $T$ is given by
\begin{equation*}
S(x,y) := (\mathbbm{1}_{\lbrace y \leq x \rbrace} - \alpha ) (g(x) - g(y) )
\end{equation*}
for any $\mathcal{F}$-integrable increasing function $g$. If $g$ is strictly increasing, then $S$ is strictly consistent.
\end{theorem}

\begin{remark}
Under additional assumptions it is also possible to show that all strictly consistent scoring functions for the $\alpha$-quantile can be represented as in Theorem~\ref{Thm:QuantilesElicitable}, see also Remark~\ref{Rem:MeanNecessaryCond}. Again we refer to \cite{GneitingQuantiles} for details.
\end{remark}

\begin{proof} 
Fix $\alpha \in (0,1)$, $F \in \mathcal{F}$ and let $Y =^d F$ under $\mathbb{P}$. Moreover, let $g$ be an increasing and $\mathcal{F}$-integrable function (this is always possible, for example by choosing a bounded $g$) and set $t:= T(F) = F^{\leftarrow} (\alpha)$. Firstly, define the left-hand limit $F(t-):=\lim_{s \nearrow t} F(s)$ and observe that $F(t-) \leq \alpha$ due to the definition of $t$. Consequently, the inequality
\begin{equation}  \label{Eqn:DistFunctionInequality}
\alpha + \mathbb{P}(Y = t) - F(t) = \alpha - \mathbb{P}(Y <t) = \alpha - F(t-) \geq 0
\end{equation}
holds. Now fix any $x \in \mathbb{R}$ such that $x <t$ and calculate
\begin{align}
\bar{S}(x,F) - \bar{S}(t,F) &= (F(x) -\alpha) g(x) - (F(t) - \alpha)g(t) + \mathbb{E} (\mathbbm{1}_{\lbrace Y \leq t \rbrace} - \mathbbm{1}_{\lbrace Y \leq x \rbrace} ) g(Y) \nonumber \\
&= (F(x) -\alpha) g(x) - (F(t) - \alpha)g(t) \nonumber \\
 &\quad+ \mathbb{E}\mathbbm{1}_{\lbrace x < Y <t \rbrace} g(Y) + \mathbb{P}(Y = t) g(t) \nonumber \\
&\geq (F(x) - \alpha + \mathbb{P}(x < Y < t)) g(x) \label{Eqn:QuantIneq1} \\
&\quad - (F(t) - \alpha - \mathbb{P}(Y= t) ) g(t) \nonumber \\
&= (\alpha - F(t) + \mathbb{P}(Y=t) ) (g(t) - g(x) ) \geq 0 \label{Eqn:QuantIneq2}
\end{align}
using Inequality~(\ref{Eqn:DistFunctionInequality}) and the monotonicity of $g$. If we have $x > t$, then 
\begingroup
\allowdisplaybreaks[0]
\begin{align}
\bar{S}(x,F) - \bar{S}(t,F) &= (F(x) -\alpha) g(x) - (F(t) - \alpha)g(t) - \mathbb{E} \mathbbm{1}_{\lbrace t < Y \leq x \rbrace}  g(Y) \nonumber \\
&\geq (F(x) -\alpha) g(x) - (F(t) - \alpha)g(t) - (F(x) - F(t)) g(x) \label{Eqn:QuantIneq3}\\
&=(F(t) - \alpha) (g(x) - g(t) ) \geq 0,  \nonumber
\end{align}
\endgroup
so $S$ is $\mathcal{F}$-consistent. Now we check if the inequalities are strict as soon as $g$ is strictly increasing. We start with the second display and assume that equality holds for every line. Then we have $F(t) = \alpha$ as well as $\mathbb{P}(t < Y < x) = 0$ due to the strict monotonicity of $g$ and the fact that (\ref{Eqn:QuantIneq3}) is an equality. This implies that for $z:= (x+t)/ 2$ we have
\begin{equation*}
\alpha = F(t) = \mathbb{P}( Y \leq t) + \mathbb{P}( t < Y \leq z) = F(z) ,
\end{equation*}
which is a contradiction to the uniqueness of the $\alpha$-quantile. We now turn to the first display and distinguish two possible cases for the limit $F(t-)$. Considering $\alpha > F(t-)$, we have a strict inequality in (\ref{Eqn:DistFunctionInequality}) and together with the strict monotonicity of $g$ the inequality in (\ref{Eqn:QuantIneq2}) is strict. If we suppose $\alpha = F(t-)$ and equality in (\ref{Eqn:QuantIneq1}), the latter implies $\mathbb{P}(x < Y <t) = 0$. As a consequence we have for $z:= (x+t)/ 2$
\begin{equation*}
\alpha = \mathbb{P}(Y <t) = \mathbb{P}(Y \leq z ) + \mathbb{P}( z < Y < t) = F(z) ,
\end{equation*}
which is again a contradiction to $F^{\leftarrow}(\alpha) = t > z$. Since we discussed all cases, $S$ is a strictly $\mathcal{F}$-consistent scoring function for $T$ if $g$ is strictly increasing.
\end{proof}

The following example shows why the uniqueness of the $\alpha$-quantile is essential for Theorem~\ref{Thm:QuantilesElicitable} to hold. Moreover, it shows a way to drop this requirement by slightly modifying the definition of elicitability.

\begin{example}  \label{Thm:ExQuantilesElicitable}
Define a distribution function $G$ on the interval $[- \frac{1}{2}, 1]$ via its density $\phi_G$, which is given by
\begin{equation*}
\phi_G (x) :=  \mathbbm{1}_{ [- \frac{1}{2} ,0) } (x) +  \mathbbm{1}_{ [\frac{1}{2}, 1] } (x).
\end{equation*}
Loosely speaking, $G$ represents a uniform distribution having a gap in the middle. For $\alpha = \frac{1}{2}$, any $ q \in [0, \frac{1}{2}]$ satisfies $G (q) = \frac{1}{2}$ and is thus an $\frac{1}{2}$-quantile. Choosing a strictly increasing $g$ and defining the scoring function $S$ as in Theorem~\ref{Thm:QuantilesElicitable}, the identity $\bar{S}(0,G) = \bar{S}(q,G)$ holds for any $q \in [0, \frac{1}{2}]$. Hence, $S$ is only $\mathcal{F}$-consistent, but not \textit{strictly} $\mathcal{F}$-consistent if $G \in \mathcal{F}$. \\
A possible solution to this problem is the modification of the functional and the scoring functions such that both are set-valued. This modification is used and discussed in Gneiting~\cite{GneitingPoints} and Fissler and Ziegel~\cite[Remark 2.3]{FissZieg}, among others. To be more precise, fix $\alpha \in (0,1)$ and define the quantile functional for the $\alpha$-quantile via
\begin{equation*}
T: \mathcal{F} \rightarrow \mathscr{P}(\mathbb{R}) , \quad F \mapsto [ F^{\leftarrow} (\alpha) , F^{\rightarrow} (\alpha) ] 
\end{equation*}
such that the functional now maps $F$ to the whole set of $\alpha$-quantiles. A function $S$ is now considered strictly $\mathcal{F}$-consistent if for any $t \in T(F)$, $x \in \mathbb{R}$, we have $\bar{S}(x,F) \geq \bar{S}(t,F)$ and the equality $\bar{S}(x,F) = \bar{S}(t,F)$ implies $x \in T(F)$. Using this definition, the scoring function $S$ of Theorem~\ref{Thm:QuantilesElicitable} is again strictly $\mathcal{F}$-consistent if $g$ is strictly increasing. In order to show this, we fix $\alpha \in (0,1)$, choose any $t \in T(F) = [ F^{\leftarrow} (\alpha) , F^{\rightarrow} (\alpha) ]$ and $x \notin T(F)$ and perform the same calculations for $\bar{S}(x,F) - \bar{S}(t,F) $ as in the proof of Theorem~\ref{Thm:QuantilesElicitable}. Since (\ref{Eqn:DistFunctionInequality}) continues to hold for any $t \in T(F)$, we obtain consistency. To see that even strict consistency holds, we inspect the occurring inequalities again: We assume $t < x$ and suppose that equality holds in the second display, which in particular implies $\mathbb{P}(t < Y <x) = 0$ due to (\ref{Eqn:QuantIneq3}). Hence, for all $\varepsilon>0$ satisfying $\varepsilon < x-t$ we have $F(x-\varepsilon) = F(t) = \alpha$, which is a contradiction to $x \notin T(F)$. Now assume that $t > x$ holds and note that the inequality $\alpha > F(t-)$ can only occur if $t= F^{\leftarrow}(\alpha)$, so in this case the argument is the same as above and Inequality~(\ref{Eqn:QuantIneq2}) is strict. Therefore, assume $\alpha = F(t-)$ and observe that equality in (\ref{Eqn:QuantIneq1}) implies $\mathbb{P}(x < Y < t) = 0$. Consequently, $F(t-\varepsilon) = F(t-) = \alpha$ holds for all $\varepsilon>0$ satisfying $\varepsilon < t-x$, which is again a contradiction to $x \notin T(F)$.
\end{example}

Having established elicitability of quantiles, we now consider an application of this property in regression. We begin with considering linear regression. To this end, let $Y$ be a real-valued and $X$ an $\mathbb{R}^p$-valued random variable. The basic concept of linear regression is based on modelling the mean of the random variable $Y \mid X = x$ as a linear function in $x$, that is $ \mathbb{E}(Y \mid X = x) = x^\top \beta$ for some $\beta \in \mathbb{R}^p$ which has to be estimated. To compute this based on data sets $y \in \mathbb{R}^k$ and $x \in \mathbb{R}^{k \times p}$, the optimization problem
\begin{equation}  \label{Eqn:RegressionMean}
\underset{\beta}{\min} \, \sum_{i=1}^{k} (y_i - x_i^\top \beta )^2
\end{equation}
is solved for an optimal $\beta^*$. Using the theory of Chapter~\ref{Chapter1}, we see why it makes sense to do so. As shown in Example~\ref{Thm:ExMeanIsElicitable}, the function $S(x,y) := (y-x)^2$ is a strictly $\mathcal{F}$-consistent scoring function if the distributions in $\mathcal{F}$ have finite second moments. Consequently, Equation~(\ref{Eqn:RegressionMean}) states that $\beta$ is selected by minimizing the expected score using a scoring function which is strictly consistent for the mean. It is therefore natural to use this minimization problem and choose the scoring function according to the functional we want to apply to the distribution $Y \mid X = x$.
\par
One famous example of this approach is \textit{quantile regression}, which is introduced in Koenker and Bassett~\cite{KoenkerBassett} and compiled in Koenker~\cite{KoenkerQuantRegression}. To illustrate the idea, choose $\alpha \in (0,1)$ and a strictly consistent scoring function $S$ in order to model the $\alpha$-quantile of $ Y \mid X =x$ via $x^\top \beta$. Using Theorem~\ref{Thm:QuantilesElicitable}, the minimization of the empirical score can be stated as
\begin{equation}  \label{Eqn:RegressionQuant}
\underset{\beta}{\min} \, \sum_{i=1}^{k} (\mathbbm{1}_{\lbrace y_i \leq x_i^\top \beta \rbrace} - \alpha) (g(x_i^\top \beta) - g(y_i) )
\end{equation}
for a strictly increasing function $g$. Choosing $\alpha = \frac{1}{2}$ replaces mean regression by median regression, which is a rather old idea to decrease the sensitivity to outliers (see \cite[Sec. 1.2]{KoenkerQuantRegression}).\\
\par
The following paragraph discusses some basics of quantile regression. The first problem which lies in replacing (\ref{Eqn:RegressionMean}) by (\ref{Eqn:RegressionQuant}) is the method of solving the minimization problem. The scoring functions for quantiles are more complicated than quadratic functions and in particular not differentiable. This problem is tackled in \cite[Sec. 1.3]{KoenkerQuantRegression} by choosing $g(x) = x$, which leads to the `pinball' scoring function. The minimization problem for linear quantile regression becomes
\begin{equation*}
\underset{\beta}{\min} \, (1-\alpha) \underset{ \lbrace y_i \leq x_i^\top \beta \rbrace }{\sum} \vert y_i - x_i^\top \beta \vert + \alpha \underset{ \lbrace y_i > x_i^\top \beta \rbrace }{\sum} \vert y_i - x_i^\top \beta \vert ,
\end{equation*}
which can be efficiently solved using linear programming techniques. After the calculation of $\beta^*$ becomes computationally cheap, quantile regression can be used to get a more detailed impression of a dataset. For example,  regression coefficients and plots of regression lines (or curves) can be inspected for different choices of $\alpha$. It can then be compared how the center as well as the lower and upper tail of the distribution $Y \mid X = x$ vary. Another nice property of quantile regression is equivariance with respect to an increasing bijection $h$. This means that if the data $y \in \mathbb{R}^k$ is transformed using $h$, the estimator $\beta^*$ is transformed accordingly. For instance, if quantile regression is performed and $x_i^\top \hat{\beta}$ is interpreted as the $\alpha$-quantile of $h(Y) \mid X = x_i$, Lemma~\ref{Thm:QuantileLemma} (v) states that it is reasonable to interpret $h^{-1}(x_i^\top \hat{\beta})$ as the $\alpha$-quantile of $Y \mid X = x_i$. For details we refer to \cite[Sec. 2.2.3]{KoenkerQuantRegression}.

\subsection{Expectile regression}  
\label{Sec:ExpectileRegression}

Expectiles can be understood as generalizations of quantiles, but simultaneously the mean can also be obtained as a special case. Their definition as well as their name are due to Newey and Powell~\cite{NeweyPowell}, who introduce them in order to define an asymmetric version of the least squares estimation method for linear regression. Using the asymmetric least squares coefficients, they test the error distribution for symmetry and homoscedasticity. Apart from this application, expectiles are also used in risk management, see Chapter~\ref{Chapter3}. 

\begin{definition}  \label{Def:Expectiles}
Let $F$ be a distribution function having a finite first moment. Then for $\tau \in (0,1)$ an $x \in \mathbb{R}$ satisfying
\begin{equation}  \label{Eqn:ExpectileDefinition}
\tau \int_{x}^{\infty} y-x \, \mathrm{d} F(y) = (1-\tau) \int_{-\infty}^{x} x-y \, \mathrm{d}F(y)
\end{equation}
is called a $\tau$\textit{-expectile} of $F$ and denoted via $e_\tau (F)$.
\end{definition}

\begin{lemma}  \label{Thm:ExpectileExists}
For $\tau \in (0,1)$ and $F$ having finite expectation there is exactly one $\tau$-expectile.
\end{lemma}

\begin{proof}
This proof follows \cite[Thm. 1]{NeweyPowell}. Fix $\tau \in (0,1)$, $F$ having finite first moments and let $Y$ be a random variable on $(\Omega, \mathscr{A}, \mathbb{P})$ such that $Y =^d F$. Define the functions $G_1, G_2 : \mathbb{R} \rightarrow \mathbb{R}_+$ via
\begin{equation*}
G_1(x) := \mathbb{E} \mathbbm{1}_{ (x, \infty) }(Y) (Y-x)  \quad \text{and} \quad G_2(x) := \mathbb{E} \mathbbm{1}_{ (-\infty, x) }(Y) (x-Y)
\end{equation*}
in order to represent the left- and right-hand side of Equation~(\ref{Eqn:ExpectileDefinition}). Now fix $x \in \mathbb{R}$ and let $(x_n)_{n \in \mathbb{N}}$ be a sequence converging to $x$. Then there are $z_1, z_2 \in \mathbb{R}$ such that $ Y-x_n \leq Y-z_1$ and $x_n - Y \leq z_2 -Y $ hold for all $n \in \mathbb{N}$. Using this and the fact that $\mathbb{E}\vert Y \vert < \infty$ we apply dominated convergence and obtain that both $G_1$ and $G_2$ are continuous. Moreover, for $a < b$ we have the inequality
\begin{equation*}
G_1(a) \geq \mathbb{E} \mathbbm{1}_{(b, \infty)}(Y) (Y-a) \geq \mathbb{E} \mathbbm{1}_{(b, \infty)} (Y)(Y-b) = G_1(b)
\end{equation*}
so $G_1$ is decreasing. Similarly, we have $G_2(a) \leq G_2(b)$ hence $G_2$ is increasing. The inequalities for $G_1$ or $G_2$ are strict if $\mathbb{P}(Y \in (b, \infty))$ or $\mathbb{P}(Y \in (-\infty, a))$ are strictly positive, respectively. Moreover, we use $\mathbb{E}\vert Y \vert < \infty$ to compute the limits $\lim_{x \rightarrow -\infty} G_1(x) = \infty$ and $\lim_{x \rightarrow \infty} G_2(x) = \infty$. Similarly, dominated convergence and a positive sequence $(x_n)_{n \in \mathbb{N}}$ with $x_n \rightarrow \infty$ can be used to show
\begin{equation*}
0 \leq \underset{x \rightarrow \infty}{\lim} G_1(x) = \underset{n \rightarrow \infty}{\lim} \mathbb{E} \mathbbm{1}_{ (x_n, \infty) }(Y) (Y-x_n)  \leq \underset{n \rightarrow \infty}{\lim} \mathbb{E} \mathbbm{1}_{(x_n, \infty)} (Y) Y = 0.
\end{equation*}
In the same way $\lim_{x \rightarrow -\infty} G_2(x) = 0$ is proved. All in all, strict monotonicity and the intermediate value theorem give a unique  $x$ in the interior of the support of $Y$ such that $\tau G_1(x) = (1-\tau) G_2(x)$. This value is the unique $\tau$-expectile of $F$.
\end{proof}

In the following, if a random variable $X$ has distribution function $F$, we also denote the $\tau$-expectile via $e_\tau(X)$. The term `expectile' highlights the fact that expectiles share properties of expectations as well as quantiles. The latter becomes obvious when considering a continuous distribution function $F$. In this case, the $\alpha$-quantile $q$ satisfies the equation $\alpha (1- F(q)) = (1-\alpha) F(q)$, which is similar to the expectile Identity~(\ref{Eqn:ExpectileDefinition}). Moreover, expectiles and quantiles share the three properties stated in the next lemma. To prove them for quantiles, we use part (v) of Lemma~\ref{Thm:QuantileLemma} for the first property and Lemma~\ref{Thm:AppendixQuantileLR} for the third one. The second follows from the definition of the quantile. The properties and their proof can also be found in \cite{NeweyPowell}.

\begin{lemma}  \label{Thm:ExpectileProperties}
Let $Y =^d F$ be a random variable with finite first moment. Then the following hold:
\begin{enumerate}[label=(\roman*)]
\item For $s \in \mathbb{R}_+$, $t \in \mathbb{R}$ and $\tilde{Y} := sY + t$ we have $e_\tau (\tilde{Y}) = s e_\tau (Y) + t$.
\item For $\tau_1 \leq \tau_2$ we have $e_{\tau_1} (Y) \leq e_{\tau_2} (Y)$.
\item $e_\tau (-Y) = - e_{1-\tau} (Y)$.
\end{enumerate}
\end{lemma}

\begin{proof}
To see (i), observe that $\tilde{Y} < e_\tau(\tilde{Y})  \Leftrightarrow Y < e_\tau(Y)$ since $s$ is positive and $\tilde{Y} - e_\tau(\tilde{Y}) = s (Y - e_\tau(\tilde{Y}))$. For (ii) consider the expectile Identity~(\ref{Eqn:ExpectileDefinition}) and observe that the left-hand side is increasing in $\tau$ while the right-hand side is decreasing. The opposite monotonicity properties hold in $x$, as shown in the proof of Lemma~\ref{Thm:ExpectileExists}, and thus $e_{\tau_1}(Y) \leq e_{\tau_2}(Y)$ holds for $\tau_1 \leq \tau_2$. In order to show (iii), take Identity~(\ref{Eqn:ExpectileDefinition}) for $-Y$ and plug in $x=e_{1-\tau} (Y)$. This gives
\begin{align*}
\tau \mathbb{E}( -Y + e_{1-\tau} (Y) )^+ &= \tau \mathbb{E} (e_{1-\tau} (Y) - Y )^+ \\
&= (1- \tau) \mathbb{E}(Y - e_{1-\tau} (Y) )^+ = (1-\tau) \mathbb{E}( - e_{1-\tau} (Y) - (-Y) )^+ 
\end{align*}
and because expectiles are unique, the relation $e_\tau (-Y) = - e_{1-\tau} (Y)$ follows.  
\end{proof}

To look at the previous discussion in a different light, recall that the median is the $\frac{1}{2}$-quantile and is thus the quantile which represents some notion of a `center' of the distribution. Similarly, the mean is the $\frac{1}{2}$-expectile and describes a different notion of center. We thus say that the expectiles generalize the mean in the same way as the quantiles generalize the median. As a consequence, we use similar techniques to elicit expectiles as we use for quantiles and expectations in Theorem~\ref{Thm:QuantilesElicitable} and Example~\ref{Thm:ExMeanIsElicitable}. As before, the derivation of an identification function is straightforward, see also Gneiting~\cite[Table 9]{GneitingPoints}.

\begin{lemma}  \label{Thm:ExpectileIdentifiable}
Fix $\tau \in (0,1)$ and let $\mathcal{F}$ be a class of distribution functions having finite first moments. Then an oriented strict $\mathcal{F}$-identification function for the functional $T: \mathcal{F} \rightarrow \mathbb{R}$, $F \mapsto e_\tau (F)$ is given by $V(x,y) := \vert \mathbbm{1}_{\lbrace y \leq x \rbrace} - \tau \vert (x-y)$.
\end{lemma}

\begin{proof}
For $\tau \in (0,1)$ and $F \in \mathcal{F}$ given, fix a random variable $Y =^d F$. Rearranging Equation~(\ref{Eqn:ExpectileDefinition}) gives
\begin{align}  \label{Eqn:ExpectileIdentification}
\tau \mathbb{E} \mathbbm{1}_{\lbrace Y > x \rbrace} (Y - x) &= (1 - \tau) \mathbb{E} \mathbbm{1}_{\lbrace Y < x \rbrace} (x-Y) \\
\Leftrightarrow \qquad \mathbb{E} \vert \mathbbm{1}_{\lbrace Y \leq x \rbrace} - \tau \vert	(x-Y) &= 0 , \nonumber
\end{align}
showing that $V$ is a strict $\mathcal{F}$-identification function. To prove its orientation, recall from the proof of Lemma~\ref{Thm:ExpectileExists} that the left-hand side of (\ref{Eqn:ExpectileDefinition}) is decreasing while the right-hand side is increasing in $x$. Hence, $x > e_\tau (F)$ implies that ``$<$'' replaces ``$=$'' in Equation~(\ref{Eqn:ExpectileIdentification}), which shows $\bar{V}(x,F) > 0$. For $x < e_\tau(F)$ the same arguments give $\bar{V}(x,F) < 0$, showing that $V$ is oriented.
\end{proof}

Inspired by the previous results, the first approach for a strictly $\mathcal{F}$-consistent scoring function is
\begin{equation*}
S(x,y) := \vert \mathbbm{1}_{\lbrace y \leq x \rbrace} - \tau \vert (y-x)^2
\end{equation*}
and $S$ is indeed strictly consistent. However, it requires all distributions in $\mathcal{F}$ to have finite second moments. This requirement can be relaxed in the same way as in Example~\ref{Thm:ExMeanIsElicitable}, as shown in the following theorem. The result and its proof are part of \cite[Thm. 10]{GneitingPoints}.

\begin{theorem}  \label{Thm:ExpectileElicitable}
Fix $\tau \in (0,1)$ and let $\mathcal{F}$ be a class of distribution functions having finite first moments. Then the functional $T: \mathcal{F} \rightarrow \mathbb{R}$, $F \mapsto e_\tau (F)$ is elicitable and an $\mathcal{F}$-consistent scoring function is given by
\begin{equation*}
S(x,y) = \vert \mathbbm{1}_{\lbrace y \leq x \rbrace} - \tau \vert (f(y) - f(x) - f'(x) (y-x) ),
\end{equation*}
where $f$ is a convex and $\mathcal{F}$-integrable function. If $f$ is strictly convex, $S$ is strictly consistent.
\end{theorem}

\begin{remark}
Under additional assumptions it is also possible to show that all strictly consistent scoring functions for the $\tau$-expectile can be represented as in Theorem~\ref{Thm:ExpectileElicitable}, see also Remark~\ref{Rem:MeanNecessaryCond}. For details we refer to \cite[Thm. 10]{GneitingPoints}.
\end{remark}

\begin{proof}
For $\tau \in (0,1)$ and $F \in \mathcal{F}$ given, let $Y$ be a random variable on $(\Omega, \mathscr{A}, \mathbb{P})$ with $Y =^d F$ under $\mathbb{P}$. Moreover, set $t:= e_\tau (F)$, take $x \in \mathbb{R}$ and choose an $\mathcal{F}$-integrable convex function $f$. The latter can always be done since the first moment exists for all $F \in \mathcal{F}$. Now consider the case $x < t$ and define the sets $A:= \lbrace Y \in (-\infty, x] \rbrace$, $B:= \lbrace Y \in (x,t] \rbrace$ and $C:= \lbrace Y \in (t, \infty) \rbrace$. The difference $\bar{S}(x,F) - \bar{S}(t, F)$ can then be split up into three parts by using $A \uplus B \uplus C = \Omega$. Note that Equation~(\ref{Eqn:ExpectileDefinition}) immediately implies
\begin{equation}  \label{Eqn:ExpectileDefABC}
\tau \mathbb{E} \mathbbm{1}_C (Y -t) = (1-\tau) \mathbb{E} \mathbbm{1}_{A \cup B} (t -Y) .
\end{equation}
Additionally, define the function
\begin{equation*}
g(y,x) := f(y) - f(x) - f'(x) (y-x),
\end{equation*}
which is nonnegative due to the convexity of $f$. Moreover, we obtain the equality
\begin{equation}  \label{Eqn:ExpectilegTransform}
g(y,x) - g(y,t) = g(t,x) + (f'(t) - f'(x)) (y-t)
\end{equation}
for any $x,y,t \in \mathbb{R}$. Using these preparations, we now calculate 
\begin{align*}
\bar{S}(x,F) - \bar{S}(t,F) &= \mathbb{E} \vert \mathbbm{1}_{\lbrace Y \leq x \rbrace} - \tau \vert g(Y,x) - \mathbb{E} \vert \mathbbm{1}_{\lbrace Y \leq t \rbrace} - \tau \vert g(Y,t) \\
&= (1- \tau) \mathbb{E}( g(Y,x) - g(Y,t) ) \mathbbm{1}_A + \tau \mathbb{E} (g(Y,x) - g(Y,t) ) \mathbbm{1}_C \\
	&\quad + \mathbb{E} (\tau g(Y,x) - (1-\tau) g(Y,t) ) \mathbbm{1}_B \\
&= (1-\tau) \left[ g(t,x) \mathbb{P}(A) + (f'(t) - f'(x)) \mathbb{E}(Y-t) \mathbbm{1}_A \right] \\
	&\quad + \tau \left[ g(t,x) \mathbb{P}(C) + (f'(t) - f'(x)) \mathbb{E}(Y-t) \mathbbm{1}_C \right] \\
	&\quad + \mathbb{E}[ \tau g(Y,x) - (1-\tau)(f(Y) - f(t) - f'(x)(Y-t)) ] \mathbbm{1}_B \\
	&\quad +  (1-\tau) (f'(t) - f'(x)) \mathbb{E}(Y-t)\mathbbm{1}_B  ,
\end{align*}
where (\ref{Eqn:ExpectilegTransform}) is used in the last step. Due to Identity~(\ref{Eqn:ExpectileDefABC}), all expectations which are multiplied by $(f'(t) - f'(x))$ vanish. Moreover, $x <Y \leq t$ holds on $B$ and thus $f'(x) (Y-t) \geq f'(Y) (Y-t)$ follows from the convexity of $f$. All in all, this implies 
\begin{align*}
\bar{S}(x,F) - \bar{S}(t,F) &= (1-\tau) g(t,x) \mathbb{P}(A) + \tau g(t,x) \mathbb{P}(C) \\
&\quad + \mathbb{E}[ \tau g(Y,x) + (1-\tau)(f(t) - f(Y) + f'(x)(Y-t)) ] \mathbbm{1}_B \\
&\geq g(t,x) [ (1-\tau) \mathbb{P}(A) + \tau \mathbb{P}(C) ] \\ 
&\quad + \mathbb{E}( \tau g(Y,x) + (1-\tau) g(t,Y) ) \mathbbm{1}_B \geq 0,
\end{align*}
which shows $\mathcal{F}$-consistency in the case $t< x$. If we have $t > x$ instead, we proceed as follows. Firstly, switch the roles of $t$ and $x$ in the definitions of $A$, $B$ and $C$ and do the same calculations as above. Since then $t <Y \leq x$ on $B$ and because $f'(x) (Y-t) \geq f'(Y) (Y-t)$ still holds, the case $t>x$ is also done. Whenever $f$ is chosen strictly convex, $f'$ is strictly increasing implying $g(y,x) > 0$ for $y \neq x$ and $f'(x) (Y-t) > f'(Y) (Y-t)$ on $B$. Consequently, at least one of the inequalities in our calculations needs to be strict because not all three sets $A$, $B$ and $C$ can have probability zero. Hence, strict convexity of $f$ implies strict consistency of $S$.
\end{proof}

Similar to least squares or quantile regression, we can now use any $\mathcal{F}$-consistent scoring functions $S$ which has the form given in Theorem~\ref{Thm:ExpectileElicitable} to perform \textit{expectile regression}. This is done by Newey and Powell~\cite{NeweyPowell} who follow Koenker and Bassett~\cite{KoenkerBassett} and chose the strictly convex function $f(x) = x^2$. As discussed in the previous subsection, they model the $\tau$-expectile of $Y \mid X = x$ as a linear function $x^\top \beta$ and solve the minimization problem
\begin{equation*}
\underset{\beta}{\min} \, (1-\tau) \underset{ \lbrace y_i \leq x_i^\top \beta \rbrace }{\sum} (y_i - x_i^\top \beta)^2 + \tau \underset{ \lbrace y_i > x_i^\top \beta \rbrace }{\sum}  (y_i - x_i^\top \beta)^2
\end{equation*}
for given data $y \in \mathbb{R}^k$ and $x \in \mathbb{R}^{k \times p}$. Although \cite{NeweyPowell} use expectile regression to examine the errors of an ordinary least squares regression, it can also be used to get an impression of the data. Similar to quantile regression, expectile regressions can be performed and plotted for several values of $\tau$ to see how the center and the tails behave. Moreover, expectiles have also an equivariance property, but only for affine transformations, as shown in part (i) of Lemma~\ref{Thm:ExpectileProperties}.
\par 
It is sometimes remarked that $\tau$-expectiles for $\tau \neq \frac{1}{2}$ are not easy to interpret. One interpretation presented by Ehm et al.~\cite{EhmGneiting} is the following: In a situation with a tax rate for gains and a deduction rate for losses, one should invest an amount $\theta$ in a start up company having payoff distribution $F$ only if a certain expectile of $F$ exceeds $\theta$. Another interpretation can be found in a risk measurement context, see Subsection~\ref{Sec:ExpectileVaR}.

\section{M-estimation} 

One of the most important parametric methods in estimation is maximum likelihood where the estimate is chosen such that it maximizes the likelihood function. If the likelihood function is differentiable, the parameter for which the derivative of the likelihood function vanishes is selected. One natural extension is to use functions which differ from the likelihood function, but lead to estimators having better statistical properties. This approach is called M-estimation and this section presents some results and examples which can be found in Huber and Ronchetti~\cite{HuberRonchetti} as well as van der Vaart~\cite{vanderVaart}. In the following, let $Y_1, \ldots, Y_n$ be random variables on $(\Omega, \mathscr{A}, \mathbb{P})$ taking values in $\mathsf{O} \subset \mathbb{R}^d$ and let $\Theta \subset \mathbb{R}^k$ be the parameter space.

\begin{definition}  \label{Def:MEstimator}
Given $n \in \mathbb{N}$, observations $Y_1, \ldots Y_n$ and a function $m : \Theta \times \mathsf{O} \rightarrow \mathbb{R}$ an estimator $\hat{\theta}(Y_1, \ldots, Y_n)$ which maximizes the function
\begin{equation*}
M_n : \mathsf{\Theta} \rightarrow \mathbb{R} , \quad \theta \mapsto M_n (\theta) := \frac{1}{n} \sum_{i=1}^{n} m(\theta, Y_i)
\end{equation*}
is called an \textit{M-estimator}.
\end{definition}

Similar to maximum likelihood estimation, it is sometimes possible to define an M-estimator to be a root of a certain equation. Loosely speaking, this corresponds to setting the derivative of $m$ with respect to $\theta$ equal to zero. Following \cite{vanderVaart}, we use a different term for such estimator.

\begin{definition}  \label{Def:ZEstimator}
Given $n \in \mathbb{N}$, observations $Y_1, \ldots Y_n$ and a function $\psi : \Theta \times \mathsf{O} \rightarrow \mathbb{R}^k$ an estimator $\hat{\theta}(Y_1, \ldots, Y_n)$ which is a root of the function
\begin{equation*}
Z_n : \Theta \rightarrow \mathbb{R}^k, \quad \theta \mapsto Z_n(\theta) := \frac{1}{n} \sum_{i=1}^{n} \psi (\theta, Y_i)
\end{equation*}
is called a \textit{Z-estimator}.
\end{definition}

\begin{remark}  \label{Rem:MEstimateElicConnection}
The connection of these two definitions to elicitability and identifiability becomes clear if the symbols $\frac{1}{n} \sum$ are replaced by $\mathbb{E}$ (which holds approximately for large $n$ if the law of large numbers is satisfied). Then, in the framework of Chapter~\ref{Chapter1}, $m$ can be interpreted as a scoring function (up to the sign) and $\psi$ can be interpreted as an identification function. Similarly, the estimate $\hat{\theta}_n$ then represents the value of a functional $T$. Below we see that such a correspondence is indeed justified under certain conditions. Moreover, we argue that in order to estimate the value $T(F)$ for a functional $T$, it is sensible to use M-estimators for elicitable and Z-estimators for identifiable functionals.
\end{remark}

The following three examples illustrate the concepts of M- and Z-estimation. The first example shows how maximum likelihood fits in the framework of M-estimation, see also van der Vaart~\cite[Example 5.3]{vanderVaart}.

\begin{example}
Let $Y_1, \ldots, Y_n$ be i.i.d. random variables having density $y \mapsto f(\theta, y)$, where $\theta \in \Theta$ parametrizes the distribution. The maximum likelihood method maximizes the likelihood function $\theta \mapsto \prod_{i=1}^{n} f(\theta, Y_i)$ or equivalently the log-likelihood $\theta \mapsto \sum_{i=1}^{n} \log f(\theta, Y_i)$. Consequently, the resulting estimator is an M-estimator for the choice $m(\theta, y) := \log f(\theta, y)$. If $f$ is differentiable with respect to $\theta$, it is possible to set $\psi (\theta , y) := \nabla_\theta \log f(\theta, y)$ and obtain a Z-estimator.
\end{example}

The following two examples are concerned with estimation of a location parameter and similar arguments can be found in \cite[Example 5.4]{vanderVaart}.

\begin{example}
In light of Remark~\ref{Rem:MEstimateElicConnection} and Example~\ref{Thm:ExMeanIsElicitable}, one idea to estimate the mean of a distribution is to choose $m(\theta,y) = -(y - \theta)^2$ or $\psi(\theta,y) = y- \theta$ since the first is a scoring (up to the sign) and the latter an identification function for the mean. Plugging these functions in and computing the solution gives $\hat{\theta}_n = \frac{1}{n} \sum_{i=1}^{n} Y_n$, which shows that this choice is reasonable.

Naturally, the same can be done to estimate the median. In this case choose $m(\theta, y) = - \vert y - \theta \vert$ or $\psi(\theta, y) = \sign ( \theta - y)$ and assume for simplicity that $n \in \mathbb{N}$ is odd. As above, computing the solution leads to $\hat{\theta}_n = Y_{(n+1)/2 : n}$, where $Y_{k:n}$ denotes the $k$-th order statistic of $Y_1, \ldots, Y_n$. This is indeed the empirical median.
\end{example}

\begin{example}
For estimates of location parameters it makes sense to define $\psi (\theta, y) := \varphi (y - \theta)$ because this guarantees the property of equivariance for the Z-estimator related to $\psi$. If the data are shifted by a vector $r \in \mathbb{R}^d$, then a shift of the estimator by $r$ will solve the new equation, i.e. we have $\hat{\theta}(Y_1 + r, \ldots, Y_n + r) = \hat{\theta}(Y_1, \ldots, Y_n) + r$ for $r \in \mathbb{R}$. One famous example for Z-estimators constructed in this manner are Huber's $K$-estimators, which for $K>0$ are defined via $\psi (\theta, y) := \varphi_K (y - \theta)$ with $\varphi_K$ given by
\begin{align*}
	\varphi_K (x) := \, \left\lbrace
	\begin{array}{ll}
	-K , &  \text{if } x \leq -K \\
	x , &  \text{if } \vert x \vert \leq K \\
	K , &  \text{if } x \geq K . 
	\end{array} \right.
\end{align*}
These estimators can be used to scale between the mean and the median. Recalling Definition~\ref{Def:ZEstimator} and the role of $\psi$ shows that for $K \rightarrow \infty$ the mean is obtained while for $K \rightarrow 0$ the median emerges. Moreover, it is also possible to choose
\begin{equation*}
m(\theta, y) := \vert \theta - y \vert^2 \mathbbm{1}_{\lbrace \vert \theta - y \vert \leq K \rbrace} + (2K \vert \theta - y \vert  - K^2) \mathbbm{1}_{\lbrace \vert \theta - y \vert > K \rbrace} 
\end{equation*}
in order to characterize these estimators in the framework of Definition~\ref{Def:MEstimator}, see also~\cite[Example 5.4]{vanderVaart}.
\end{example}

For simplicity, we consider only one-dimensional M- or Z-estimators in the following, i.e. we fix $k=1$ and $\Theta \subset \mathbb{R}$. Recall that an estimator $\hat{\theta}_n$ is called consistent for $\theta_0$ if $\hat{\theta}_n \rightarrow \theta_0$ in probability as $n \rightarrow \infty$. The next step is to show that a sequence of M-estimators $(\hat{\theta}_n)_{n \in \mathbb{N}}$ is consistent for $\theta_0$, where $\theta_0$ is the maximizer of some asymptotic function $M : \Theta \rightarrow \mathbb{R}$. The following theorem shows that this is indeed possible. Moreover, it is not needed that $\hat{\theta}_n$ maximizes $M_n$ but only that it nearly maximizes $M_n$. This means that for all $n \in \mathbb{N}$ we have
\begin{equation}  \label{Eqn:MestAlmostMax}
M_n (\hat{\theta}_n ) \geq \underset{\theta \in \Theta}{\sup} \, M_n(\theta) - \delta_n ,
\end{equation}
where $(\delta_n)_{n \in \mathbb{N}} >0$ is a sequence of random variables such that $\delta_n \rightarrow 0$ in probability.

\begin{theorem}[van der Vaart {\cite[Thm. 5.7]{vanderVaart}}]  \label{Thm:MEstimatorI}
Let $M_n : \Theta \rightarrow \mathbb{R}$ be random functions, $M : \Theta \rightarrow \mathbb{R}$ a deterministic function and $(\hat{\theta}_n)_{n \in \mathbb{N}}$ a sequence of estimators which satisfies $M_n(\hat{\theta}_n) \geq M_n (\theta_0) - \delta_n$, where $(\delta_n)_{n \in \mathbb{N}} >0$ is a sequence of random variables satisfying $\delta_n \rightarrow 0$ in probability. If for every $\varepsilon > 0$ we have
\begin{enumerate}[label=(\roman*)]
	\item $\sup_{\theta \in \Theta} \vert M_n(\theta) - M(\theta) \vert \rightarrow 0$ in probability and
	\item $\sup_{ \vert \theta - \theta_0 \vert > \varepsilon} M(\theta) < M(\theta_0)$
\end{enumerate}
then $\hat{\theta}_n \rightarrow \theta_0$ in probability.
\end{theorem}

\begin{proof}
We follow \cite{vanderVaart} and use the assumption on $(\hat{\theta}_n)_{n \in \mathbb{N}}$ together with (i) to show that $M(\hat{\theta}_n)$ converges to $M(\theta_0)$ in probability. In detail we have
\begin{align*}
\vert M(\theta_0) - M(\hat{\theta}_n) \vert &\leq \vert M(\theta_0) - M_n(\theta_0) \vert  + \vert M_n(\theta_0) - M_n(\hat{\theta}_n) \vert  + \vert M_n(\hat{\theta}_n) - M(\hat{\theta}_n) \vert \\
&\leq 2 \, \underset{\theta \in \Theta}{\sup} \vert M_n(\theta) - M(\theta) \vert + \delta_n
\end{align*}
and both terms converge to zero in probability by assumption. Due to condition (ii), we have for any $\varepsilon > 0$ that $\vert \hat{\theta}_n - \theta_0 \vert > \varepsilon$ implies $M(\hat{\theta}_n)  < M(\theta_0)$. Consequently, there is an $\eta(\varepsilon) > 0$ such that $M(\hat{\theta}_n) < M(\theta_0) - \eta(\varepsilon)$. This gives
\begin{equation*}
\mathbb{P} ( \vert \hat{\theta}_n - \theta_0 \vert  > \varepsilon) \leq \mathbb{P} ( M(\theta_0) - M(\hat{\theta}_n) > \eta(\varepsilon) ) \rightarrow 0
\end{equation*}
for $n \rightarrow \infty$, by using the convergence of the first part of the proof.
\end{proof}

\begin{remark}  \label{Rem:EstOuterMeasure}
In condition (i) of Theorem~\ref{Thm:MEstimatorI} it is implicitly assumed that the mapping $G \mapsto \sup_{\theta \in \Theta} \vert G(\theta) - M(\theta) \vert$ is measurable. If this is not the case, the theorem can still be proved if the convergence in (i) holds in \textit{outer measure}, i.e. there is a sequence $(A_n)_{n \in \mathbb{N}} \subset \mathscr{A}$ such that $ \lbrace \sup_{\theta \in \Theta} \vert M_n(\theta) - M(\theta) \vert > \varepsilon \rbrace \subset A_n$ holds for all $n \in \mathbb{N}$ and $\mathbb{P}(A_n) \rightarrow 0$ as $n \rightarrow \infty$. For more details we refer to van der Vaart~\cite[Sec. 18.2]{vanderVaart}.
\end{remark}

Theorem~\ref{Thm:MEstimatorI} is now used to analyze M-estimators. To this end, let $T: \mathcal{F} \rightarrow \Theta$ be an elicitable functional and $S$ a strictly $\mathcal{F}$-consistent scoring function for $T$. Fix $F \in \mathcal{F}$ and let $Y_i$, $i = 1, \ldots, n$ be i.i.d. with $Y_i =^d F$. Moreover, define the functions $M_n(\theta) : = - \frac{1}{n} \sum_{i=1}^{n} S(\theta, Y_i)$ and  $M(\theta) := -\bar{S}(\theta, F) $. This is indeed a reasonable choice because the law of large numbers guarantees $M_n(\theta) \rightarrow M(\theta)$ in probability for any $\theta \in \Theta$. If the convergence $M_n \rightarrow M$ is even uniform in $\theta$ and $M$ satisfies condition (ii), then a straightforward way to estimate $\theta_0 := T(F)$ (the unique maximum of $M$) consistently, is using a $\hat{\theta}_n$ which nearly maximizes $M_n$, i.e. it satisfies condition~(\ref{Eqn:MestAlmostMax}).\\
\par
The condition of uniform convergence in probability is closely connected to the Glivenko-Cantelli theorem (see for instance Klenke~\cite[Thm. 5.23]{KlenkeProb}), as pointed out in \cite{vanderVaart}. More precisely, a class of integrable functions $\mathcal{C}$ is called a $\mathbb{P}$-Glivenko-Cantelli class if the convergence
\begin{equation*}
\underset{f \in \mathcal{C}}{\sup} \, \Big| \frac{1}{n} \sum_{i=1}^{n} f(X_i) - \mathbb{E}f(X) \Big| \rightarrow 0, \quad n \rightarrow \infty
\end{equation*}
holds $\mathbb{P}$-almost surely (where again the outer measure might be necessary, as mentioned in Remark~\ref{Rem:EstOuterMeasure}). The following lemma summarizes sufficient conditions mentioned in \cite{vanderVaart} which ensure that both requirements of Theorem~\ref{Thm:MEstimatorI} hold in the situation where $m(\theta, y) = -S(\theta, y)$ and $M(\theta) = -\bar{S}(\theta,F)$.

\begin{lemma}
Let $\Theta \subset \mathbb{R}^k$ be compact and let $S$ be such that $\theta \mapsto S(\theta, y)$ is continuous for every $y \in \mathsf{O}$. Suppose there is an $\mathcal{F}$-integrable $g$ such that $\vert S(\theta, \cdot) \vert \leq g$  almost surely for any $\theta$. Then the conditions (i) and (ii) of Theorem~\ref{Thm:MEstimatorI} are satisfied. 
\end{lemma}

\begin{proof}
Due to the assumptions on $S$, the class $\lbrace S(\theta, \cdot) \mid \theta \in \Theta \rbrace$ has the Glivenko-Cantelli property (see \cite[Example 19.8]{vanderVaart} for a detailed proof) and hence, condition (i) is satisfied. Furthermore, observe that since $g$ dominates $S$, we apply dominated convergence to show that $M$ is continuous. Now suppose condition (ii) is not satisfied. Then there is an $\varepsilon' > 0$ and a (deterministic) sequence $(\theta_n)_{n \in \mathbb{N}}$ such that $\vert \theta_n - \theta_0 \vert > \varepsilon'$ and $M(\theta_n) \rightarrow c \geq M(\theta_0)$. Since $\Theta$ is compact, we choose a subsequence $(\theta_{n_k} )_{k \in \mathbb{N}}$ such that $\theta_{n_k} \rightarrow \tilde{\theta}$ for some $\tilde{\theta} \in \Theta$. Due to the continuity of $M$, we conclude $M(\tilde{\theta}) = c \geq M(\theta_0)$, but at the same time we have $\vert \tilde{\theta} - \theta_0 \vert \geq \varepsilon' > 0$, contradicting the fact that $\theta_0$ is the unique maximum of $M$.
\end{proof}

The assumptions imposed on $S$ as well as condition (ii) of Theorem~\ref{Thm:MEstimatorI} are rather strong. In order to proof consistency in a more general setting we consider Z-estimators according to Definition~\ref{Def:ZEstimator}. The following proposition and its proof can be found in Huber and Ronchetti~\cite[Prop. 3.1]{HuberRonchetti} and van der Vaart~\cite[Lemma 5.10]{vanderVaart}.

\begin{prop}  \label{Thm:MEstimatorII}
For any $n \in \mathbb{N}$ let $Z_n : \Theta \rightarrow \mathbb{R}$ be a decreasing random function and $Z: \Theta \rightarrow \mathbb{R}$ a deterministic function such that $Z_n(\theta) \rightarrow Z(\theta)$ in probability for all $\theta \in \Theta$. Moreover, let $\theta_0$ be the unique point such that for any $\varepsilon >0$ it holds that $Z(\theta_0 - \varepsilon) > 0 > Z(\theta_0 + \varepsilon)$. If $(\hat{\theta}_n)_{n \in \mathbb{N}}$ is a sequence satisfying $Z_n(\hat{\theta}_n) \rightarrow 0$ in probability, then $\hat{\theta}_n \rightarrow \theta_0$ in probability.
\end{prop}

\begin{proof}
Take $\varepsilon >0$ and set $\eta := \eta(\varepsilon) < \min (  Z(\theta_0 - \varepsilon ), \vert Z(\theta_0 + \varepsilon ) \vert )$. Then the monotonicity of $Z_n$ implies for any $n \in \mathbb{N}$ the inequality
\begin{align*}
\mathbb{P}(\hat{\theta}_n < \theta_0 - \varepsilon) &\leq \mathbb{P} ( Z_n(\hat{\theta}_n) \geq Z_n(\theta_0 - \varepsilon) ) \\
&= \mathbb{P}( Z_n (\hat{\theta}_n) \geq Z_n( \theta_0 - \varepsilon), \, Z_n (\theta_0 - \varepsilon) > \eta ) \\
&\quad + \mathbb{P}( Z_n (\hat{\theta}_n) \geq Z_n( \theta_0 - \varepsilon), \, Z_n (\theta_0 - \varepsilon) \leq \eta ) \\
&\leq \mathbb{P} ( Z_n(\hat{\theta}_n) > \eta ) + \mathbb{P}( Z_n(\theta_0 - \varepsilon) \leq \eta)
\end{align*}
and analogously
\begin{equation*}
\mathbb{P}(\hat{\theta}_n > \theta_0 + \varepsilon) \leq \mathbb{P}( Z_n(\hat{\theta}_n) < - \eta ) + \mathbb{P}( Z_n(\theta_0 + \varepsilon) \geq - \eta) .
\end{equation*}
Finally, it follows that $\mathbb{P}( \vert \hat{\theta}_n - \theta_0 \vert > \varepsilon)$ vanishes for $n \rightarrow \infty$, by applying both previous inequalities and the fact that $Z_n(\hat{\theta}_n) \rightarrow 0$ as well as $Z_n( \theta_0 \pm \varepsilon) \rightarrow Z( \theta_0 \pm \varepsilon)$ in probability.
\end{proof}

Similar to the discussion above, Proposition~\ref{Thm:MEstimatorII} can now be used to analyze Z-estimators. To this end, let $T: \mathcal{F} \rightarrow \Theta$ be an identifiable functional with strict $\mathcal{F}$-identification function $V$, fix $F \in \mathcal{F}$ and let $Y_i$, $i = 1, \ldots, n$ be i.i.d. with $Y_i =^d F$. It is then sensible to define $Z_n(\theta) := \frac{1}{n} \sum_{i=1}^{n} V(\theta, Y_i)$ as well as $Z(\theta) := \bar{V}(\theta, F)$. Due to the identification property of $V$, $\theta_0 := T(F)$ is the unique root of $Z$ and for all $\theta \in \Theta$ the convergence $Z_n(\theta) \rightarrow Z(\theta)$ follows from the law of large numbers. If $V$ can be chosen such that $\theta \mapsto V(\theta, y)$ is decreasing for any $y \in \mathsf{O}$, then $Z_n$ is decreasing as well and $Z(\theta_0 - \varepsilon) > 0 > Z(\theta_0 + \varepsilon)$ is also satisfied. Therefore, $T(F)$ can be consistently estimated using a Z-estimator. Compared to the assumptions imposed on $S$, the consistency of $\hat{\theta}_n$ is now ensured with only the monotonicity of $V$. \\
\par
This section considers M-estimation and Z-estimation in the one-dimensional case only. However, it is possible to generalize consistency results like Theorem~\ref{Thm:MEstimatorI} or Proposition~\ref{Thm:MEstimatorII} to higher dimensions using a larger collection of assumptions. For detailed results and more references we refer to Huber and Ronchetti~\cite{HuberRonchetti}. Apart from that, it is also possible to prove asymptotic normality for M- or Z-estimators, see for instance \cite{HuberRonchetti} as well as van der Vaart~\cite{vanderVaart}.

\chapter{Elicitability in risk management}
\label{Chapter3}

This chapter introduces coherent and convex measures of risk and presents three risk measures which are popular in practice or academia. Moreover, law-invariant risk measures are interpreted as functionals on a set of distribution functions $\mathcal{F}$ in order to analyze which of them are elicitable. The recent result of Fissler and Ziegel~\cite{FissZieg} that Value at Risk and Expected Shortfall are jointly elicitable is proved. Finally, it is explained why elicitability is a desirable property for backtesting risk measure estimates.

\section{Risk measures and their properties} 

In this section we introduce risk measures and the properties of coherence and convexity, which are considered desirable properties for measures of risk. The majority of the presented material is based on the book chapters by F\"ollmer and Schied~\cite[Ch. 4]{FoellmerSchied} and McNeil et al.~\cite[Ch. 8]{McNeilRisk}. Informally, a risk measure is a mapping which assigns a number $\rho (X)$ to a random variable $X$. Usually, $X$ is interpreted as a financial position and $\rho(X)$ as stating the riskiness of $X$. If $\rho (X)$ is interpreted in terms of money, $\rho$ is also called a monetary measure of risk. \\
\par
In the following, we fix some probability space $(\Omega, \mathscr{A}, \mathbb{P})$ which supports all random variables appearing in this section. Moreover, for $p \geq 1$ we define $\mathcal{L}^p (\Omega, \mathscr{A}, \mathbb{P})$ to be the vector space of random variables $X$ on $(\Omega, \mathscr{A}, \mathbb{P})$ such that $\mathbb{E} \vert X \vert^p$ is finite. Similarly, $\mathcal{L}^0(\Omega, \mathscr{A}, \mathbb{P})$ is called the space of all random variables and $\mathcal{L}^\infty (\Omega, \mathscr{A}, \mathbb{P})$ the space of all random variables which are $\mathbb{P}$-almost surely bounded. To shorten notation, we suppress  $(\Omega, \mathscr{A}, \mathbb{P})$ and write $\mathcal{L}^p$ instead of $\mathcal{L}^p (\Omega, \mathscr{A}, \mathbb{P})$, for instance.

\begin{definition}
Let $\mathcal{X}$ be a collection of random variables on $(\Omega, \mathscr{A})$ which forms a vector space and is closed under addition of constants. Then a mapping 
\begin{equation*}
\rho : \mathcal{X} \rightarrow \mathbb{R}
\end{equation*}
is called a \textit{measure of risk} or just a \textit{risk measure}.
\end{definition}

\begin{remark}  \label{Rem:DefinitionRisk}
There exist several conventions to define risk measures and each represents a way of interpreting $X$ and $\rho(X)$. Firstly, $X$ can be thought of as either positions in a portfolio or losses resulting from these positions. In the first case, high values of $X$ are desirable, while in the latter low values are. Secondly, high values of $\rho$ can express a high as well as a low level of risk. We use the following convention: Elements of  $\mathcal{X}$ are interpreted as portfolio values and higher values of $\rho$ represent riskier positions. 
\end{remark}

\begin{remark}
Note that the definition of measures of risk does not need a probability measure on $(\Omega, \mathscr{A})$. This is only needed for special choices of $\mathcal{X}$, for example subsets of  $\mathcal{L}^p$, $\mathcal{L}^0$, or $\mathcal{L}^\infty$. Moreover, if a probability measure $\mathbb{P}$ on $(\Omega, \mathscr{A})$ is fixed, it is possible to define $\rho(X)$ via the distribution of $X$ under $\mathbb{P}$. This approach leads to \textit{law-invariant} risk measures an  is essential in Subsection~\ref{Sec:SpectralRiskMeasures} as well as in Section~\ref{Sec:VaRESandEVaR}.
\end{remark}

When looking at $\rho(X)$ instead of $X$ (or its distribution), most of the information contained in $X$ is usually lost. Therefore, this approach needs some justification. There exist two main motivations for studying measures of risk:
\begin{itemize}
\item If $\rho (X)$ is interpreted as a simple number of riskiness, a comparison of two different portfolios is simple, while a comparison in terms of distributions or random variables can be arbitrarily difficult.
\item If $\rho (X)$ is interpreted as an amount of cash, this amount can be thought of to be necessary to protect the portfolio $X$ against losses. In this case, $\rho (X)$ plays the role of a capital buffer required by some financial regulator or a margin demanded by a counterparty in trading.
\end{itemize}

In both interpretations, we use the convention that higher values of $\rho$ represent riskier positions. Since a risk measure $\rho$ is just a real-valued mapping on $\mathcal{X}$, it is reasonable to impose more structure on it. This is done by introducing the properties of coherence and convexity. We proceed chronologically and start with coherent measures of risk.

\subsection{Coherent measures of risk} 

Coherent measures of risk arise from an axiomatic approach by Artzner et al.~\cite{ArtznerCoherent} to find an appropriate $\rho$. The approach uses a finite sample space $\Omega$, but is extended to general $\Omega$ by Delbaen~\cite{DelbaenCoherent}. Due to their simplicity and intuitive interpretation, the axioms of the following definition are the most prominent characteristics of risk measures.

\begin{definition}  \label{Def:CoherentRisk}
Let $\rho : \mathcal{X} \rightarrow \mathbb{R}$ be a measure of risk. If for all $X, Y \in \mathcal{X}$ the risk measure $\rho$ satisfies the following four properties, it is called \textit{coherent}.
\begin{enumerate}[label=\arabic*)]
\item Monotonicity: If $X \leq Y$, then $\rho(Y) \leq \rho(X)$.
\item Positive homogeneity: If  $\lambda \geq 0$, then $\rho(\lambda X) = \lambda \rho(X)$.
\item Translation invariance: If $c \in \mathbb{R}$, then $\rho (X + c) = \rho(X) - c$.
\item Subadditivity: $ \rho (X + Y) \leq \rho (X) + \rho (Y)$.
\end{enumerate}
\end{definition}

\begin{remark}
In \cite{ArtznerCoherent}, the definition of coherent measures of risk is stated for $\mathcal{X} = L^{\infty}$, where $L^\infty$ is the space of all equivalence classes with respect to $\mathbb{P}$-almost sure equality of $\mathcal{L}^\infty$. In \cite{DelbaenCoherent}, it is shown that it is impossible to extend this definition to the space $L^0$. More precisely, if $ (\Omega, \mathscr{A}, \mathbb{P})$ is an atomless probability space, there is no coherent risk measure on $L^0 (\Omega, \mathscr{A}, \mathbb{P})$. This problem can be solved if the definition of risk measures is extended such that $\rho$ takes values in $\mathbb{R} \cup \lbrace \infty \rbrace$. We do not use such an extension and only consider measures of risk which are real-valued.
\end{remark}

As noted above and suggested by the name, the properties in Definition~\ref{Def:CoherentRisk} are selected because they allow for a coherent interpretation. In particular, the following arguments can be used to justify them (see also \cite{ArtznerCoherent}):
\begin{enumerate}[label=\arabic*)]
\item This axiom is natural, since it states that positions with higher values \textit{regardless of what happens} are less risky.
\item Positive homogeneity means that if a position is increased or decreased by a certain factor, then the risk of the position is scaled by the same factor.
\item A constant is usually interpreted to be a cash position. Therefore, if cash is added to a position, the risk of this position is reduced by the amount of added cash.
\item The subadditivity property has received the most attention. Since the sum of two positions can never be riskier than the risk of the two positions combined, it is supposed to represent the effects of diversification. The importance of subadditivity can also be seen in the way of \cite{ArtznerCoherent}: If the inverse inequality would hold, companies would be incentivized to split up in order to reduce risk.
\end{enumerate}

Acceptance sets, which are also introduced in~\cite{ArtznerCoherent}, are closely related to measures of risk. Loosely speaking, the acceptance set of a risk measure contains all positions having an acceptable level of risk. This level of risk can be set by a risk manager or a financial regulator, for example.

\begin{definition}  \label{Def:AcceptanceSet}
Let $\rho$ be a monotonic and translation invariant measure of risk. Then the set 
\begin{equation*}
\mathcal{A}_\rho := \lbrace X \in \mathcal{X} \mid \rho(X) \leq 0 \rbrace
\end{equation*}
is called \textit{acceptance set} of $\rho$.
\end{definition}

When thinking in terms of a monetary measure of risk $\rho$, the set $\mathcal{A}_\rho$ contains all positions for which no additional cash amount has to be added to make them acceptable. Conversely, if a set of positions $\mathcal{A}$ is given, it also defines a risk measure $\rho$, as follows: For $X \in \mathcal{X}$ the value $\rho (X)$ is the minimal amount of cash such that $X + \rho(X) \in \mathcal{A}$ is satisfied. For details we refer to \cite{FoellmerSchied}.

\subsection{Convex measures of risk} 

The next step in the theory of risk measures is to relax the conditions 2) and 4) of Definition~\ref{Def:CoherentRisk}. This is done by F\"ollmer and Schied~\cite{FoellmerSchiedConvex} and the reasoning behind this is that positive homogeneity is problematic in some situations. For example, if the size of a position increases dramatically, there might not be enough supply or demand to liquidate the positions if necessary. Hence, it is reasonable to suggest that a decline in market liquidity leads to a case where $\rho(\lambda X) > \lambda \rho(X)$ for large $\lambda$. The following definition generalizes coherent measures of risk.

\begin{definition}
Let $\rho : \mathcal{X} \rightarrow \mathbb{R}$ be a measure of risk. If for all $X, Y \in \mathcal{X}$ the risk measure $\rho$ satisfies the following three properties, it is called \textit{convex}.
\begin{enumerate}[label=\arabic*)]
\item Monotonicity: If $X \leq Y$, then $\rho(Y) \leq \rho(X)$.
\item Translation invariance: If $c \in \mathbb{R}$, then $\rho (X + c) = \rho(X) - c$.
\item Convexity: If $\lambda \in [0,1]$, then $\rho ( \lambda X + (1-\lambda) Y) \leq \lambda \rho(X) + (1-\lambda) \rho(Y)$.
\end{enumerate}
\end{definition}

Note that every coherent measure of risk is also a convex measure of risk, since subadditivity together with positive homogeneity implies convexity.  We now check how a convex risk measure $\rho$ behaves for a position $\lambda X$. For simplicity, assume that $\rho(0) \leq 0$ holds. Then convexity of $\rho$ implies
\begin{equation*}
\rho (\lambda X) = \rho ( (1- \lambda) \cdot 0 + \lambda X) \leq \lambda \rho(X)
\end{equation*}
for $\lambda \in (0,1)$ and similarly 
\begin{equation*}
\lambda \rho(X) \leq \rho (\lambda X) + (\lambda - 1) \rho (0)
\end{equation*}
for $\lambda >1$. Therefore, convex risk measures satisfy the inequalities
\begin{equation} \label{Eqn:PosHomogeneityNot}
\rho (\lambda X) \leq \lambda \rho (X) \quad \text{for } \lambda \in [0,1] \quad \text{ and } \quad \rho (\lambda X) \geq \lambda \rho (X) \quad \text{for } \lambda \geq 1
\end{equation}
instead of positive homogeneity. This subsection ends with the following example, which discusses a well-known risk measure mentioned in F\"ollmer and Schied~\cite[Example 4.13]{FoellmerSchied} called the \textit{entropic measure of risk}.

\begin{example}
Fix $\alpha > 0$, a probability measure $\mathbb{P}$, and let $\mathbb{E} \exp(-\alpha X)$ be finite for any $X \in \mathcal{X}$. Then the mapping
\vspace{-0.8mm}
\begin{equation*}
\vspace{-0.8mm}
\rho : \mathcal{X} \rightarrow \mathbb{R}, \quad X \mapsto  \alpha^{-1} \log \mathbb{E} \exp (-\alpha X)
\end{equation*}
is a convex measure of risk. Indeed, monotonicity and translation invariance follow from the properties of $\log$ and $\exp$. Convexity follows by applying translation invariance, the monotonicity of the logarithm and the convexity of the mapping $x \mapsto \exp(- \alpha x)$. In detail, this results in
\begingroup
\allowdisplaybreaks[0]
\begin{align*}
\alpha \rho ( \lambda X &+ (1-\lambda) Y) - \lambda \rho (X) - (1- \lambda) \rho (Y) \\
&= \alpha \rho \big( \lambda (X + \rho (X) ) + (1-\lambda) (Y + \rho (Y) ) \big) \\
&= \log \mathbb{E} \exp \big( - \alpha \lambda (X + \rho(X)) - \alpha (1- \lambda) (Y + \rho (Y) ) \big) \\
&\leq \log \big[ \lambda \mathbb{E} \exp ( - \alpha (X + \rho(X)) ) + (1-\lambda) \mathbb{E} \exp( - \alpha (Y + \rho(Y) )) \big] \\
&\leq \log \max \big( \mathbb{E} \exp (- \alpha (X + \rho(X)) ) \, , \, \mathbb{E} \exp( - \alpha (Y + \rho(Y) )) \big) \\
&= \max \big( \log \mathbb{E} \exp (- \alpha (X + \rho(X)) ) \, , \, \log \mathbb{E} \exp( - \alpha (Y + \rho(Y) ))  \big) \\
&= \alpha \max ( \rho ( X + \rho(X)) , \rho ( Y + \rho (Y) ) ) = 0 ,
\end{align*}
\endgroup
where the last equality follows from the fact that $\rho(X + \rho(X)) =0$ for any translation invariant risk measure. This shows that $\rho$ is a convex risk measure, however, it is not a coherent measure of risk, because it does not satisfy positive homogeneity. In order to show this, let $X$ have a Bernoulli distribution with parameter $p=\frac{1}{2}$. Then positive homogeneity of $\rho$ would imply that $\lambda \log( 1/2 \, e^{-\alpha} + 1/2) = \log (1/2 \, e^{- \alpha \lambda} + 1/2)$ holds for any choice of $\alpha , \lambda >0$. But this is not true, if for example $\alpha = 1$ and $\lambda = 2$ are selected, then ``$<$'' holds, hence $\rho$ cannot be positive homogeneous. This is in accordance with the results in~(\ref{Eqn:PosHomogeneityNot}).
\end{example}

\subsection{Spectral measures of risk} 
\label{Sec:SpectralRiskMeasures}

Spectral measures of risk are introduced by Acerbi~\cite{AcerbiSpectral} and Kusuoka~\cite{KusuokaSpectral}. The latter derives spectral risk measures while characterizing all law-invariant, coherent, and comonotone measures of risk. Acerbi constructs reasonable risk measures by using convex combinations of other risk measures. This section follows the approach of \cite{AcerbiSpectral} and introduces the concept of comonotone risk measures afterwards. In the following it is essential to fix a probability measure $\mathbb{P}$ on $(\Omega, \mathscr{A})$, since we consider measures of risk $\rho$ for which $\rho(X)$ only depends on the law of $X$. These risk measures are called \textit{law-invariant} and they satisfy $\rho (X) = \rho(Y)$ whenever $X$ and $Y$ have the same law under $\mathbb{P}$. The next lemma is stated and proved in \cite[Prop. 2.2]{AcerbiSpectral}.

\begin{lemma}  \label{Thm:ConvexCombiRisk}
Let $(Z, \mathcal{Z}, \nu)$ be a measurable space, $\nu$ a probability measure and $(\rho_z)_{z \in Z}$ a family of risk measures such that $z \mapsto \rho_z (X)$ is integrable w.r.t. $\nu$ for any $X \in \mathcal{X}$. Define the measure of risk $\rho : \mathcal{X} \rightarrow \mathbb{R}$ via
\begin{equation*}
\rho (X) := \int_Z \rho_z (X) \, \mathrm{d} \nu (z),
\end{equation*}
then the following hold:
\begin{enumerate}[label=(\roman*)]
\item If $\rho_z$ is coherent for all $z \in Z$, then $\rho$ is coherent.
\item If $\rho_z$ is convex for all $z \in Z$, then $\rho$ is convex.
\end{enumerate}
\end{lemma}

\begin{proof}
Monotonicity, positive homogeneity, subadditivity and convexity all carry over from $\rho_z$ to $\rho$ due to the properties of the integral. Translation invariance follows from the fact that $\nu (Z) = 1$.
\end{proof}

There exist two equivalent definitions of spectral measures of risk and their construction is motivated by Lemma~\ref{Thm:ConvexCombiRisk}. We start with the one used in McNeil et al.~\cite[Sec. 8.2]{McNeilRisk} and continue with a lemma showing an equivalent representation. For a random variable $X$ with distribution function $F$, we denote its lower quantile function by $F^{\leftarrow}_X$, see also Definition~\ref{Def:quantilefunction}. 

\begin{definition}  \label{Def:SpectralRiskMeasure}
Let $\phi : [0,1] \rightarrow \mathbb{R}_+$ be a positive, integrable, and decreasing function such that $\int \phi(s) \, \mathrm{d}s =1$ holds. Then the mapping
\begin{equation*}
\rho_\phi : \mathcal{X} \rightarrow \mathbb{R}, \quad X \mapsto \rho_\phi (X) := - \int_{0}^{1} F^{\leftarrow}_X(s) \phi (s) \, \mathrm{d}s
\end{equation*}
is called the \textit{spectral measure of risk} associated to the \textit{risk aversion function} $\phi$.
\end{definition}

Although spectral risk measures are formed as a convex combination, it is not possible to apply Lemma~\ref{Thm:ConvexCombiRisk} to show coherence or convexity since the mappings $X \mapsto - F^{\leftarrow}_X (s)$, $s \in (0,1)$ are not convex risk measures (see Example~\ref{Thm:ExVaRnotConvex}). They are called \textit{Value at Risk} and their properties are discussed in Subsection~\ref{Sec:ValueAtRisk}. However, the following lemma shows that spectral measures of risk can be represented as convex combinations of the coherent measure of risk \textit{Expected Shortfall}, which is considered in Subsection~\ref{Sec:ExpectedShortfall}. The result and its proof can be found in \cite{AcerbiSpectral} as well as in~\cite[Prop. 8.18]{McNeilRisk}.

\begin{lemma}  \label{Thm:SpectralRiskAlternativeRep}
Let $\rho_\phi : \mathcal{X} \rightarrow \mathbb{R}$ be a spectral measure of risk with risk aversion function $\phi$. Moreover, for any probability measure $\nu$ on $([0,1], \mathcal{B}([0,1]))$ define a risk measure $\rho_\nu : \mathcal{X} \rightarrow \mathbb{R}$ via
\begin{equation*}
\rho_\nu (X) := \int_{[0,1]} \mathrm{ES}_\alpha (X) \, \mathrm{d}\nu (\alpha),
\end{equation*}
where
\begin{equation*}
\mathrm{ES}_\alpha (X) := - \frac{1}{\alpha} \int_{0}^{\alpha} F^{\leftarrow}_X (u) \, \mathrm{d}u .
\end{equation*}
Then for any $\phi$ there is a $\nu$ such that we have $\rho_\phi = \rho_\nu$ and the converse implication is also correct. 
\end{lemma}

\begin{proof}
At first, let $\rho_\phi$ be given. Since $\phi$ is decreasing, there exists a right-continuous version, which is denoted by $\phi_r$ and coincides with $\phi$ for a.e. $u \in [0,1]$. Now define the Lebesgue-Stieltjes measure w.r.t. $\phi_r$ on intervals via $\mu ( (u,1]) = \phi_r (u) - \phi_r(1)$ and define another measure $\tilde{\mu} := \mu + \phi_r (1) \delta_1$, which leads to the identity $\phi_r(u) =  \tilde{\mu}( (u,1] )$ for any $u \in [0,1)$. Using the Fubini-Tonelli theorem, we calculate
\begin{align*}
\rho_\phi (X) &= - \int_{0}^{1} F^{\leftarrow}_X (u) \phi_r (u) \, \mathrm{d} u = - \int_{0}^{1} F^{\leftarrow}_X (u) \int_{(u,1]} \mathrm{d} \tilde{\mu} (\alpha) \mathrm{d}u \\
&= \int_{[0,1]} \alpha  \left( - \frac{1}{\alpha} \int_{0}^{1} \mathbbm{1}_{(u,1]} (\alpha) F^{\leftarrow}_X (u) \, \mathrm{d}u \right) \mathrm{d} \tilde{\mu}(\alpha) = \int_{[0,1]} \mathrm{ES}_\alpha (X) \, \mathrm{d} \nu (\alpha) ,
\end{align*}
where the measure $\nu$ is defined via $\mathrm{d}\nu (\alpha) = \alpha \, \mathrm{d}\tilde{\mu} (\alpha) $. Using integration by parts for Lebesgue-Stieljes integrals (see for instance Hewitt~\cite{Hewitt}) gives
\begin{align*}
\nu ( [0,1] ) = \int_{[0,1]} \alpha \, \mathrm{d} \mu(\alpha) + \phi_r (1) = -\phi_r (1) - \int_{0}^{1} - \phi_r (u) \, \mathrm{d} u + \phi_r (1) = 1,
\end{align*}
showing that $\nu$ is a probability measure on $([0,1], \mathcal{B}([0,1]))$. For the converse implication, let $\rho_\nu$ be given and define the function $\phi (u) := \int_{ [u,1] } \frac{1}{\alpha} \, \mathrm{d} \nu (\alpha)$, which is positive and decreasing on $(0,1]$. As above we interchange the order of integration to obtain $\rho_\nu (X) = \rho_\phi(X)$. Moreover, we have
\begin{equation*}
\int_{ [0,1] } \frac{1}{\alpha} \int_{0}^{1} \mathbbm{1}_{[0, \alpha]} ( u) \, \mathrm{d}u \, \mathrm{d}\nu(\alpha ) = \int_{[0,1]} \mathrm{d}\nu (\alpha) = 1 
\end{equation*}
and using the Fubini-Tonelli theorem again this implies $\int \phi (u) \, \mathrm{d}u = 1$.
\end{proof}

In the following, the measure $\nu$ is called the \textit{spectrum} of $\rho_\nu$. As mentioned above, Kusuoka~\cite{KusuokaSpectral} characterizes spectral measures of risk  as law-invariant, coherent, and comonotone measures of risk. Therefore, the rest of this subsection introduces comonotone risk measures and shows that spectral risk measures are comonotone. Both of the following definitions are part of \cite[Def. 6]{KusuokaSpectral}.

\begin{definition}
Two random variables $X$ and $Y$ on $(\Omega, \mathscr{A}, \mathbb{P})$ are called \textit{comonotonic} if the inequality
\begin{equation*}
( X(\omega) - X(\omega^\prime)) (Y(\omega) - Y(\omega^\prime) ) \geq 0
\end{equation*}
holds $ \mathrm{d} \mathbb{P}(\omega) \otimes \mathrm{d}\mathbb{P}(\omega^\prime)$-almost surely.
\end{definition}

\begin{definition}
Let $\rho : \mathcal{X} \rightarrow \mathbb{R}$ be a measure of risk. If for all comonotonic $X, Y \in \mathcal{X}$ it holds that
\begin{equation*}
\rho (X + Y) = \rho(X) + \rho (Y),
\end{equation*}
then $\rho$ is called a \textit{comonotone} measure of risk.
\end{definition}

\begin{remark}  \label{Rem:ComonotoneInterpret}
The concept of a comonotone measure of risk is closely connected to subadditivity, which formalizes the effect of diversification. Although it is reasonable to assume that diversification does not increase risk, it seems doubtful to assume that it decreases risk in all cases. For instance, if two positions $X$ and $Y$ move up and down collectively, it is not plausible that adding them to one position $X+Y$ reduces risk. The concept of two positions moving together is formalized by comonotonic random variables, and a comonotone measure of risk will not show a decrease in risk if such random variables are added.
\end{remark}

The next step is to show that all spectral measures of risk are comonotone. We begin with a lemma which gives equivalent conditions for comonotonicity. It collects results from \cite{KusuokaSpectral} and \cite{McNeilRisk}.

\begin{lemma}  \label{Thm:ComonotonicEquivalent}
Let $X$ and $Y$ be random variables. Then the following are equivalent:
\begin{enumerate}[label=(\roman*)]
\item $X$ and $Y$ are comonotonic.
\item $F_{X,Y}(a, b) = \min( F_X (a), F_Y(b))$, where $F_{X,Y}$ is the joint distribution function of $X$ and $Y$.
\item $(X,Y) =^d ( F^{\leftarrow}_X (U) , F^{\leftarrow}_Y(U) )$ for a random variable $U =^d \mathcal{U}([0,1])$.
\end{enumerate}
\end{lemma}

\begin{proof}
We start with ``(i)$\Rightarrow$(ii)'' and use the same idea as \cite[Prop. 17]{KusuokaSpectral}. Fix arbitrary $a, b \in \mathbb{R}$ and define the sets $A := \lbrace X \leq a \rbrace \cap \lbrace Y > b \rbrace$ as well as $B:= \lbrace Y \leq b \rbrace \cap \lbrace X > a \rbrace$. For any $(\omega , \omega^\prime) \in A \times B$ we have $( X(\omega) - X(\omega^\prime)) (Y(\omega) - Y(\omega^\prime) ) < 0$ and hence the comonotonicity of $X$ and $Y$ implies $0 = (\mathbb{P} \otimes \mathbb{P}) (A \times B) = \mathbb{P}(A) \mathbb{P}(B)$. So $A$ or $B$ is a null set. If $A$ is a null set, this gives $\mathbb{P}(X \leq a) = \mathbb{P}(X \leq a, Y \leq b) \leq \mathbb{P}(Y \leq b)$ and repeating the same argument for the case where $B$ is a null set shows (ii). To prove ``(ii)$\Rightarrow$(iii)'' we proceed as in \cite[Prop. 7.18]{McNeilRisk}. For $a,b \in \mathbb{R}$ and   $U =^d \mathcal{U}([0,1])$ we obtain
\begin{align*}
F(a,b) &= \min( F_X (a), F_Y(b)) = \mathbb{P}( U \leq \min( F_X (a), F_Y(b)) ) \\
&= \mathbb{P}( U \leq F_X(a) , U \leq F_Y (b) ) = \mathbb{P}( F^{\leftarrow}_X (U) \leq a , F^{\leftarrow}_Y (U) \leq b) ,
\end{align*}
where Lemma~\ref{Thm:QuantileLemma} (i) is used in the last step. Finally, we show ``(iii)$\Rightarrow$(i)''. Observe that it is enough to prove that $(X - \tilde{X}) (Y - \tilde{Y}) \geq 0 $ holds $(\mathbb{P} \otimes \mathbb{P})$-a.s. where $(\tilde{X}, \tilde{Y})$ and $(X,Y)$ are i.i.d. random vectors. Hence, take independent $U_1, U_2 =^d \mathcal{U}([0,1])$, which implies that the events
\begin{equation*}
\big\lbrace (X - \tilde{X}) (Y - \tilde{Y}) < 0 \big\rbrace \quad \text{and} \quad \big\lbrace (F^{\leftarrow}_X(U_1) - F^{\leftarrow}_X (U_2)) (F^{\leftarrow}_Y(U_1) - F^{\leftarrow}_Y (U_2)) < 0 \big\rbrace
\end{equation*}
have the same probability. Since $F^{\leftarrow}_X$ and $F^{\leftarrow}_Y$ are by definition increasing, the latter set cannot have positive probability and thus $X$ and $Y$ are comonotonic.
\end{proof}

Using the previous lemma, we now show that quantiles are additive if they correspond to comonotonic random variables. The result and its proof are presented in \cite[Prop. 7.20]{McNeilRisk}. It is then used to show that spectral risk measures are comonotone.

\begin{lemma}  \label{Thm:QuantAdditiveComon}
Fix $\alpha \in (0,1)$ and let $X$ and $Y$ be comonotonic random variables. Then
\begin{equation*}
F^{\leftarrow}_{X + Y} (\alpha) = F^{\leftarrow}_X (\alpha) + F^{\leftarrow}_Y (\alpha).
\end{equation*}
\end{lemma}

\begin{proof}
At first, define the mapping $H : (0,1) \rightarrow \mathbb{R}$, $u \mapsto F^{\leftarrow}_X (u) + F^{\leftarrow}_Y (u)$, which is increasing and left-continuous due to Lemma~\ref{Thm:QuantileLemma} (iv). Moreover, by property (iii) of Lemma~\ref{Thm:ComonotonicEquivalent} we have $(X,Y) =^d ( F^{\leftarrow}_X (U) , F^{\leftarrow}_Y(U) )$ for $U =^d \mathcal{U}([0,1])$. Consequently, part (v) of Lemma~\ref{Thm:QuantileLemma} gives
\begin{align*}
F^{\leftarrow}_{X+Y} (\alpha) = F^{\leftarrow}_{H(U)} (\alpha) = H ( F^{\leftarrow}_U (\alpha)) = H( \alpha ) =   F^{\leftarrow}_X (\alpha) + F^{\leftarrow}_Y (\alpha)
\end{align*}
for any $\alpha \in (0,1)$.
\end{proof}

\begin{theorem}
Every spectral measure of risk $\rho_\phi$ is coherent and comonotone.
\end{theorem}

\begin{proof}
Using the previous lemma and the additivity of the integral, we obtain that all spectral measures of risk are comonotone. Coherence follows from Lemma~\ref{Thm:ConvexCombiRisk} and the coherence of $\mathrm{ES}$, which is shown in Theorem~\ref{Thm:EScoherent}, independently of the results of this subsection.
\end{proof}


\section{Examples of law-invariant risk measures} 
\label{Sec:VaRESandEVaR}

This section presents and discusses three measures of risk which are used in industry or studied in academia. All measures considered are law-invariant, hence they are interpreted as mappings on some class of distribution functions $\mathcal{F}$ and the notation $\rho (F)$ instead of $\rho(X)$ is used if $X$ has distribution function $F$. Moreover, this interpretation makes it possible to check whether the risk measures are elicitable. We start with Value at Risk, the most widely used risk measure in practice, and continue with Expected Shortfall which is proposed as an alternative to VaR. An introduction to Expectile Value at Risk, which recently received much attention in academia, concludes the section.

\subsection{Value at Risk} 
\label{Sec:ValueAtRisk}

For some $\alpha \in (0,1)$, the Value at Risk (VaR) at level $\alpha$ is a threshold such that the probability of the portfolio loss exceeding this threshold is equal or below $\alpha$. Consequently, the suitable mathematical object to represent the risk of a distribution function $F$ is one of its $\alpha$-quantiles, see Definition~\ref{Def:quantilefunction}. However, since quantiles are not always unique and risk measure conventions differ, as noted in Remark~\ref{Rem:DefinitionRisk}, there are different definitions of VaR and one of them is the following.

\begin{definition} \label{Def:ValueatRisk}
Fix $\alpha \in (0,1)$ and let $X$ have distribution $F$. Then
\begin{equation*}
\mathrm{VaR}_\alpha (F) := \mathrm{VaR}_\alpha (X) := - F^{\leftarrow} (\alpha)
\end{equation*}
is called the \textit{Value at Risk} at level $\alpha$ of $X$ (or $F$).
\end{definition} 

\begin{remark}
Typically, $\alpha$ will be some small value, like $1 \% $ or $5 \% $ so that VaR considers the left tail of the distribution function $F$. As mentioned above, other conventions are popular as well. Fissler and Ziegel~\cite{FissZieg} define VaR via $F^{\leftarrow} (\alpha)$ because they use opposite signs (higher values of $\rho$ represent less risk). Artzner et al.~\cite{ArtznerCoherent} use the same sign convention but define $\mathrm{VaR}_\alpha (X)$ via $ - F^{\rightarrow}_X (\alpha) =  F^{\leftarrow}_{-X} (1- \alpha)$, which coincides with Definition~\ref{Def:ValueatRisk} for continuous distributions $F$ (see Lemma~\ref{Thm:AppendixQuantileLR}).
\end{remark}

Value at Risk is widely used in practice for risk management, financial reporting, and computation of capital requirements. In particular, the Basel framework for the regulation of the banking industry as well as the Solvency II Directive for EU insurance regulation use VaR to compute certain capital requirements. For details we refer to McNeil et al.~\cite[Sec. 2.3]{McNeilRisk} and the references therein. We continue by studying the properties of Value at Risk. At first, we consider VaR for normal distributions.

\begin{example}[VaR for normal distributions]  \label{Thm:ExVaRforNormal}
Let $X$ be a random variable with distribution $\mathcal{N} (\mu, \sigma^2)$ and fix $\alpha \in (0,1)$. We denote the quantile function of $\mathcal{N}(0,1)$ by $\Phi^{-1}$ and hence
\begin{equation}  \label{Eqn:VaRforNormal}
\mathrm{VaR}_\alpha (X) = -\mu - \sigma \Phi^{-1} (\alpha) ,
\end{equation}
since the quantile function of $X$ is given by $p \mapsto \mu + \sigma \Phi^{-1}(p)$.
\end{example}

The following result shows that VaR satisfies the properties 1) to 3) of Definition~\ref{Def:CoherentRisk} and is also a comotone measure of risk. Due to the definition of VaR, it is an immediate consequence of previous findings concerning quantiles.

\begin{lemma}  \label{Thm:PropertiesOfVaR}
Fix $\alpha \in (0,1)$. Then the risk measure $\mathrm{VaR}_\alpha$ satisfies monotonicity, translation invariance and positive homogeneity. Moreover, it is comonotone.
\end{lemma}

\begin{proof}
Let $X,Y$ be random variables. Assume $X \leq Y$, then this implies $F_X \geq F_Y$ and hence $F^{\leftarrow}_X \leq F^{\leftarrow}_Y$, which gives $\mathrm{VaR}_\alpha (X) \geq \mathrm{VaR}_\alpha (Y)$. For the other two properties, note that the mappings $x \mapsto cx$ (with $c>0$) as well as $x \mapsto x -c$ are increasing and left-continuous. Consequently, they follow from Lemma~\ref{Thm:QuantileLemma} (v). If $X$ and $Y$ are comonotonic, Lemma~\ref{Thm:QuantAdditiveComon} implies that $\mathrm{VaR}_\alpha$ is comonotone. 
\end{proof}

In light of the previous lemma, only subadditivity or convexity are missing to make VaR a coherent measure of risk. However, VaR is not subadditive in general and thus fails to be a coherent risk measure. The following simple example shows that VaR is not convex and hence cannot be coherent. Similar examples can be found in \cite{ArtznerCoherent}, for instance.  Note that it is possible to construct similar examples for continuous distributions, for example by using independent losses with heavy-tailed distributions. We refer to \cite[Example 2.25]{McNeilRisk}.

\begin{example}  \label{Thm:ExVaRnotConvex}
Define the distribution $P := 0.04 \, \delta_{-1} + 0.96 \, \delta_r$ on $(\mathbb{R}, \mathcal{B}(\mathbb{R}))$ for some $r >0$. Let $X_1$ and $X_2$ be independent random variables having distribution $P$, which implies that $X_1 + X_2$ has distribution $ 0.0016 \, \delta_{-2} + 0.0768 \, \delta_{r-1} + 0.9216 \, \delta_{2r}$. Setting $\alpha = 0.05$ and using positive homogeneity gives
\begin{align*}
\mathrm{VaR}_\alpha (X_1 + X_2) = 1- r > -r = \mathrm{VaR}_\alpha (X_1) = \frac{1}{2} \mathrm{VaR}_\alpha (2 X_1) + \frac{1}{2} \mathrm{VaR}_\alpha (2 X_2),
\end{align*}
which shows that VaR violates convexity.
\end{example}

Although in general VaR is  not a convex measure of risk, there are special distributions such that this holds, for example the (multivariate) normal distribution. This well-known result is stated in the following lemma and is an immediate consequence of Example~\ref{Thm:ExVaRforNormal}.

\begin{lemma}
Fix $\alpha \in (0,\frac{1}{2} ]$ and let $\mathcal{X}$ be a space of random variables such that any pair $X, Y \in \mathcal{X}$ is jointly normally distributed. Then $\mathrm{VaR}_\alpha$ defined on $\mathcal{X}$ is a coherent measure of risk.
\end{lemma}

\begin{proof}
Due to Lemma~\ref{Thm:PropertiesOfVaR}, only subadditivity needs to be shown. To this end, fix some $X,Y \in \mathcal{X}$ which are jointly normal by assumption. It is well-known that then $X + Y$ is either normally distributed or constant almost surely. For the first case the Cauchy-Schwarz inequality gives
\begin{align*}
\Var (X + Y) &\leq \Var (X) + \Var (Y) + 2 \vert \Cov (X, Y) \vert \\
&\leq \Var (X) + \Var(Y) + 2 \sqrt{ \Var (X)} \sqrt{\Var (Y)} ,
\end{align*}
implying $ \sigma_{X + Y} \leq \sigma_X + \sigma_Y$ for the standard deviations. Using Equation~(\ref{Eqn:VaRforNormal}) it follows that
\begin{align*}
\mathrm{VaR}_\alpha (X + Y) &= - \mu_{X} - \mu_{Y} - \Phi^{-1} (\alpha) \sigma_{X + Y} \\
&\leq - \mu_X - \mu_Y - \Phi^{-1} (\alpha) ( \sigma_X + \sigma_Y) = \mathrm{VaR}_\alpha (X) + \mathrm{VaR}_\alpha (Y) ,
\end{align*}
since $\Phi^{-1} (\alpha) \leq 0$ for $\alpha \leq \frac{1}{2}$. Finally, in the case where $X + Y = c$ a.s. for some $c \in \mathbb{R}$, we have $\mu_X = - \mu_Y + c$ and $\sigma_{X + Y} = 0$, which gives
\begin{equation*}
\mathrm{VaR}_\alpha (X + Y) = c \leq c - \Phi^{-1} (\alpha) ( \sigma_X + \sigma_Y) = \mathrm{VaR}_\alpha (X) + \mathrm{VaR}_\alpha (Y) ,
\end{equation*}
using $\Phi^{-1} (\alpha) \leq 0$ again.
\end{proof}

This result can be extended to hold for random variables which are jointly elliptically distributed. A detailed treatment can again be found in \cite[Thm. 8.28]{McNeilRisk}. The next result is an immediate consequence of Theorem~\ref{Thm:QuantilesElicitable} and Remark~\ref{Rem:RevelationPrinciple}.

\begin{theorem}
For $\alpha \in (0,1)$, consider the $\mathrm{VaR}$ functional on a set of distribution functions $\mathcal{F}$, i.e. $\mathrm{VaR}_\alpha : \mathcal{F} \rightarrow \mathbb{R}$. If all elements in $\mathcal{F}$ have unique $\alpha$-quantiles, it follows that $\mathrm{VaR}_\alpha$ is elicitable.
\end{theorem}

To complete this subsection, we consider the acceptance set of VaR. For simplicity, assume that all random variables in $\mathcal{X}$ have continuous, strictly increasing distribution functions and fix $\alpha \in (0, 1)$. For $X =^d F$ we have by definition $\mathrm{VaR}_\alpha (X) \leq 0 \Leftrightarrow F^{\leftarrow} (\alpha) \geq 0$ and this is equivalent to $\alpha (1- F(0) ) \geq (1- \alpha) F(0)$ due to monotonicity of $F$. Recalling Definition~\ref{Def:AcceptanceSet}, it follows that
\begin{equation*}
\mathcal{A}_{ \mathrm{VaR}_\alpha } := \left\lbrace X \in \mathcal{X} \, \Big| \, \frac{ \mathbb{P}( X>0 ) }{ \mathbb{P}( X \leq 0) } \geq \frac{1 - \alpha}{\alpha} \right\rbrace
\end{equation*}
is the acceptance set of VaR. Economically speaking, a position is acceptable in terms of Value at Risk if the probability of gain exceeds the probability of loss by a certain factor. \\
\par
This subsection shows that Value at Risk has some desirable risk measure properties and is moreover elicitable. However, it also has some disadvantages as discussed in \cite{ArtznerCoherent} among others. It is not a coherent and not even a convex measure of risk due to its lack of convexity. This lack can cause serious problems in practice, since it is impossible to guarantee that the risk measured by VaR decreases if a position is diversified. Another serious disadvantage is illustrated by its acceptance set. It shows that Value at Risk only controls the probability of occurring losses (versus the probability of occurring gains). All losses which are worse than VaR are irrelevant and thus remain undetected, which may have disastrous consequences. Both of these well-known drawbacks of Value at Risk led to the introduction of Expected Shortfall, which is discussed in the following subsection.

\subsection{Expected Shortfall} 
\label{Sec:ExpectedShortfall}

Although the concept of Expected Shortfall (ES) has been given many names in the literature, it is always based on the expectation of the tail of a random variable. A similar measure called `tail conditional expectation' is mentioned by Artzner et al.~\cite{ArtznerCoherent} and other names include `Conditional Value at Risk' or `Average Value at Risk' in Rockafellar and Uryasev~\cite{Rockafellar} or F\"ollmer and Schied~\cite{FoellmerSchied}, respectively. Further development is due to Acerbi and Tasche~\cite{AcerbiTascheLong,AcerbiTascheShort} among others. This subsection mainly presents their results.

\begin{definition}  \label{Def:ExpectedShortfall}
Fix $\alpha \in (0,1)$ and let $X$ have distribution function $F$ and finite first moment. Then
\begin{equation*}
\mathrm{ES}_\alpha (F) := \mathrm{ES}_\alpha (X) := - \frac{1}{\alpha} \int_{0}^{\alpha} F^{\leftarrow} (u) \, \mathrm{d}u
\end{equation*}
is called the \textit{Expected Shortfall} at level $\alpha$ of $X$ (or $F$).
\end{definition}

\begin{remark}
An immediate consequence of the definition is the representation
\begin{equation}  \label{Eqn:ESaverageVaR}
\mathrm{ES}_\alpha (X) = \frac{1}{\alpha} \int_{0}^{\alpha} \mathrm{VaR}_s (X) \, \mathrm{d}s ,
\end{equation}
which is the reason why ES is sometimes called `Average Value at Risk'. Note that the same identity is obtained if VaR is defined using $F^{\rightarrow}$ instead of $F^{\leftarrow}$, since both quantile functions coincide for all but a countable number of points. Finally, for any $\alpha \in (0,1)$ we obtain the inequality
\begin{equation}  \label{Eqn:ESVaRInequality}
\mathrm{ES}_\alpha (X) \geq \mathrm{VaR}_\alpha (X) ,
\end{equation}
which follows from (\ref{Eqn:ESaverageVaR}) and the monotonicity of VaR.
\end{remark}

\begin{remark}
As clarified in \cite{AcerbiTascheShort}, one reason why Expected Shortfall is considered to be more reasonable than Value at Risk is given by the type of question the risk measures answer. While VaR answers the question `What is the minimum loss incurred in the $\alpha 100 \%$ worst cases of the portfolio?', ES is an answer to `What is the expected loss incurred in the $\alpha 100 \%$ worst cases of the portfolio?'. Like VaR, it makes sense to consider ES for small values of $\alpha$.
\end{remark}

Definition~\ref{Def:ExpectedShortfall} is similar to the one used by Fissler and Ziegel~\cite{FissZieg}, however Acerbi and Tasche~\cite{AcerbiTascheLong,AcerbiTascheShort} use Representation~(\ref{Eqn:ESAcerbiDefinition}). The next lemma shows that both definitions are in fact equivalent. The result and its proof can be found in \cite{AcerbiTascheLong} as well as in McNeil et al.~\cite[Prop. 8.13]{McNeilRisk}.

\begin{lemma}
For $\alpha \in (0,1)$ and an integrable random variable $X$ we have
\begin{equation}  \label{Eqn:ESAcerbiDefinition}
\mathrm{ES}_\alpha (X) = - \frac{1}{\alpha} \big[ \mathbb{E} X \mathbbm{1}_{ \lbrace X \leq F^{\leftarrow}_X (\alpha) \rbrace } + F^{\leftarrow}_X (\alpha) \big( \alpha - \mathbb{P}( X \leq F^{\leftarrow}_X (\alpha) ) \big) \big].
\end{equation}
\end{lemma}

\begin{proof}
Fix $\alpha \in (0,1)$ and let $X$ be some integrable random variable. Starting with the definition of ES, let $U$ be a random variable on $(\Omega, \mathscr{A}, \mathbb{P})$ with $U =^d \mathcal{U}( [0,1])$ under $\mathbb{P}$. This implies $F^{\leftarrow}_X (U) =^d X$ under $\mathbb{P}$ (see for instance \cite[Lemma A.23]{FoellmerSchied}) and hence we obtain
\begin{equation*}  
\mathrm{ES}_\alpha (X) = - \frac{1}{\alpha} \int_{0}^{\alpha} F^{\leftarrow}_X (u) \, \mathrm{d}u = - \frac{1}{\alpha} \mathbb{E} F^{\leftarrow}_X (U) \mathbbm{1}_{\lbrace U \leq \alpha \rbrace}.
\end{equation*}
Now consider the set $\lbrace U \leq \alpha \rbrace$. Part (i) of Lemma~\ref{Thm:QuantileLemma} implies that $F^{\leftarrow}_X (U) \leq F^{\leftarrow}_X (\alpha)$ if and only if $U \leq F_X( F^{\leftarrow}_X (\alpha))$, implying the decomposition
\begin{equation}  \label{Eqn:ESsetdecomposition}
\lbrace F^{\leftarrow}_X (U) \leq F^{\leftarrow}_X (\alpha) \rbrace = \lbrace U \leq \alpha \rbrace \uplus \lbrace \alpha < U \leq F_X ( F^{\leftarrow}_X (\alpha) ) \rbrace .
\end{equation}
Since the sets on the right-hand side are disjoint, the identity carries over to indicator functions of the sets. This gives
\begin{align*}
\mathbb{E} F^{\leftarrow}_X (U) \mathbbm{1}_{\lbrace U \leq \alpha \rbrace} &= \mathbb{E}  F^{\leftarrow}_X (U) \mathbbm{1}_{ \lbrace X \leq F^{\leftarrow}_X (\alpha) \rbrace } - \mathbb{E}  F^{\leftarrow}_X (U) \mathbbm{1}_{ \lbrace \alpha < U \leq F_X ( F^{\leftarrow}_X (\alpha) ) \rbrace } \\
&= \mathbb{E}  X \mathbbm{1}_{ \lbrace X \leq F^{\leftarrow}_X (\alpha) \rbrace } - F^{\leftarrow}_X (\alpha) \mathbb{P} (\alpha < U \leq F_X ( F^{\leftarrow}_X (\alpha) ) ) \\
&= \mathbb{E}  X \mathbbm{1}_{ \lbrace X \leq F^{\leftarrow}_X (\alpha) \rbrace } + F^{\leftarrow}_X (\alpha) \big[ \alpha - F_X (F^{\leftarrow}_X (\alpha) ) \big],
\end{align*}
since $F^{\leftarrow}_X (\alpha) = F^{\leftarrow}_X (U)$ must hold on the set $\lbrace \alpha < U \leq F_X ( F^{\leftarrow}_X (\alpha) ) \rbrace$.
\end{proof}

\begin{remark}  \label{Rem:ESconditionalVaR}
If $X$ is a continuously distributed random variable, Lemma~\ref{Thm:QuantileLemma} (ii) implies $\alpha = F_X ( F^{\leftarrow}_X (\alpha))$ and hence the left-hand side of (\ref{Eqn:ESsetdecomposition}) reduces to the set $\lbrace U \leq \alpha \rbrace$. As a consequence, Representation~(\ref{Eqn:ESAcerbiDefinition}) simplifies to
\begin{equation}  \label{Eqn:ESconditinoalVaR}
\mathrm{ES}_\alpha (X) = - \frac{1}{\alpha} \mathbb{E} X \mathbbm{1}_{ \lbrace X \leq F^{\leftarrow}_X (\alpha) \rbrace } =  \mathbb{E} [ -X \mid
- X \geq \mathrm{VaR}_\alpha (X) ] ,
\end{equation}
which justifies the name `tail conditional expectation' used by Acerbi~\cite{AcerbiSpectral} as well as the term `Conditional Value at Risk' (CVaR). Note that some authors distinguish between ES and CVaR while others use both terms interchangeably. If a distinction is made, the right-hand side of (\ref{Eqn:ESconditinoalVaR}) is used to define CVaR. However, defining CVaR this way does not produce a coherent measure of risk, unless all distributions are continuous. For a counterexample see for instance Acerbi and Tasche~\cite[Example 5.4]{AcerbiTascheLong}.
\end{remark}

The following theorem shows that ES can be regarded as an improvement of VaR, in the sense that it is a coherent measure of risk. The result and its proof can be found in \cite{AcerbiTascheLong} and \cite[Example 2.26]{McNeilRisk}. We provide our own proof, which uses the scoring functions for quantiles from Theorem~\ref{Thm:QuantilesElicitable} to show subadditivity.

\begin{theorem}  \label{Thm:EScoherent}
Fix $\alpha \in (0,1)$. Then $\mathrm{ES}_\alpha$ is a coherent and comonotone measure of risk.
\end{theorem}

\begin{proof}
In Lemma~\ref{Thm:PropertiesOfVaR} it is shown that VaR satisfies monotonicity, positive homogeneity, translation invariance and comonotonicity. Due to the Identity~(\ref{Eqn:ESaverageVaR}), these properties carry over to ES. It remains to be shown that ES is subadditive. To this end, fix $\alpha \in (0,1)$ and let $X$ and $Y$ be integrable random variables on $(\Omega, \mathscr{A}, \mathbb{P})$. Moreover, observe that for any $A \in \mathscr{A}$ and $c \in \mathbb{R}$ we have
\begin{equation}  \label{Eqn:ESsetInequality}
\mathbb{E} (c - X) \mathbbm{1}_A \leq \mathbb{E} (c -X) \mathbbm{1}_{ \lbrace X \leq c \rbrace } ,
\end{equation}
and the same holds for $X$ replaced by $Y$. Moreover, define the scoring function $S(x,y) := (\mathbbm{1}_{\lbrace y \leq x \rbrace} - \alpha ) (x - y )$, which is consistent for $F^{\leftarrow}(\alpha)$ due to Theorem~\ref{Thm:QuantilesElicitable}. Using Representation~(\ref{Eqn:ESAcerbiDefinition}) for $\mathrm{ES}_\alpha$ gives
\begin{align*}
\mathrm{ES}_\alpha (X + Y) + \mathbb{E} (X + Y) &= - \frac{1}{\alpha} \mathbb{E} ( X + Y - F^{\leftarrow}_{X+Y} (\alpha) ) ( \mathbbm{1}_{ \lbrace X + Y \leq F^{\leftarrow}_{X+Y} (\alpha) \rbrace } - \alpha ) \\
&= \frac{1}{\alpha} \bar{S} ( F^{\leftarrow}_{X+Y} (\alpha) , F_{X + Y} ) \\
&\leq \frac{1}{\alpha} \bar{S} ( F^{\leftarrow}_{X} (\alpha) + F^{\leftarrow}_{Y} (\alpha) , F_{X + Y} )  \\
&= \frac{1}{\alpha} \mathbb{E} (F^{\leftarrow}_X (\alpha) - X) ( \mathbbm{1}_{ \lbrace X + Y \leq F^{\leftarrow}_{X} (\alpha) + F^{\leftarrow}_{Y} (\alpha) \rbrace } - \alpha ) \\
&\quad+  \frac{1}{\alpha} \mathbb{E} (F^{\leftarrow}_Y (\alpha) - Y) ( \mathbbm{1}_{ \lbrace X + Y \leq F^{\leftarrow}_{X} (\alpha) + F^{\leftarrow}_{Y} (\alpha) \rbrace } - \alpha ) \\
&\leq \frac{1}{\alpha} \mathbb{E} (F^{\leftarrow}_X (\alpha) - X) ( \mathbbm{1}_{ \lbrace X \leq F^{\leftarrow}_{X} (\alpha)  \rbrace } - \alpha ) \\
&\quad+  \frac{1}{\alpha} \mathbb{E} (F^{\leftarrow}_Y (\alpha) - Y) ( \mathbbm{1}_{ \lbrace  Y \leq F^{\leftarrow}_{Y} (\alpha) \rbrace } - \alpha ) \\
&= \mathrm{ES}_\alpha (X) + \mathrm{ES}_\alpha (Y) + \mathbb{E} (X + Y),
\end{align*}
where the first inequality uses consistency of $S$ and the second inequality is an application of (\ref{Eqn:ESsetInequality}). Subtracting $\mathbb{E}(X + Y)$ shows that ES is subadditive.
\end{proof}

The next Example is part of McNeil et al.~\cite[Example 2.14]{McNeilRisk}.

\begin{example}[ES for normal distributions]
Similar to Example~\ref{Thm:ExVaRforNormal}, we compute the Expected Shortfall at level $\alpha \in (0,1)$ of a random variable $X =^d \mathcal{N}(\mu, \sigma^2)$. Using either Equation~(\ref{Eqn:ESAcerbiDefinition}) or Remark~\ref{Rem:ESconditionalVaR} we obtain
\begin{equation*}
\mathrm{ES}_\alpha (X) = - \mu  - \sigma \frac{1}{\alpha} \mathbb{E} \left( \frac{X - \mu}{\sigma} \mathbbm{1}_{ \lbrace (X - \mu)/\sigma \leq (F^{\leftarrow}_X (\alpha) - \mu)/\sigma \rbrace } \right),
\end{equation*}
which together with Example~\ref{Thm:ExVaRforNormal} shows that $\mathrm{ES}_\alpha (X) = - \mu + \sigma  \mathrm{ES}_\alpha (N)$ for a random variable $N =^d \mathcal{N}(0,1)$. Now denote the quantile function of $N$ via $\Phi^{-1}$ and use (\ref{Eqn:ESAcerbiDefinition}) to calculate
\begin{align*}
\mathrm{ES}_\alpha (N) = - \frac{1}{\alpha} \int_{-\infty}^{\Phi^{-1} (\alpha)} x \varphi (x) \, \mathrm{d}x = \left. \frac{1}{\alpha} \varphi (x) \, \right|^{x=\Phi^{-1} (\alpha)}_{x= - \infty} = \frac{\varphi ( \Phi^{-1} (\alpha) )}{\alpha} ,
\end{align*}
which finally implies
\begin{equation*}
\mathrm{ES}_\alpha (X) = - \mu + \sigma \frac{\varphi ( \Phi^{-1} (\alpha) )}{\alpha} .
\end{equation*}
\end{example}

Considering the previous results, it seems that Expected Shortfall is preferable compared to Value at Risk. However, ES also has some drawbacks. Most obviously, the distribution of the considered position is required to have a finite first moment, so ES is not universally applicable. Secondly, the nice properties of ES are achieved at the cost of higher mathematical complexity compared to VaR. Finally, when defined as a functional on a `rich enough' set of distribution functions, ES fails to be elicitable. All proofs of this fact show that the necessary condition of convex level sets as stated in Proposition~\ref{Thm:LevelSetCond} is not satisfied. Concrete counterexamples are provided by Gneiting~\cite[Thm. 11]{GneitingPoints} using discrete distributions and by Weber~\cite[Example 3.10]{WeberConsistency} using a mixture of discrete and continuous distributions. In the next example only continuous distributions are used to construct a situation where a level set of ES is not convex.

\begin{example}[Non-elicitability of Expected Shortfall]  \label{Thm:ExESnotElicitable}
Fix $\alpha < 2/3$ and define two distribution functions $F_1, F_2$ via their densities $f_1, f_2$, which are given by
\begin{align*}
f_1 (x) &:= \frac{\alpha}{2} \mathbbm{1}_{ [-2,-1] } (x) + \frac{\alpha}{2}  \mathbbm{1}_{ [1,2] } (x) + (1-\alpha) \mathbbm{1}_{ (2,3] } (x) , \\
f_2 (x) &:= \frac{3 \alpha}{2} \mathbbm{1}_{ [-1/2, 1 ] } (x) + \big( 1 - \frac{3\alpha}{2} \big) \mathbbm{1}_{ (1,2] } (x).
\end{align*}
This definition immediately implies $F^{\leftarrow}_1 (\alpha) = 2$ and $F^{\leftarrow}_2 (\alpha) = \frac{1}{2}$. Now let $\mathcal{F}$ be a convex class of continuous distribution functions with finite first moments which contains $F_1$ and $F_2$. We consider $\mathrm{ES}_\alpha$ as a functional on $\mathcal{F}$ and use the simplified version of ES as stated in Remark~\ref{Rem:ESconditionalVaR}, since all members of $\mathcal{F}$ are continuous. This implies $ \mathrm{ES}_\alpha (F_1) = \mathrm{ES}_\alpha (F_2) = 0$. If we define $F:= \frac{1}{2} (F_1 + F_2)$ we have $F^{\leftarrow} (\alpha) = 1$ and $\mathrm{ES}_\alpha$ can be computed as the mean of two scaled uniform distributions on the intervals $[-2,-1]$ and $[-1/2,1]$. In detail we obtain
\begin{equation*}
\mathrm{ES}_\alpha (F) = - \frac{1}{\alpha} \left( \frac{\alpha}{4} \left( \frac{-2-1}{2} \right) + \frac{3 \alpha}{4} \left( \frac{1 - 1/2}{2} \right) \right) = \frac{3}{8} - \frac{3}{16} = \frac{3}{16} ,
\end{equation*}
showing that $\mathrm{ES}_\alpha ( \frac{1}{2} (F_1 + F_2) ) \neq \frac{1}{2} ( \mathrm{ES}_\alpha (F_1) + \mathrm{ES}_\alpha (F_2) )$. Hence,  $\mathrm{ES}_\alpha$ cannot be elicitable relative to the class $\mathcal{F}$ due to Proposition~\ref{Thm:LevelSetCond}. As noted in Remark~\ref{Rem:FnotElicitable}, this does not imply that ES is non-elicitable for all choices of $\mathcal{F}$. For instance, $\mathrm{ES}_\alpha$ is elicitable relative to some class $\mathcal{F}$ if all members of $\mathcal{F}$ have the same $\alpha$-quantile, see also Lemma~\ref{Thm:SpectralCondElicit}.
\end{example}

In addition to the previous example and the mentioned references, Ziegel~\cite{ZiegelCoherence} demonstrates that all spectral measures of risk given in Definition~\ref{Def:SpectralRiskMeasure} fail to be elicitable.

\subsection{Expectile Value at Risk} 
\label{Sec:ExpectileVaR}

Since every random variable with finite first moment has a unique $\tau$-expectile, it is possible to consider expectiles as measures of risk. Guided by the insights on Value at Risk and the non-elicitability of spectral risk measures, Ziegel~\cite{ZiegelCoherence} poses the question whether expectiles are the only coherent elicitable risk measures. The affirmative answer is given by Steinwart et al.~\cite{SteinwartPasinWilliam} and Bellini and Bignozzi~\cite{BelliniBignozzi} and due to this result, expectiles have received increasing attention in the literature. This subsection gives a short introduction which is mainly based on Bellini and Bernardino~\cite{BelliniBernardino}.

\begin{definition}  \label{Def:ExpectileVaR}
Let $X$ be a random variable with distribution function $F$ having finite first moment. For $\tau \in (0,1)$, let $e_\tau (F)$ be the $\tau$-expectile as in Definition~\ref{Def:Expectiles}. Then
\begin{equation*}
\mathrm{EVaR}_\tau (X) := \mathrm{EVaR}_\tau (F) := - e_\tau (F)
\end{equation*}
is called the \textit{Expectile Value at Risk} at level $\tau$ of $X$ (or $F$).
\end{definition} 

\begin{remark}
Similar to Value at Risk and Expected Shortfall, it makes sense to consider the Expectile Value at Risk for small values of $\tau \in (0,1)$. The reason for this is that a small $\tau$ increases the value of the right-hand side of the expectile Identity~(\ref{Eqn:ExpectileDefinition}) and thus gives more weight to the left tail of the distribution under consideration.
\end{remark}

The following theorem shows that EVaR is a sensible risk measure choice, at least from a theoretical point of view, as stated and proved in \cite{BelliniBignozzi} and \cite[Prop. 8.25]{McNeilRisk}.

\begin{theorem}  \label{Thm:EVaRiscoherent}
Fix $\tau \in (0, \frac{1}{2}]$. Then $\mathrm{EVaR}_\tau$ is a coherent measure of risk.
\end{theorem}

\begin{proof}
In the following, fix $\tau \in (0, \frac{1}{2}]$ and let $X,Y$ be some integrable random variables. At first, note that positive homogeneity and translation invariance follow from the first part of Lemma~\ref{Thm:ExpectileProperties}. For monotonicity, assume that $X \leq Y$ holds, which implies the inequality
\begin{align*}
\tau \mathbb{E} ( X - e_\tau (Y) )^+ \leq \tau \mathbb{E} (Y - e_\tau (Y) )^+ &= (1- \tau) \mathbb{E} ( e_\tau (Y) - Y)^+  \\
&\leq (1- \tau) \mathbb{E} ( e_\tau (Y) - X)^+.
\end{align*}
From the proof of Lemma~\ref{Thm:ExpectileExists} it is known that the left-hand side of the expectile Identity~(\ref{Eqn:ExpectileDefinition}) is decreasing in $x$ while the right-hand side is increasing in $x$. The previous calculation shows that the left term is smaller than the right term if $x= e_\tau (Y)$ is plugged in. Consequently, we must have $e_\tau (X) \leq e_\tau (Y)$ and thus $\mathrm{EVaR}_\tau (X) \geq \mathrm{EVaR}_\tau (Y)$. In order to prove subadditivity, note first that the identity $(x -y)^+ - (y - x)^+ = x - y$ holds for any $x,y \in \mathbb{R}$. As a consequence, $\tau$-expectiles satisfy the relation
\begingroup
\allowdisplaybreaks[0]	
\begin{align}
(1- \tau) \mathbb{E} ( e_\tau (X)  - X)^+ &= \tau \mathbb{E} (X - e_\tau (X) )^+ \nonumber \\
&= \tau \mathbb{E} ( e_\tau (X) - X )^+ + \tau ( \mathbb{E} X - e_\tau (X) )  \nonumber \\
\Leftrightarrow \qquad \tau ( \mathbb{E} X - e_\tau (X) ) &= (1- 2 \tau) \mathbb{E} ( e_\tau (X) - X)^+  ,  \label{Eqn:ExpectileModIdentity} 
\end{align}
\endgroup
and the same holds for $X$ replaced by $Y$. Now take the expectile Identity~(\ref{Eqn:ExpectileDefinition}) for  $X + Y$ and plug in $x = e_\tau (X) + e_\tau (Y)$. The difference between the left- and right-hand side is then given by
\begin{align*}
\Delta \mathsf{E} := \tau \mathbb{E} ( X + Y - e_\tau (X) - e_\tau (Y) )^+ - (1-\tau) \mathbb{E} ( e_\tau (X) + e_\tau (Y) - X - Y)^+ \\
= (2 \tau - 1) \mathbb{E} (e_\tau (X) + e_\tau (Y) - X - Y)^+ + \tau \mathbb{E} ( X + Y - e_\tau (X) - e_\tau (Y) ).
\end{align*}
Since $x \mapsto x^+$ is a subadditive function and $\tau \leq \frac{1}{2}$ we obtain
\begin{align}  \label{Eqn:EVaRsubEst}
\Delta \mathsf{E} \geq (2 \tau - 1) &\big[ \mathbb{E}(e_\tau (X) - X)^+ + \mathbb{E}( e_\tau (Y) - Y)^+ \big]  \\
+ \tau &\big[ \mathbb{E} X - e_\tau (X)  + \mathbb{E} Y - e_\tau (Y)  \big] = 0 \nonumber ,
\end{align}
where Equation~(\ref{Eqn:ExpectileModIdentity}) is used in the last step. For $\tau \leq \frac{1}{2}$ this implies
\begin{equation*}
\tau \mathbb{E} ( X + Y - e_\tau (X) - e_\tau (Y) )^+  \geq (1-\tau) \mathbb{E} ( e_\tau (X) + e_\tau (Y) - X - Y)^+  ,
\end{equation*}
and we follow the same arguments as were used to show monotonicity in the beginning of the proof. This gives $e_\tau (X) + e_\tau(Y) \leq e_\tau (X + Y)$, which implies subadditivity of $\mathrm{EVaR}_\tau$.
\end{proof}

\begin{remark}[EVaR is not comonotone]  \label{Rem:EVaRnotcomon}
Note that EVaR is only a coherent but not a comonotone measure of risk. This is shown by Delbaen~\cite[Remark 6]{DelbaenExpectiles} and also remarked by Acerbi and Szekely~\cite[Sec. 3.2]{AcerbiSzekely}. To understand why EVaR lacks this property, we take a closer look at the proof of Theorem~\ref{Thm:EVaRiscoherent}. More precisely, we consider under which conditions on $X$ and $Y$ the estimate in (\ref{Eqn:EVaRsubEst}) is strict. To this end, note that the subadditivity of $x \mapsto x^+$ is used in order to show (\ref{Eqn:EVaRsubEst}). Moreover, observe that for $a,b \neq 0$ we have $(a+ b)^+ < a^+ + b^+$ if and only if $a$ and $b$ have different sign. Consequently, for a strict inequality in (\ref{Eqn:EVaRsubEst}) it is sufficient to find random variables $X$ and $Y$ which have different sign on a set with positive probability (by Lemma~\ref{Thm:ExpectileProperties} (i) we assume w.l.o.g. that $e_\tau (X)$ and $e_\tau (Y)$ are zero). The choice $N =^d \mathcal{N}(0,1)$, $X := N$ and $Y := \exp (N)$ satisfies this condition and $X$ and $Y$ are comonotonic since they are increasing functions of the same random variable. Finally, a strict inequality in (\ref{Eqn:EVaRsubEst}) implies $e_\tau (X) + e_\tau(Y) < e_\tau (X + Y)$, but since $X$ and $Y$ are comonotonic this shows that EVaR cannot be a comonotone risk measure.
\end{remark}

The next theorem states that Expectile Value at Risk is elicitable. It is an immediate consequence of Theorem~\ref{Thm:ExpectileElicitable} and Remark~\ref{Rem:RevelationPrinciple}.

\begin{theorem}
Fix $\tau \in (0,1)$ and let  $\mathcal{F}$  be a set of distribution functions with finite first moments. If we consider $\mathrm{EVaR}$ as a functional on $\mathcal{F}$, i.e. $\mathrm{EVaR}_\tau : \mathcal{F} \rightarrow \mathbb{R}$, then it is elicitable.
\end{theorem}

Finally, in order to interpret the risk measure EVaR, we take a look at its acceptance set and compare this set to the acceptable positions for VaR (which is also done in \cite{BelliniBernardino}). Let $X$ be an integrable random variable and $\tau \in (0, \frac{1}{2}]$. Recall Definition~\ref{Def:AcceptanceSet} and the familiar fact that the left-hand side of the expectile Identity~(\ref{Eqn:ExpectileDefinition}) is decreasing in $x$. Therefore, we have $\mathrm{EVaR}_\tau (X) \leq 0 \Leftrightarrow e_\tau (X) \geq 0$ and this is equivalent to $ \tau \mathbb{E} X^+ \geq (1- \tau) \mathbb{E}(-X)^+$. This shows that
\begin{equation*}
\mathcal{A}_{ \mathrm{EVaR}_\tau } = \left\lbrace X \in \mathcal{L}^1 \, \Big| \, \frac{\mathbb{E} X^+}{- \mathbb{E}X^-} \geq \frac{1- \tau}{\tau} \right\rbrace
\end{equation*}
is the acceptance set of EVaR. Hence, an intuitive interpretation is the following: A position is acceptable in terms of EVaR as long as the ratio of expected gains and expected losses exceeds a certain threshold.

\subsection{Comparison of the different risk measures} 

Due to the lack of coherence of Value at Risk, Expected Shortfall presents itself as a coherent alternative, although only for integrable random variables. If the elicitability of risk measures is also considered, VaR becomes attractive again, since ES fails to be elicitable. However, one is not forced to make a choice between elicitability and coherence, since the Expectile Value at Risk represents a third alternative satisfying both properties. Moreover, as argued in \cite{BelliniBernardino}, EVaR behaves like VaR and ES for real-world data and its acceptance set allows for an intuitive interpretation. All things considered, it seems reasonable to choose EVaR as a replacement for VaR. However, this discussion takes a new turn if the result of Fissler and Ziegel~\cite{FissZieg} is taken into account, who are able to show that ES is jointly elicitable with VaR. Although joint elicitability is slightly more complex, this eliminates one of the biggest advantages of EVaR. Additionally, as discussed in Remark~\ref{Rem:EVaRnotcomon}, EVaR fails to be a comonotone measure of risk which can lead to serious allocation problems (see also Remark~\ref{Rem:ComonotoneInterpret}). 
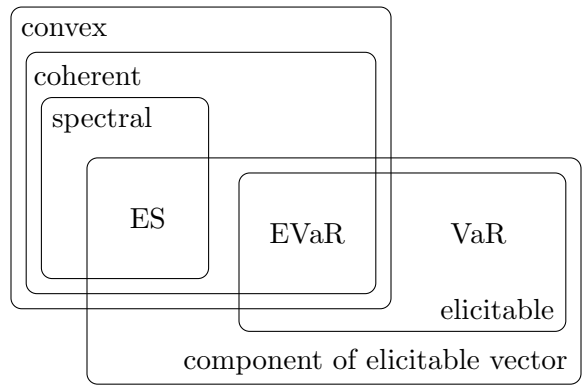
\begin{figure}[h]
\centering
\begin{tikzpicture}[rounded corners=4pt]
\draw (0,0) rectangle (5,-4);
\draw (0.7,-0.3) node {convex};
\draw (0.2,-0.6) rectangle (4.8,-3.8);
\draw (1.01,-0.9) node {coherent};
\draw (0.4,-1.2) rectangle (2.6,-3.6);
\draw (1.2,-1.5) node {spectral};
\draw (3,-2.2) rectangle (7.3,-4.3);
\draw (6.4,-4) node {elicitable};
\draw (1,-2) rectangle (7.5,-5);
\draw (4.8,-4.7) node {component of elicitable vector};
\draw (1.8,-2.8) node {ES};
\draw (3.9,-3) node {EVaR};
\draw (6.15,-3) node {VaR};
\end{tikzpicture}
\caption{Relationship between elicitability and selected risk measure properties, together with the three risk measures Expected Shortfall (ES), Value at Risk (VaR), and Expectile Value at Risk (EVaR).}
\label{Fig:RiskMeasureRelations}
\end{figure}
In contrast, ES is a comonotone risk measure. From a theoretical point of view, we thus conclude that ES is the most convincing (law-invariant) measure of risk. Naturally, spectral measures of risk with a discrete spectrum compare equally well. This conclusion is based on the joint elicitability of spectral risk measures, which is proved in the next section. Figure~\ref{Fig:RiskMeasureRelations} provides a graphical overview of the properties and risk measures discussed in this section. For a detailed discussion of the advantages and drawbacks of the presented risk measures, also including further aspects, we refer to Emmer et al.~\cite{EmmerKratzTasche}.

\section{Joint elicitability of spectral measures of risk} 

This section proves joint elicitability of spectral measures of risk and quantiles, which in particular implies that the pair (VaR, ES) is elicitable. A first approach to prove the latter result is found in Acerbi and Szekely~\cite[Sec. 3.3]{AcerbiSzekely}. The authors assume all $F \in \mathcal{F}$ to be continuous and use Identity~(\ref{Eqn:ESconditinoalVaR}) to construct an identification function for (VaR, ES), from which they are then able to construct a family of scoring functions. However, they only show that the integrated score has a local minimum at (VaR, ES) which is not enough to prove strict consistency. Besides, they assume that there is a $W \in \mathbb{R}$ such that $\mathrm{ES}_\alpha (F) < W \mathrm{VaR}_\alpha (F)$ is satisfied for all $F \in \mathcal{F}$.

A rigorous proof of the elicitability of (VaR, ES) without such an assumption is done by Fissler and Ziegel~\cite{FissZieg}, who in fact show an even stronger result. Loosely speaking, they demonstrate that every spectral risk measure having a spectral measure with finite support is jointly elicitable together with all quantiles it uses. We begin with assumptions and notation necessary to prove the result. \\
\par
For the rest of this section, let $\mathcal{F}$ be a class of continuous distribution functions having finite first moments. Moreover, let $\mathsf{O} \subseteq \mathbb{R}$ be a set containing the support of all corresponding random variables. The representation of spectral measures of risk as given in Lemma~\ref{Thm:SpectralRiskAlternativeRep} leads to the following definition.

\begin{definition}  \label{Def:SpectralRiskFunctional}
Fix $k \in \mathbb{N}$, $k > 1$ and let $p_i, q_i \in (0, 1]$ for $i=1, \ldots, k-1$ be such that the $q_i$ are pairwise distinct and the $p_i$ satisfy $\sum_{i=1}^{k-1} p_i = 1$. Then 
\begin{equation*}
T_k : \mathcal{F} \rightarrow \mathbb{R}, \quad F \mapsto T_k (F) := - \sum_{i=1}^{k-1} \mathrm{ES}_{q_i} (F) p_i
\end{equation*}
is called \textit{spectral risk measure functional} of order $k$, determined by $(p_i, q_i)_{i=1, \ldots, k-1}$.
\end{definition}

If $q_i =1$ for some $i$, the convention $\mathrm{ES}_1 (X) := - \mathbb{E} X$ is used. Moreover, the elements in $\mathcal{F}$ are continuous, hence the representation of ES given in Remark~\ref{Rem:ESconditionalVaR} leads to
\begin{equation}  \label{Eqn:ESforContinuous}
T_k (F) = \sum_{i=1}^{k-1} \frac{p_i}{q_i} \int_{-\infty}^{ F^{\leftarrow} (q_i) } y \, \mathrm{d}F(y)
\end{equation}
for any $F \in \mathcal{F}$. Finally, note that the spectral risk measure functional represents a negative spectral measure of risk. The reason lies in the different sign conventions for risk measures used in \cite{FissZieg} and this work, see also Remark~\ref{Rem:DefinitionRisk}. This is not a problem, since it is possible to use the revelation principle to adapt the result to the risk measure convention used here, as done in Corollary~\ref{Thm:JointElicitVaRES}. The following theorem is the first part of~\cite[Thm. 5.2]{FissZieg}.

\begin{theorem}   \label{Thm:SpectralRiskElicitable}
Let $T_k$ be a spectral risk measure functional of order $k$ determined by $(p_i, q_i)_{i=1, \ldots, k-1}$ which satisfies $q_i < 1$ for $i=1, \ldots, k-1$. Moreover, define the functionals $T_i (F) := F^{\leftarrow} (q_i)$, $i=1, \ldots, k-1$ and set
\begin{equation*}
T : \mathcal{F} \rightarrow \mathsf{A} , \quad F \mapsto T(F) := (T_1(F), \ldots , T_k(F) )^\top
\end{equation*}
for $\mathsf{A} := T_1(\mathcal{F}) \times \ldots \times T_k(\mathcal{F}) \subset \mathbb{R}^k$. For $a \in \mathbb{R}$ and $i=1, \ldots, k$  define the sets
\begin{align*}
&\mathsf{A}_i^\prime := \lbrace x \in \mathbb{R} \mid \exists z \in \mathsf{A} : x = z_i \rbrace \quad \text{ and}\\
&\mathsf{A}_{i,a}^\prime := \lbrace x \in \mathbb{R} \mid \exists z \in \mathsf{A} : x = z_i, a = z_k \rbrace .
\end{align*}
Then the function $S : \mathsf{A} \times \mathsf{O} \rightarrow \mathbb{R}$ defined via
\begin{align*}
S(x,y) &:= \sum_{i=1}^{k-1} \big[ ( \mathbbm{1}_{\lbrace y \leq x_i \rbrace} - q_i ) g_i (x_i) - \mathbbm{1}_{\lbrace y \leq x_i \rbrace} g_i (y) \big] \\
&\quad+ g_k (x_k) \left( x_k + \sum_{i=1}^{k-1} \frac{p_i}{q_i} ( \mathbbm{1}_{\lbrace y \leq x_i \rbrace} (x_i - y) - q_i x_i )  \right) - G_k (x_k) 
\end{align*}
is a scoring function for $T$, where $g_i : \mathsf{A}_i^\prime \rightarrow \mathbb{R}$, $i = 1, \ldots, k-1$ are some functions such that $\mathbbm{1}_{(- \infty , u ]} g_i$ is $\mathcal{F}$-integrable for any $u \in \mathsf{A}_i^\prime$ and the functions $G_k, g_k : \mathsf{A}_k^\prime \rightarrow \mathbb{R}$ satisfy $G_k^\prime = g_k$. \\
\par
If $G_k$ is convex and for all $i = 1, \ldots, k-1$ and $u \in \mathsf{A}_i^\prime$ the function
\begin{equation*}
H_{i,u} : \mathsf{A}_{i,u}^\prime \rightarrow \mathbb{R}, \quad v \mapsto v \frac{p_i}{q_i} g_k (u) + g_i (v)
\end{equation*}
is increasing, $S$ is an $\mathcal{F}$-consistent scoring function for $T$. If additionally, $G_k$ is strictly convex, $H_{i,u}$ is strictly increasing for any $u \in \mathsf{A}_i^\prime$, $i= 1, \ldots, k-1$ and all $q_i$-quantiles are unique, then $S$ is a strictly $\mathcal{F}$-consistent scoring function and $T$ is elicitable.
\end{theorem}

\begin{remark}
In Remark~\ref{Rem:MeanNecessaryCond} it is mentioned that strictly consistent scoring functions for the mean have a necessary structure. A similar statement can be proved for the previous theorem. More precisely, \cite[Thm. 5.2]{FissZieg} shows that in the situation of Theorem~\ref{Thm:SpectralRiskElicitable} and given a certain identification function, any strictly consistent scoring function for the functional $T$ is of the form $S(x,y) + a(y)$ for some $\mathcal{F}$-integrable function $a$.
\end{remark}

\begin{proof}
We do a proof similar to \cite{FissZieg}. For $S$ to be a scoring function only $\mathcal{F}$-integrability is needed, which follows from the integrability assumptions imposed on $g_i$, $i \in I_{k-1} := \lbrace 1, \ldots, k-1 \rbrace$. It remains to be shown that $S$ is (strictly) $\mathcal{F}$-consistent. To this end, suppose that $G_k$ is convex and for all $i \in I_{k-1}$ and $u \in \mathsf{A}_i^\prime$ the function $H_{i,u}$ is increasing. Fix $F \in \mathcal{F}$, let $x = (x_1, \ldots , x_k) \in \mathsf{A}$ and set $t = (t_1, \ldots, t_k) = T(F)$. For some $w \in \mathsf{A}_k^\prime$ we add the equation
\begin{align*}
\sum_{i=1}^{k-1} &( \mathbbm{1}_{\lbrace y \leq x_i \rbrace} - q_i ) \frac{p_i}{q_i} g_k (w) (x_i - y) \\
&= g_k (w) \sum_{i=1}^{k-1} \frac{p_i}{q_i} ( \mathbbm{1}_{\lbrace y \leq x_i \rbrace} (x_i - y) - q_i (x_i - y) ) \\
&=  g_k (w) \sum_{i=1}^{k-1} \frac{p_i}{q_i} ( \mathbbm{1}_{\lbrace y \leq x_i \rbrace} (x_i - y) - q_i x_i ) + g_k (w) y
\end{align*}
to the definition of $S$ and moreover add and subtract the term $g_k(w) x_k$. The resulting identity is
\begin{align*}
S(x,y) = &\sum_{i=1}^{k-1} \Big[ ( \mathbbm{1}_{\lbrace y \leq x_i \rbrace} - q_i ) \Big(g_i (x_i) + \frac{p_i}{q_i} g_k (w) (x_i - y) \Big) - \mathbbm{1}_{\lbrace y \leq x_i \rbrace} g_i (y) \Big] \\
&+ (g_k (x_k) - g_k (w) )  \left( x_k + \sum_{i=1}^{k-1} \frac{p_i}{q_i} ( \mathbbm{1}_{\lbrace y \leq x_i \rbrace} (x_i - y) - q_i x_i )  \right) \\
&- G_k(x_k) + g_k(w) (x_k - y).
\end{align*}
Using this representation and setting $w := \min (x_k, t_k)$ gives $\bar{S}(x,F) - \bar{S}(t,F) = \sum_{i=1}^{k-1} \xi_i + \mathsf{R}$,  where $\xi_i$ and $\mathsf{R}$  are given by
\begin{align*}
\xi_i = ( F(x_i) - q_i ) \Big( g_i (x_i) + \frac{p_i}{q_i} g_k (w) x_i \Big) - \int_{t_i}^{x_i} g_i (y) + \frac{p_i}{q_i} g_k (w) y \, \mathrm{d} F(y)
\end{align*}
for $i \in I_{k-1}$ and
\begin{align*}
\mathsf{R} &= ( g_k (x_k) - g_k (w) ) \left( x_k + \sum_{i=1}^{k-1} \frac{p_i}{q_i} ( \int_{-\infty}^{x_i} x_i - y \, \mathrm{d} F(y) - q_i x_i) \right) \\
&\quad - G_k (x_k) + G_k (t_k) + g_k (w) (x_k - t_k).
\end{align*}
In order to obtain the terms $\xi_i$, $i \in I_{k-1}$, we use that $F(t_i) = q_i$, which follows from the continuity of $F$. For $\mathsf{R}$ the identity
\begin{equation}  \label{Eqn:Tkcancelsout}
t_k + \sum_{i=1}^{k-1} \frac{p_i}{q_i} ( \int_{-\infty}^{t_i} t_i - y \, \mathrm{d} F(y) - q_i t_i ) = t_k - \sum_{i=1}^{k-1} \frac{p_i}{q_i} \int_{-\infty}^{t_i} y \, \mathrm{d} F(y) = 0
\end{equation}
is applied, which is an implication of (\ref{Eqn:ESforContinuous}). The proof concludes by showing that both $\xi_i$ and $\mathsf{R}$ are nonnegative. \\
\par
Regarding $\xi_i$, pick any $i \in I_{k-1}$ and suppose $t_i < x_i$. The term $\xi_i$ contains $H_{i,w}$ in the integral as well as in parenthesis, and this function is increasing by assumption. Using $F(t_i) = q_i$  this implies
\begin{align}  \label{Eqn:BoundForXi}
\xi_i &= (F(x_i) - q_i) H_{i,w} (x_i) - \int_{t_i}^{x_i} H_{i,w} (y) \, \mathrm{d} F(y)  \nonumber \\
&\geq ( F(x_i) - q_i ) H_{i,w} (x_i) - (F(x_i) - F(t_i) ) H_{i,w} (x_i) = 0 ,
\end{align}
and the same holds true if $t_i > x_i$ is assumed. Consequently, $\xi_i \geq 0$ for $i \in I_{k-1}$ and if all $q_i$-quantiles are unique and all $H_{i,w}$ are strictly increasing, the strict inequality $\xi_i > 0$ follows. \\
\par
Turning to $\mathsf{R}$, the first step is to use integration by parts in order to transform all integrals for $i \in I_{k-1}$ as follows
\begin{align*}
\int_{-\infty}^{x_i} x_i - y \, \mathrm{d} F(y) &= -\int_{-\infty}^{t_i} y \, \mathrm{d} F(y) - \int_{t_i}^{x_i} y \, \mathrm{d} F(y) + F(x_i) x_i \\
&= - \int_{- \infty}^{t_i} y \, \mathrm{d} F(y) + t_i F(t_i) + \int_{t_i}^{x_i} F(y) \, \mathrm{d} y.
\end{align*}
Using this identity together with Equation~(\ref{Eqn:ESforContinuous}) shows that the term in big parenthesis equals
\begin{equation}  \label{Eqn:BracketsRep}
x_k - t_k + \sum_{i=1}^{k-1} p_i ( t_i - x_i + \frac{1}{q_i} \int_{t_i}^{x_i} F(y) \, \mathrm{d} y ) \geq x_k - t_k ,
\end{equation}
where the inequality follows from the fact that
\begin{equation*}
\int_{t_i}^{x_i} \frac{F(y)}{q_i} \, \mathrm{d} y = \int_{t_i}^{x_i} \frac{F(y)}{F(t_i)} \, \mathrm{d} y \geq \int_{t_i}^{x_i} \mathrm{d} y = x_i - t_i
\end{equation*}
for $t_i \neq x_i$ and any $i \in I_{k-1}$. Since $G_k$ is assumed convex, $g_k$ is increasing. Consequently, $g_k (x_k) - g_k (w)$ is nonnegative and hence (\ref{Eqn:BracketsRep}) implies that the first summand of $\mathsf{R}$ is nonnegative. This gives
\begin{align*}
\mathsf{R} &\geq ( g_k (x_k) - g_k (w) ) (x_k - t_k) - G_k (x_k) + G_k(t_k) + g_k (w) (x_k - t_k) \\
&= G_k (t_k) - G_k (x_k) - g_k (x_k) (t_k - x_k) \geq 0  ,
\end{align*}
since $G_k$ is convex and $G_k^\prime = g_k$. The last inequality is strict if $G_k$ is strictly convex and $t_k \neq x_k$. All in all, $S$ is a (strictly) $\mathcal{F}$-consistent scoring function.
\end{proof}

\pagebreak

\begin{remark}
It would be desirable to modify the proof such that $S$ remains (strictly) consistent for discontinuous distribution functions in $\mathcal{F}$. One idea to do so is to apply the same techniques as in Example~\ref{Thm:ExQuantilesElicitable} and the proof of Theorem~\ref{Thm:QuantilesElicitable} in order to bound $\xi_i$ in (\ref{Eqn:BoundForXi}). However, there would still be a problem concerning $\mathsf{R}$, since (\ref{Eqn:Tkcancelsout}) is not always valid for discontinuous distributions. Nevertheless, different techniques can be used to show that that $S$ remains (strictly) consistent for discontinuous distribution functions, see Corollary~\ref{Thm:SpectralRiskElicitableII}.
\end{remark}

The rest of this section presents two corollaries from \cite{FissZieg} which consider two special cases of functionals. The first is concerned with the restriction $q_i < 1$ in Theorem~\ref{Thm:SpectralRiskElicitable} and deals with spectral measures of risk having spectral mass at $1$. This means that it considers a spectral risk measure functional $T_k$ determined by $(p_i, q_i)_{i=1, \ldots, k-1}$, for which $q_i = 1$ for one $i$.  If this is the case and $k=2$, $T_k$ coincides with the mean functional, which is elicitable with scoring functions given in Example~\ref{Thm:ExMeanIsElicitable}. Hence, we suppose that for one $j = 1 , \ldots, k-1$ we have $q_j = 1$ and $p_j \in (0,1)$. Without loss of generality we assume $j = k-1$. The following result is a modified version of \cite[Corollary 5.4]{FissZieg} and we add a proof.

\begin{cor}
Let $T_k$ be a spectral risk measure functional of order $k$ determined by $(p_i, q_i)_{i=1, \ldots, k-1}$, such that $q_{k-1} = 1$ and $p_{k-1} \in (0,1)$. Moreover, define the functionals $T_i (F) := F^{\leftarrow} (q_i)$, $i=1, \ldots, k-2$ and $T_{k-1} (F) := M(F) := \int y \, \mathrm{d}F(y)$. Then the following assertions hold:
\begin{enumerate}[label=(\roman*)]
\item The functional $T:= (T_1, \ldots, T_k)^\top$ is elicitable.
\item Let $S_1$ be an $\mathcal{F}$-consistent scoring function for $\tilde{T} = (T_1, \ldots, T_{k-2}, \tilde{T}_{k-1})^\top$, where $\tilde{T}_{k-1}$ is a spectral risk measure functional of order $k-1$ determined by $(p_i / (1-p_{k-1}), q_i)_{i=1, \ldots, k-2}$. If additionally $S_2$ is an $\mathcal{F}$-consistent scoring function for the mean, an $\mathcal{F}$-consistent scoring function for $T$ as defined in (i) is given by
\begin{align*}
S (x, y) := S_1( x_1, \ldots, x_{k-2},  (x_k - p_{k-1} x_{k-1}) / (1- p_{k-1}), y) + S_2 (x_{k-1}, y) .
\end{align*}
The function $S$ is strictly consistent if $S_1$ and $S_2$ are strictly consistent.
\end{enumerate}
\end{cor}

\begin{proof}
We only need to show that $S$ is (strictly) $\mathcal{F}$-consistent. Let $T_k$ determined by $(p_i, q_i)_{i=1, \ldots, k-1}$ be as required and define $\tilde{p}_i := p_i / (1-\lambda)$ for $i \in I_{k-2}$ and $\lambda := p_{k-1} \in (0,1)$. This gives the representation $T_k = (1 - \lambda) \tilde{T}_{k-1} + \lambda M$, where $\tilde{T}_{k-1}$ is a spectral risk measure functional of order $k-1$ determined by $(\tilde{p}_i, q_i)_{i=1, \ldots, k-2}$. Now define $\tilde{T}_i (F) = F^{\leftarrow}(q_i)$ for $i \in I_{k-2}$ and use Theorem~\ref{Thm:SpectralRiskElicitable} to obtain a (strictly) $\mathcal{F}$-consistent scoring function $S_1 : \mathbb{R}^{k-1} \times \mathsf{O} \rightarrow \mathbb{R}$ for $\tilde{T} = (\tilde{T}_1, \ldots, \tilde{T}_{k-1})^\top$. Due to Lemma~\ref{Thm:AllCompElicitable}, a consistent scoring function for the functional $(\tilde{T}, M)$ is given by $S_1 + S_2$, where $S_2$ is a consistent scoring function for $M$. This scoring function is strictly consistent if $S_1$ and $S_2$ are strictly consistent. The proof concludes by applying the revelation principle stated in Proposition~\ref{Thm:RevelationPrinciple}. To this end, define
\begin{align*}
g: \, &\mathbb{R}^k \rightarrow \mathbb{R}^k, \quad x \mapsto g(x) := (x_1, \ldots, x_{k-2}, x_k, (1-\lambda) x_{k-1} + \lambda x_k )^\top \text{ and} \\
g^{-1}: \, &\mathbb{R}^k \rightarrow \mathbb{R}^k, \quad x \mapsto g^{-1}(x) := (x_1, \ldots, x_{k-2}, (x_k - \lambda x_{k-1}) / (1- \lambda), x_{k-1} )^\top  ,
\end{align*}
where $g^{-1}$ is the inverse function of $g$. As a consequence, we obtain $ g( (\tilde{T}, M) ) = ( \tilde{T}_1, \ldots \tilde{T}_{k-2}, M, T_k )^\top = T$ and due to Proposition~\ref{Thm:RevelationPrinciple}, (strictly) $\mathcal{F}$-consistent scoring functions for $g( (\tilde{T}, M) )$ are given by $S(x,y) = (S_1 + S_2)( g^{-1} (x), y)$, concluding the proof.
\end{proof}

Finally, we consider the special case (VaR, ES) and state which scoring functions are (strictly) consistent for this functional. In light of Equation~(\ref{Eqn:ESVaRInequality}), this functional can only take values in $\mathsf{A}_0 := \lbrace x \in \mathbb{R}^2 \mid x_1 \leq x_2 \rbrace$. The following result is part of \cite[Corollary 5.5]{FissZieg} and we add our own proof since the risk measures ES and VaR have different sign in this work, see also Remark~\ref{Rem:DefinitionRisk}.

\begin{cor}  \label{Thm:JointElicitVaRES}
For $\alpha \in (0,1)$ define the functional $T(F) := ( \mathrm{VaR}_\alpha (F), \mathrm{ES}_\alpha (F))^\top .$ Then $T$ is elicitable and $\mathcal{F}$-consistent scoring functions $S : \mathsf{A}_0 \times \mathsf{O} \rightarrow \mathbb{R}$ are given by
\begin{align*}
S(x_1, x_2, y) &= ( \mathbbm{1}_{\lbrace y \leq - x_1 \rbrace} - \alpha) g_1 (- x_1) - \mathbbm{1}_{\lbrace y \leq - x_1 \rbrace} g_1 (y) \\
&\quad+ g_2 (- x_2) \Big( x_1 - x_2 + \frac{1}{\alpha}  \mathbbm{1}_{\lbrace y \leq - x_1 \rbrace} ( -x_1 - y) \Big) - G_2 (- x_2) ,
\end{align*}
where $g_1$ is increasing, $G_2$ is differentiable, convex, and increasing, $g_2 = G^\prime_2$ and $\mathbbm{1}_{(-\infty, u]} g_1 $ is $\mathcal{F}$-integrable for any $u \in \mathbb{R}$. If additionally $g_1$ is strictly increasing, $G_2$ is strictly convex and strictly increasing and all $\alpha$-quantiles are unique, $S$ is strictly consistent.
\end{cor}

\begin{proof}
Fix $\alpha \in (0,1)$ and observe that $- \mathrm{ES}_\alpha$ is a spectral risk measure functional of order $2$ determined by $( 1, \alpha)$. Hence, if all requirements are met, Theorem~\ref{Thm:SpectralRiskElicitable} proves the (strict) $\mathcal{F}$-consistency of the scoring function
\begin{equation*}
\tilde{S} : \lbrace x \in \mathbb{R}^2 \mid x_1 \geq x_2 \rbrace \times \mathsf{O} \rightarrow \mathbb{R}, \quad (x,y) \mapsto \tilde{S}(x,y) := S (-x_1, -x_2, y)
\end{equation*}
for the functional $(- \mathrm{VaR}_\alpha, - \mathrm{ES}_\alpha)^\top$. Theorem~\ref{Thm:SpectralRiskElicitable} is indeed applicable since the function $H_{1,u} : [u, \infty) \rightarrow \mathbb{R}$ is increasing for any $u \in \mathbb{R}$ due to the fact that $g_1$ is increasing and $g_2$ is positive. Moreover, $\tilde{S}$ is strictly consistent if $g_1$ and $G_2$ are even strictly increasing and $G_2$ is strictly convex. An application of Proposition~\ref{Thm:RevelationPrinciple} with $g : \lbrace x \in \mathbb{R}^2 \mid x_1 \geq x_2 \rbrace \rightarrow \mathsf{A}_0$, $(x_1, x_2) \mapsto (-x_1, -x_2)$ now implies that $S$ is (strictly) consistent for $( \mathrm{VaR}_\alpha , \mathrm{ES}_\alpha )^\top$, which shows the claim.
\end{proof}

\section{Comparative backtesting} 

The previous sections in this chapter are concerned with measures of risk and the question whether they are elicitable. The aim of this section is to clarify why elicitability is a desirable property of risk measures. To do so, we summarize the use of elicitability for comparative backtesting, building on the results on forecast ranking and Diebold-Mariano tests which are discussed in Section~\ref{Sec:ForecastAndBacktest}. Moreover, the arguments presented here can be found in Fissler et al.~\cite{FissZiegGneit} as well as in Nolde and Ziegel~\cite{NoldeZiegel}. \\
\par
If the risk of a position $Y$ is supposed to be measured by a risk measure $\rho$, for example Value at Risk, it is necessary to select a procedure which outputs an estimate of $\rho (Y)$ for given input data. It is then common to check if the used model, or more general the risk estimation procedure, is fit to correctly quantify the risk $\rho(Y)$. This aim is often achieved by \textit{backtesting} the model or estimation procedure. As described in \cite{FissZiegGneit}, a traditional backtest is designed to test the hypothesis $H_0$ : \textit{The used risk measurement procedure is correct}. It is thus meant to assess whether a good estimator for $\rho (Y)$ is used by the risk manager. Therefore, if the null hypothesis is rejected, the risk manager wants to rethink the measurement approach. If the hypothesis is not rejected, the procedure is not changed. As remarked by Acerbi and Szekely~\cite{AcerbiSzekely}, elicitability is not needed for this traditional approach and tests for this situation are found in McNeil et al.~\cite[Sec. 9.3]{McNeilRisk} and \cite{AcerbiSzekely}. However, when the aim is to compare two different forecasts for $\rho(Y)$, elicitability of risk measures becomes essential (see also the discussion in Subsection~\ref{Sec:ForecastRankRemark}). This approach is called \textit{comparative backtesting}.
\par
One example where the comparison of risk measurement procedures is important is the regulation of the financial industry. There, a comparative backtest can be applied to compare the risk measurement procedure used by a financial institution with another benchmark procedure which was devised by some regulator. While the firm may prefer its own risk measurement system, the regulator wants to avoid the use of a misspecified procedure. Such a procedure might lower capital requirements, giving the firm an unfair advantage over its competitors. Moreover, it might lead to risky positions which remain hidden from the regulator and jeopardize financial stability. It is thus convenient for a regulator to have a test at hand with which it can be checked if an internal model is at least as good as the benchmark model.
\par
In order to elaborate on this situation, we consider two risk measurement procedures in the following which output forecasts of $\rho(Y_t)$ for the time periods $t=1, \ldots, n$. One is the procedure which is employed by a hypothetical financial institution and is thus called `internal procedure'. The other one is the procedure of the regulator who is in charge of supervising this firm and its risk management. It is thus called `standard procedure'. We assume that $\rho$ is elicitable and fix a scoring function $S$ which is strictly consistent for $\rho$. In this setting, we now present the \textit{three-zone approach} as proposed by Fissler et al.~\cite{FissZiegGneit}. At first, the test statistic for a Diebold-Mariano test is repeated. Similar to Subsection~\ref{Sec:CompBacktest}, define
\begin{equation*}
\mathsf{T}_n :=  \sqrt{ \frac{n}{\hat{\sigma}^2_n} } \mathsf{K}_n = \frac{1}{ \sqrt{\hat{\sigma}^2_n  n} } \sum_{t=1}^{n} ( S( \hat{x}_t, y_t) - S(\hat{z}_t , y_t) ),
\end{equation*}
where $(\hat{x}_t)_{t=1, \ldots,n}$ are forecasts of the internal procedure, $(\hat{z}_t)_{t=1, \ldots,n}$ are forecasts of the standard procedure and $(y_t)_{t=1, \ldots,n}$ are realizations of the portfolio variable $Y$ at time points $t=1, \ldots, n$. Furthermore, assume that asymptotic normality as stated in Theorem~\ref{Thm:CLTforStationary} holds for the statistic $\mathsf{T}_n$. The following two hypotheses are proposed by \cite{FissZiegGneit}:
\begin{itemize}
\item[] $H_0^1$ : The risk measure estimates of the internal procedure are \textit{at most as good} as the ones from the standard procedure.
\item[] $H_0^2$ : The risk measure estimates of the internal procedure are \textit{at least as good} as the ones from the standard procedure.
\end{itemize}
Similar to Subsection~\ref{Sec:CompBacktest} we reformulate the hypotheses into statements concerning the expected value of the mean score differences. Firstly, the hypothesis $H_0^1$ is stated as $\mathbb{E} \mathsf{K}_n \geq 0$. Due to the asymptotic normality of $\mathsf{K}_n$, a test of $H_0^1$ would have a rejection region given by $(-\infty , C_1]$, where $C_1 < 0$ if the level of significance is below $0.5$. Repeating this argument for $H_0^2$ gives $\mathbb{E} \mathsf{K}_n \leq 0$ and again, a rejection region is given by $[C_2, \infty)$, where $C_2 > 0$ holds. All in all, the two hypotheses partition the real line into three zones given by $(-\infty, C_1]$, $(C_1, C_2)$ and $[C_2, \infty)$. This results in the following interpretation given by \cite{FissZiegGneit}.
\begin{itemize}
\item All values less or equal $C_1$ make up the \textbf{green zone}. If the test statistic falls into this interval, the hypothesis $H_0^1$ is rejected and the risk measure estimates of the internal procedure are considered more accurate than the estimates of the standard procedure. The comparative backtest is passed.
\item All values in $(C_1, C_2)$ constitute the \textbf{yellow zone}. If the test statistic takes values in this interval, it is not possible to reject either $H_0^1$ or $H_0^2$. Consequently, it is not clear which procedure is preferable. A conservative approach would argue that the internal procedure was not proven to be better, hence the backtest should be failed. Alternatively, it can be argued that it is not clear whether the standard procedure is better, hence the backtest should be passed.
\item All values greater or equal $C_2$ form the \textbf{red zone}. If the test statistic lies in this interval, the hypothesis $H_0^2$ is rejected and it is thus assumed that the risk measure estimates of the internal procedure are worse than the estimates of the standard procedure. The comparative backtest is failed.
\end{itemize}

Using this approach, a financial regulator can compare internal procedures of different institutions to a benchmark. If a firm fails the test, the regulator can require the firm to adapt the benchmark procedure in order to ensure that risk measure estimates are of sufficient quality.

To the best of our knowledge, this three-zone approach is not applied in practice right now, however there are simulation studies for different settings. While Fissler et al.~\cite{FissZiegGneit} do comparisons based on i.i.d. data, Nolde and Ziegel~\cite{NoldeZiegel} use the extension of the Diebold-Mariano framework by Giacomini and White~\cite{GiacominiWhite} (as mentioned in Subsection~\ref{Sec:ForecastRankRemark}) and consider data generated by a stochastic process. Finally, it is important to note that all drawbacks and/or remarks which are mentioned in Subsection~\ref{Sec:ForecastRankRemark} remain valid for comparative backtesting of risk measures.

\chapter{Further topics and discussion}
\label{Chapter4}

This chapter contains further topics related to elicitability and identifiability. At first, functionals with elicitable components are studied and a well-known characterization of scoring functions is presented. Moreover, a special class of functionals is introduced, which can be used to generalize the results of Fissler and Ziegel~\cite{FissZieg} as stated in Theorem~\ref{Thm:SpectralRiskElicitable}. A discussion about the presented results as well as the remaining open problems finishes the chapter and this thesis.

\section{Functionals with elicitable components} 
\label{Sec:FunctionalsComponents}

This section considers functionals $T : \mathcal{F} \rightarrow \mathsf{A} \subseteq \mathbb{R}^k$, for which all component functionals $T_1, \ldots, T_k$ are elicitable. As shown in Proposition~\ref{Thm:LevelSetCond}, a necessary condition for elicitability is the `convexity of level sets', and as pointed out in Remark~\ref{Rem:LevelSetSufficient}, this condition is also sufficient in the one-dimensional case, as long as certain assumptions are met. For higher dimensions a similar characterization of elicitability is unknown, thus an intuitive starting point to study higher order elicitability is to consider functionals for which all components are elicitable. This section presents one characterization of such functionals which is due to \cite{FissZieg}. \\
\par
For $k > 1$ let $T_i : \mathcal{F} \rightarrow \mathsf{A}_i \subseteq \mathbb{R}$, $i=1, \ldots, k$,  be functionals and define
\begin{equation}  \label{Eqn:FunctionalTDef}
T : \mathcal{F} \rightarrow \mathsf{A} \subseteq \mathsf{A}_1 \times \ldots \times \mathsf{A}_k , \quad F \mapsto T(F) := (T_1(F), \ldots, T_k(F))^\top .
\end{equation}
A standing assumption of this section is $\mathsf{A} = T(\mathcal{F})$, such that $T$ is surjective. Lemma~\ref{Thm:AllCompElicitable} gives a sufficient condition for the elicitability of $T$: If all $T_i$, $i=1, \ldots, k$, are elicitable, then $T$ is elicitable and (strictly) $\mathcal{F}$-consistent scoring functions are given by $\sum_{i=1}^{k} S_i (x_i, y)$, where $S_i$ is a (strictly) $\mathcal{F}$-consistent scoring function for $T_i$. Following \cite{FissZieg}, we call scoring functions for $T$ having this structure \textit{separable}. The central question which arises is: Are all strictly consistent scoring functions for $T$ separable?
\par
The most important tool to answer this question is Osband's principle (see Theorem~\ref{Thm:OsbandPrinciple1} and Theorem~\ref{Thm:OsbandPrinciple2}), which states that under certain conditions there is a matrix-valued mapping $h: \intr( \mathsf{A} ) \rightarrow \mathbb{R}^{k \times k}$ such that
\begin{equation*}
\nabla \bar{S}(x,F) = h(x) \bar{V}(x,F)
\end{equation*}
holds for all $x \in \intr (\mathsf{A})$ and $F \in \mathcal{F}$. In order to calculate $S$ using Osband's principle, it is assumed that $T_i$ is identifiable with oriented strict $\mathcal{F}$-identification function $V_i$ for $i=1, \ldots, k$. This implies that $T$ is identifiable with oriented strict $\mathcal{F}$-identification function
\begin{equation}  \label{Eqn:IdentFunctionDef}
V : \mathsf{A} \times \mathsf{O} \rightarrow \mathbb{R}^k, \quad (x,y) \mapsto V(x,y):= (V_1 (x_1, y) ,\ldots, V_k(x_k,y) )^\top  ,
\end{equation}
see part (ii) and (iii) of Lemma~\ref{Thm:AllCompElicitable}. Moreover, the following assumption which is due to \cite{FissZieg} is imposed on $V$ and $\mathcal{F}$.

\phantomsection
\begin{assumption2}[V4] \label{As:V4}
Let Assumption (\nameref{As:V3}) hold. For all $r \in \lbrace 1, \ldots, k \rbrace$ and for all $t \in \intr (\mathsf{A}) \cap T(\mathcal{F})$ there are $F_1, F_2 \in \mathcal{F}$ with $T(F_1) = T(F_2) = t$ such that
\begin{equation*}
\partial_l \bar{V}_l (t,F_1) = \partial_l \bar{V}_l (t, F_2) \quad \text{for } l \in \lbrace 1, \ldots, k \rbrace \backslash \lbrace r\rbrace \text{ and} \quad \partial_r \bar{V}_r (t, F_1) \neq \partial_r \bar{V}_r (t,F_2)
\end{equation*}
hold.
\end{assumption2}

Similar to \cite{FissZiegArxiv}, the following two examples calculate the partial derivatives of $\bar{V}(\cdot,F)$ for quantiles and expectiles in order to show how Assumption~(\nameref{As:V4}) looks like in these special cases. They illustrate that Assumption~(\nameref{As:V4}) is not only a condition on the identification function $V$, but also on the richness of the class $\mathcal{F}$.

\begin{example}[Assumption (\nameref{As:V4}) for quantiles]
Assume that all $F \in \mathcal{F}$ have continuous densities with respect to the Lebesgue measure. Let $T$ be defined as in (\ref{Eqn:FunctionalTDef}) and suppose $T_i (F) = F^{\leftarrow}(\alpha_i)$ for $\alpha_i \in (0,1)$, $i \in I_k := \lbrace 1, \ldots, k\rbrace$. Moreover, let $V$ be defined as in (\ref{Eqn:IdentFunctionDef}). As mentioned in Subsection~\ref{Sec:QuantileReg}, an $\mathcal{F}$-identification function for $F^{\leftarrow}(\alpha_i)$ is defined via $V_i(x_i,y) := \mathbbm{1}_{\lbrace y \leq x_i \rbrace} - \alpha_i$. For any $F \in \mathcal{F}$ and $i \in I_k$ we calculate the derivative $\partial_i \bar{V}_i (x, F) = (F(x_i) - \alpha_i)^\prime = f(x_i)$, where $f$ denotes the density of $F$. Hence, Assumption~(\nameref{As:V4}) states that for any $r \in I_k$ and $t \in \intr(\mathsf{A})$ there exist $F_1, F_2 \in \mathcal{F}$ having the same $\alpha_i$-quantiles $t_i$, $i \in I_k$, and such that their densities coincide at $t_i$, $i \in I_k \backslash \lbrace r \rbrace$, but not at $t_r$.
\end{example}

\begin{example}[Assumption (\nameref{As:V4}) for expectiles]
Given the situation of the previous example, assume additionally that all members of $\mathcal{F}$ have finite first moments. Let $T$ be defined as in (\ref{Eqn:FunctionalTDef}) and suppose $T_i (F) = e_{\tau_i} (F)$ for $\tau_i \in (0,1)$, $i \in I_k$. As shown in Lemma~\ref{Thm:ExpectileIdentifiable}, $V_i (x_i,y) := \vert \mathbbm{1}_{\lbrace y \leq x_i \rbrace} - \tau_i \vert (x_i - y)$ defines a strict $\mathcal{F}$-identification function for the $\tau_i$-expectile. As above, let $f$ denote the density of $F$ and calculate the derivative
\begin{align*}
\partial_i \bar{V}_i (x, F) &= \frac{\mathrm{d}}{\mathrm{d}x} \Big( \tau_i \int_{x}^{\infty} (x-y) f(y) \, \mathrm{d}y + (1-\tau_i) \int_{-\infty}^{x} (x - y) f(y) \, \mathrm{d}y \Big) \\
&= \frac{\mathrm{d}}{\mathrm{d}x} \Big( - \tau_i \int_{x}^{\infty} y f(y) \, \mathrm{d}y - (1-\tau_i) \int_{-\infty}^{x} y f(y) \, \mathrm{d}y \Big) \\
&\quad+ \frac{\mathrm{d}}{\mathrm{d}x} \big( x ( \tau_i (1- F(x)) + (1-\tau_i) F(x)) \big)  = (1 - 2 \tau_i) F(x) + \tau_i ,
\end{align*}
which implies $\partial_i \bar{V}_i (x, F) = (1 - 2 \tau_i) F(x_i) + \tau_i $ for any $i \in I_k$. Consequently, Assumption~(\nameref{As:V4}) is more complicated for expectiles, but there is no obvious reason why it should not be satisfied for certain classes $\mathcal{F}$.
\end{example}

\begin{remark}  \label{Rem:ExpectationNotSeparable}
A simple but well-known situation where Assumption~(\nameref{As:V4}) cannot be satisfied for any class $\mathcal{F}$ occurs when $T$ consists of expectations or ratios of expectations with the same denominator. To see this, let $g : \mathsf{O} \rightarrow \mathbb{R}^k$ and $q : \mathsf{O} \rightarrow (0, \infty)$ be $\mathcal{F}$-integrable functions and set $T_i (F) := \bar{g}_i(F) / \bar{q}(F)$ for $i=1, \ldots, k$. Then an identification function for $T_i$ is given by $V_i(x_i,y) = g_i(y) - x_i q(y)$ and the derivative of $\bar{V}_i(x_i,F)$ depends only on $F$. Hence, Assumption~(\nameref{As:V4}) can never be satisfied for such functionals. At the same time there exist strictly $\mathcal{F}$-consistent scoring functions for $T$ which are not separable. One example of such a scoring function is obtained by choosing the (strictly convex) function $f(x) = \exp( - \sum_{i=1}^{k} x_i )$ in Theorem~\ref{Thm:GenExpectElicit}.
\end{remark}

In view of the previous examples and remarks, the next aim is to show that functionals as defined in (\ref{Eqn:FunctionalTDef}) which satisfy Assumption~(\nameref{As:V4}) only admit separable strictly consistent scoring functions. To this end, a corollary of Osband's principle is needed in order to study the structure of the function $h$.

\begin{cor}[Fissler and Ziegel {\cite[Corollary 3.3]{FissZieg}}]  \label{Thm:CorForOsband}
Let $T: \mathcal{F} \rightarrow \mathsf{A} \subseteq \mathbb{R}^k$ be a surjective, elicitable, and identifiable functional with strict $\mathcal{F}$-identification function $V : \mathsf{A} \times \mathsf{O} \rightarrow \mathbb{R}^k$ and strictly $\mathcal{F}$-consistent scoring function $S : \mathsf{A} \times \mathsf{O} \rightarrow \mathbb{R}$. If $S$ and $V$ satisfy the Assumptions~(\nameref{As:V1}), (\nameref{As:V3}), and (\nameref{As:S2}) and $h$ is the function from Theorem~\ref{Thm:OsbandPrinciple1}, the second-order derivatives satisfy
\begin{align}  \label{Eqn:SymmetryOfDeriv}
\partial_m \partial_l \bar{S}(x,F) &= \sum_{i=1}^{k} \partial_m h_{li} (x) \bar{V}_i (x,F) + h_{li} (x) \partial_m \bar{V}_i(x,F) \\
&= \sum_{i=1}^{k} \partial_l h_{mi} (x) \bar{V}_i (x,F) + h_{mi} (x) \partial_l \bar{V}_i (x,F) = \partial_l \partial_m \bar{S}(x,F) \nonumber
\end{align}
for all $l,m \in \lbrace 1, \ldots, k \rbrace$, for all $F \in \mathcal{F}$ and almost all $x \in \intr (\mathsf{A})$. In particular, the identity holds for $x=T(F) \in \intr( \mathsf{A})$. 
\end{cor}

\begin{proof}
We argue similar to \cite{FissZieg}. By using Assumptions~(\nameref{As:V1}), (\nameref{As:V3}), and (\nameref{As:S2}), it is possible to apply  Theorem~\ref{Thm:OsbandPrinciple1} and obtain a locally Lipschitz continuous $h$ such that $\nabla \bar{S}(x,F) = h(x) \bar{V}(x,F)$ holds for any $F \in \mathcal{F}$. Due to Assumption~(\nameref{As:S2}), $ \nabla \bar{S}(\cdot, F)$ is also locally Lipschitz continuous, hence $h$ and $\nabla \bar{S}(\cdot,F)$ are differentiable for a.e. $x \in \intr (\mathsf{A})$ by Rademacher's theorem (see Theorem~\ref{Thm:AppendixRademacher}). Since $\bar{S}(\cdot, F)$ has differentiable partial derivatives almost everywhere, Schwarz's theorem on the symmetry of second derivatives (see for instance Grauert and Fischer~\cite[Satz 3.3]{GrauertFischer}) gives (\ref{Eqn:SymmetryOfDeriv}) for any $F \in \mathcal{F}$ and a.e. $x \in \intr (\mathsf{A})$. Since (\nameref{As:S2}) requires $\bar{S}(\cdot,F)$ to be twice continuously differentiable in $t = T(F) \in  \intr( \mathsf{A})$, Schwarz's theorem shows that (\ref{Eqn:SymmetryOfDeriv}) holds also in $t$.
\end{proof}

Using Assumption~(\nameref{As:V4}) and Corollary~\ref{Thm:CorForOsband}, we prove the following proposition which is the first part of \cite[Prop. 4.2]{FissZieg}. It states that under certain conditions the connection function $h$ from Osband's principle has a simple diagonal structure.

\begin{prop}  \label{Thm:FunchisDiagonal}
For $r=1, \ldots, k$ let $T_r : \mathcal{F} \rightarrow \mathsf{A}_r \subseteq \mathbb{R}$ be an elicitable and identifiable functional with oriented strict $\mathcal{F}$-identification function $V_r : \mathsf{A} \times \mathsf{O} \rightarrow \mathbb{R}$. Let $T$ be defined as in (\ref{Eqn:FunctionalTDef}) with strict $\mathcal{F}$-identification function $V$ as defined in (\ref{Eqn:IdentFunctionDef}) and strictly $\mathcal{F}$-consistent scoring function $S: \mathsf{A} \times \mathsf{O} \rightarrow \mathbb{R}$. Moreover, let Assumptions (\nameref{As:V1}), (\nameref{As:V4}), and (\nameref{As:S2}) hold, suppose $\mathsf{A}$ is connected and define the set
\begin{equation*}
\intr (\mathsf{A})_r^\prime := \lbrace x \in \mathbb{R} \mid \exists z \in \intr ( \mathsf{A} ) : x = z_r \rbrace. 
\end{equation*}
Then the function $h$  from Theorem~\ref{Thm:OsbandPrinciple1} satisfies the following:
\begin{enumerate}[label=(\roman*)]
\item For $r=1, \ldots, k$ there are functions $g_r : \intr (\mathsf{A})_r^\prime \rightarrow (0, \infty)$ such that $h_{rr} (x) = g_r (x_r)$ for any $x \in \intr (\mathsf{A})$.
\item For $r,l \in \lbrace 1, \ldots, k \rbrace$ and $r \neq l$ we have $h_{rl} (x) = 0$ for all $x \in \intr (\mathsf{A})$.
\end{enumerate}
\end{prop}

\begin{proof}
Following the proof of \cite{FissZieg}, we begin by showing property (ii) and define $I_k := \lbrace 1, \ldots, k \rbrace$. By construction, the identification function $V$ satisfies
\begin{equation}  \label{Eqn:Chara1CrossDiffzero}
\partial_r \bar{V}_l (x,F) = 0 \quad \text{for any } r,l \in I_k, \, r \neq l \text{ and any } F \in \mathcal{F}, \, x \in \intr (\mathsf{A}) .
\end{equation}
Using this fact and Equation~(\ref{Eqn:SymmetryOfDeriv}) gives
\begin{equation}  \label{Eqn:Chara2partialdiff}
h_{lr} (t) \partial_r \bar{V}_r (t , F)  = h_{rl} (t) \partial_l \bar{V}_l (t, F)
\end{equation}
for any $r,l \in I_k$, $r \neq l$ and any $F \in \mathcal{F}$ such that $ T(F) = t \in \intr (\mathsf{A})$. If for any $t \in \intr(\mathsf{A})$  and $r \in I_k$ distribution functions $F_1, F_2 \in \mathcal{F}$ are chosen according to Assumption~(\nameref{As:V4}), Equation~(\ref{Eqn:Chara2partialdiff}) implies 
\begin{align*}
h_{lr} (t) \partial_r \bar{V}_r (t, F_1 ) = h_{rl} (t) \partial_l \bar{V}_l (t, F_1) = h_{rl} (t) \partial_l \bar{V}_l(t, F_2) = h_{lr} (t) \partial_r \bar{V}_r (t, F_2)
\end{align*}
for any $l \in I_k$, $l \neq r$. Consequently, we must have $h_{lr} (t) = 0$ for $l, r \in I_k$, $r \neq l$ and due to the surjectivity of $T$ we repeat this argument for any $t \in \intr (\mathsf{A})$, which gives $h_{lr} = 0$ for $l, r \in I_k$, $r \neq l$.\\
For the first property, observe that part (ii) together with (\ref{Eqn:Chara1CrossDiffzero}) implies that Equation~(\ref{Eqn:SymmetryOfDeriv}) simplifies to
\begin{equation}  \label{Eqn:Chara3linIndarg}
\sum_{i=1}^{k} ( \partial_l h_{ri} (x)  - \partial_r h_{li} (x) ) \bar{V}_i (x, F) = 0
\end{equation}
for any $F \in \mathcal{F}$, $l, r \in I_k$, $r \neq l$ and a.e. $x \in \intr (\mathsf{A})$. Assumption~(\nameref{As:V1}) gives that for a.e. $x \in \intr (\mathsf{A})$ there are $F_1, \ldots, F_k \in \mathcal{F}$ such that the vectors $\bar{V}(x, F_1), \ldots, \bar{V}(x,F_k)$ are linearly independent. If the matrix having these vectors as columns is called $\mathbb{V}(x)$, Equation~(\ref{Eqn:Chara3linIndarg}) implies
\begin{equation*}
( \partial_l h_{ri} (x)  - \partial_r h_{li} (x) )_{i \in I_k}^\top \mathbb{V}(x) = 0
\end{equation*}
and since $\mathbb{V}(x)$ has full rank, it follows that $\partial_l h_{ri} (x) = \partial_r h_{li} (x)$ for any $l, r, i \in I_k$ and a.e. $x \in \intr (\mathsf{A})$. Setting $i=r$ and using (ii) this gives $\partial_l h_{rr}(x) = \partial_r h_{lr} (x) = 0$, showing that $h_{rr}$ is constant in $x_l$ for any $r,l \in I_k$, $r \neq l$ and a.e. $x \in \intr (\mathsf{A})$. Since $\mathsf{A}$ is connected and $h$ is continuous, there is some function $g_r : \intr (\mathsf{A})^\prime_r \rightarrow \mathbb{R}$ such that $h_{rr} (x) = g_r (x_r)$ holds for all $x \in \intr (\mathsf{A})$ and $r \in I_k$. It remains to be shown that the functions $g_r$, $r \in I_k$, are strictly positive. Using Lemma~\ref{Thm:SufficientCondElicit} (i) and the strict consistency of $S$, it follows that for any $v \in \mathbb{S}^{k-1}$ and $F \in \mathcal{F}$, $t=T(F)$ the function $\Psi_{F,v} (s) := \bar{S}(t + sv, F)$ has a unique global minimum at $s=0$. Choosing $F \in \mathcal{F}$ with $T(F) = t \in \intr(\mathsf{A})$ and using the diagonal structure of $h$ to compute the derivative gives
\begin{equation*}
\Psi_{F,v}^\prime (s) = v^\top \nabla \bar{S}(t + sv, F) = \sum_{r=1}^{k} g_r (t_r + s v_r) \bar{V}_r (t_r + s v_r, F) v_r ,
\end{equation*}
which has to be positive for $s>0$ and negative for $s < 0$ if $s$ is small enough. For $r \in I_k$ we let $v_r$ be the $r$-th unit vector and conclude that $\Psi_{F,v_r}^\prime (s) = g_r (t_r + s) \bar{V}_r (t_r + s, F) > 0$ for small enough $s > 0$. Since $V_r$ is an oriented strict identification function, this gives $g_r (t_r + s) > 0$. Similarly, $g_r (t_r + s) > 0$ must hold for small enough $s <0$.  Due to the surjectivity of $T$, this argument is repeated for any $t_r \in \intr (\mathsf{A})_r^\prime$ and hence $g_r$ must be strictly positive for any $r \in I_k$.
\end{proof}

The next corollary shows that all scoring functions are separable if Assumption~(\nameref{As:V4}) is satisfied. It uses the diagonal structure of $h$ shown in Proposition~\ref{Thm:FunchisDiagonal} together with the pointwise version of Osband's principle, Theorem~\ref{Thm:OsbandPrinciple2}. The result and its proof can be found in \cite[Prop. 4.2 (ii)]{FissZieg} and since  Theorem~\ref{Thm:OsbandPrinciple2} is used the assumptions are slightly modified.

\begin{cor}
Given the situation of Proposition~\ref{Thm:FunchisDiagonal}, assume furthermore that $\intr (\mathsf{A})$ is a hyperrectangle and that Assumptions (\nameref{As:F1}), (\nameref{As:VS1}) and (\nameref{As:B1}) are satisfied. Then $S$ is a strictly $\mathcal{F}$-consistent scoring function for $T$ if and only if it is of the form
\begin{equation*}
S(x,y) = \sum_{i=1}^{k} S_i (x_i ,y)  
\end{equation*}
for almost all $(x,y) \in \mathsf{A} \times \mathsf{O}$, where for any $i=1, \ldots, k$ the function $S_i$ is a strictly $\mathcal{F}$-consistent scoring function for $T_i$.
\end{cor}

\begin{proof}
At first, note that the function $(x,y) \mapsto \sum_{i=1}^{k} S_i (x_i ,y)$ is strictly $\mathcal{F}$-consistent due to Lemma~\ref{Thm:AllCompElicitable} (i). To show that this representation is necessary, observe that all requirements are met to apply Proposition~\ref{Thm:FunchisDiagonal} as well as Theorem~\ref{Thm:OsbandPrinciple2}. This implies the representation
\begin{equation*}
S(x,y) = \sum_{i=1}^{k} \int_{z_i}^{x_i} g_i (v) V_i (v, y) \, \mathrm{d}v  + a(y)
\end{equation*}
for a.e. $(x,y) \in \mathsf{A} \times \mathsf{O}$, where $z \in \intr (\mathsf{A})$ and $a$ is some $\mathcal{F}$-integrable function. If we fix $r \in I_k$ and define the scoring function
\begin{equation*}
S_r : \mathsf{A}_r \times \mathsf{O} \rightarrow \mathbb{R}, \quad (x,y) \mapsto S_r(x,y) := \int_{z_r}^{x} g_r (v) V_r (v,y) \, \mathrm{d}v  + \frac{1}{k} a(y),
\end{equation*}
we only need to show that $S_r$ is strictly $\mathcal{F}$-consistent for $T_r$. To this end, let $F \in \mathcal{F}$ be arbitrary, set $t= T(F)$ and let $x_r \in \mathsf{A}_r$, $x_r \neq t_r$ be such that $\tilde{x} := (t_1, \ldots, t_{r-1} , x_r, t_{r+1}, \ldots, t_k)^\top \in \mathsf{A}$ holds. Then the strict $\mathcal{F}$-consistency of $S$ implies
\begin{align*}
\bar{S}_r (x_r, F) - \bar{S}_r (t_r,F) = \int_{t_r}^{x_r} g_r (v) \bar{V}_r (v,F) \, \mathrm{d}v = \bar{S} ( \tilde{x}, F) - \bar{S}(t,F) > 0,
\end{align*}
finishing the proof.
\end{proof}

In view of the previous corollary, we conclude that (under certain assumptions) any strictly $\mathcal{F}$-consistent scoring function for a vector consisting of quantiles or expectiles is separable. Note that it is necessary to assume that $\mathcal{F}$ is rich enough, in order to guarantee that Assumption~(\nameref{As:V4}) is satisfied. Moreover, the domain $\mathsf{A}$ cannot be chosen arbitrarily. To see this, assume $k=2$ and define $T_1 (F) := F^{\leftarrow}(\alpha_1)$ and $T_2(F) := F^{\leftarrow} (\alpha_2)$ for $0 < \alpha_1 < \alpha_2 < 1$. Since $s \mapsto F^{\leftarrow} (s)$ is monotone, we have $T_1 (F)  < T_2 (F)$ for any $F \in \mathcal{F}$, which implies that the functional $(T_1, T_2)^\top$ can only take values in $\tilde{\mathsf{A}} := \lbrace x \in \mathbb{R}^2 \mid x_1 < x_2 \rbrace$. Since in the previous proof, Theorem~\ref{Thm:OsbandPrinciple2} is used and $T$ is assumed to be surjective, the separability of $S$ can only be shown for sets $\mathsf{A} \times \mathsf{O}$ where $\intr (\mathsf{A})$ is a hyperrectangle contained in $\tilde{\mathsf{A}}$ (see also Remark~\ref{Rem:WhyHyperrectangle}). Naturally, the same reasoning applies to expectiles since they are also monotonic, see Lemma~\ref{Thm:ExpectileProperties} (ii).
\par
As discussed in Remark~\ref{Rem:ExpectationNotSeparable}, a separability statement as in the previous Corollary cannot hold for functionals consisting of ratios of expectations with the same denominator. However, the techniques of the proof of Proposition~\ref{Thm:FunchisDiagonal} can be used to prove a characterization of strictly consistent scoring functions for such functionals. This can be found in Fissler and Ziegel~\cite[Prop. 4.4 (ii)]{FissZieg} and a similar result is stated in Frongillo and Kash~\cite[Thm. 13]{FrongilloKash}.

\section{Perspectives on non-elicitable functionals} 

This thesis considers two important non-elicitable functionals, namely variance and Expected Shortfall, which are both part of an elicitable vector. For the variance, it is relatively simple to show this using the revelation principle, while the situation is rather complicated for ES. This section discusses two concepts other than the revelation principle which give arguments as to why a non-elicitable functional is jointly elicitable. The first concept argues that both variance and ES are elicitable if restricted to certain subclasses of $\mathcal{F}$ and conjectures that this can be extended to $\mathcal{F}$. The second one shows that variance and ES can be written as mean scores of strictly $\mathcal{F}$-consistent scoring functions and proves that such functionals are always jointly elicitable. Before presenting any details, it should be remarked that, although both perspectives are well suited for variance and ES, there are other functionals for which it is unknown if they fit into any of the two concepts. For instance, as shown by Heinrich~\cite{HeinrichMode}, the mode functional is not elicitable and it is not known whether it is part of an elicitable vector or not.

\subsection{Conditional elicitability} 

For both variance and ES it can be argued that they fail to be elicitable because of their dependence on other functionals, namely the mean and the quantile. In view of this dependence, it is natural to define a concept which calls a functional $T$ elicitable \textit{conditional} on an other elicitable functional $T_1$. Such a definition is proposed by Emmer et al. \cite{EmmerKratzTasche} and we present the slightly more general formulation of \cite{FissZieg}.

\begin{definition}[Conditional elicitability/identifiability]  \label{Def:CondElicitability}
A functional $T: \mathcal{F} \rightarrow \mathsf{A} \subseteq \mathbb{R}$ is called \textit{conditionally elicitable of order} $k$ if there are $k-1$ elicitable functionals $T_i : \mathcal{F} \rightarrow \mathsf{A}_i \subset \mathbb{R}$, $i = 1, \ldots, k-1$, such that for any $(x_1, \ldots, x_{k-1} )^\top \in \mathsf{A}_1 \times \ldots \times \mathsf{A}_{k-1}$ the restriction of $T$ to the class
\begin{equation}  \label{Eqn:DefRestrictClass}
\mathcal{F}_{ (x_1, \ldots, x_{k-1}) } := \lbrace F \in \mathcal{F} \mid T_1(F) = x_1, \ldots , T_{k-1} (F) = x_{k-1} \rbrace
\end{equation}
is elicitable. Similarly, it is \textit{conditionally identifiable of order} $k$ if there are $k-1$ identifiable functionals $T_i$, $i= 1,\ldots, k-1$, such that the restriction of $T$ to any class $\mathcal{F}_{ (x_1, \ldots, x_{k-1}) }$ is identifiable.
\end{definition}

It is shown in Example~\ref{Thm:ExVariNotElicitable} that the variance functional is not elicitable for certain choices of $\mathcal{F}$. However, in view of Example~\ref{Thm:ExVariElicitableForCenter}, it is straightforward to show that the variance is conditionally elicitable, a fact which is also mentioned in \cite{EmmerKratzTasche}. 

\begin{example}[Conditional elicitability of the variance]  \label{Thm:ExVarianceCondElicit}
Let $\mathcal{F}$ be a class of distribution functions having finite second moments and let $T$ be the variance functional as defined in (\ref{Eqn:VariFunctionalDef}). If we choose $z \in \mathbb{R}$ and define the subclass $\mathcal{F}_z := \lbrace F \in \mathcal{F} \mid  \int y \, \mathrm{d}F(y) = z \rbrace$, we obtain $T_{\vert \mathcal{F}_z } (F) = \int (y - z)^2 \, \mathrm{d}F(y)$. Since this is an expectation, the restricted variance functional is elicitable and (strictly) $\mathcal{F}_z$-consistent scoring functions are given in Theorem~\ref{Thm:GenExpectElicit}. This shows that the variance is conditionally elicitable of order 2.
\end{example}

Similar to Expected Shortfall, spectral risk measures (and thus also their functional counterpart as introduced in Definition~\ref{Def:SpectralRiskFunctional}) fail to be elicitable for large enough classes $\mathcal{F}$ (see Ziegel~\cite{ZiegelCoherence}). However, similar to the variance, ES is conditionally elicitable, which is shown for continuous distribution functions in \cite{EmmerKratzTasche}. We extend this result to spectral measures of risk and drop the continuity condition.

\begin{lemma}[Conditional elicitability of spectral risk measure functionals]  \label{Thm:SpectralCondElicit}
Let $\mathcal{F}$ be a class of distribution functions having finite first moments. Fix $k \geq 2$ and let $T$ be a spectral risk measure functional of order $k$ determined by $(p_i, q_i)_{i=1, \ldots, k-1}$ which satisfies $q_i < 1$ for $i=1, \ldots, k-1$. If all distributions in $\mathcal{F}$ have unique $q_i$-quantiles, then $T$ is conditionally elicitable of order $k$. 
\end{lemma}

\begin{proof}
Suppose all distributions in $\mathcal{F}$ have unique $q_i$-quantiles. Define the functionals $T_i (F) := F^{\leftarrow} (q_i)$ for $i=1, \ldots, k-1$, which are all elicitable relative to $\mathcal{F}$ due to Theorem~\ref{Thm:QuantilesElicitable}. For any $(x_1, \ldots, x_{k-1})^\top \in \mathbb{R}^{k-1}$ let $\mathcal{F}_{(x_1, \ldots, x_{k-1})}$ be the restricted class as defined in (\ref{Eqn:DefRestrictClass}). Denoting the restriction of $T$ to this class via $T'$, and using Identity~(\ref{Eqn:ESAcerbiDefinition}) gives
\begin{align*}
T' (F) =  - \sum_{i=1}^{k-1} \mathrm{ES}_{q_i} (F) p_i 
= \sum_{i=1}^{k-1}  \frac{p_i}{q_i} \int_\mathsf{O} y \mathbbm{1}_{ \lbrace y \leq x_i \rbrace } + x_i ( q_i - \mathbbm{1}_{\lbrace y \leq x_i \rbrace} ) \, \mathrm{d}F(y) ,
\end{align*}
which shows that the restricted functional reduces to an expectation. It is therefore elicitable and (strictly) $\mathcal{F}_{(x_1, \ldots, x_{k-1})}$-consistent scoring functions for $T'$ are given in Theorem~\ref{Thm:GenExpectElicit}. Consequently, $T$ is conditionally elicitable of order $k$.
\end{proof}

\begin{remark}  \label{Rem:PlugInFunctionals}
There is one single way to describe both the variance functional as well as the spectral risk measure functional. To see this, let $T_i : \mathcal{F} \rightarrow \mathsf{A}_i$, $i=1, \ldots, k-1$ be elicitable functionals and define the function $H : \mathsf{A}_1 \times \ldots \times \mathsf{A}_{k-1} \times \mathsf{O} \rightarrow \mathbb{R} $. Moreover, suppose that $H(x_1, \ldots, x_{k-1}, \cdot) $ is $\mathcal{F}$-integrable for any choice of $x_i \in \mathsf{A}_i$, $i=1,\ldots, k-1$. Then the functional $T$ considered in Example~\ref{Thm:ExVarianceCondElicit} and Lemma~\ref{Thm:SpectralCondElicit} can be written as $T(F) = \bar{H}(T_1(F), \ldots, T_{k-1} (F), F)$ for a certain choice of $H$ and $T_i$. This representation guarantees that the restricted functional reduces to an expected value in both cases and is thus a sufficient condition for conditional elicitability of order $k$.
\end{remark}

The next example shows conditional elicitability of a functional which is different from the functionals considered in Example~\ref{Thm:ExVarianceCondElicit} and Lemma~\ref{Thm:SpectralCondElicit}. It shows that Definition~\ref{Def:CondElicitability} is not only applicable to functionals $T$ for which $T(F)$ is an expectation, but also to certain transformations of elicitable functionals.

\begin{example}[Sum of functionals]
Let $\mathcal{F}$ be a class of distribution functions. Define two elicitable functionals $T_i : \mathcal{F} \rightarrow \mathsf{A}_i$, $i=1,2$, and set $T := T_1 + T_2$. For any $z \in \mathsf{A}_2$ define the subclass $\mathcal{F}_z := \lbrace F \in \mathcal{F} \mid T_2(F) = z \rbrace$. This gives $T_{\vert \mathcal{F}_z } (F) = T_1(F) + z$, which is an affine transformation of the elicitable functional $T_1$. Consequently, as noted in Remark~\ref{Rem:RevelationPrinciple}, the revelation principle immediately implies that $T_{\vert \mathcal{F}_z }$ is elicitable and thus $T$ is conditionally elicitable of order 2. This argument can naturally be extended to show that the sum of $k$ different functionals is conditionally elicitable of order $k$. Instead of sums of elicitable functionals, other transformations can also be considered as long as $T_{\vert \mathcal{F}_z }$ can be connected to $T_1$ via a bijection in order to apply the revelation principle.
\end{example}

As mentioned in \cite{FissZieg} without proof, every conditionally identifiable functional $T$ is part of an identifiable vector of functionals. Hence, we say that conditional identifiability implies joint identifiability. We state the precise result and prove it in the next proposition, using an additional integrability assumption. This assumption is needed, because Definition~\ref{Def:CondElicitability} only guarantees the existence of $\mathcal{F}_{ (x_1, \ldots, x_{k-1}) }$-identification functions. That these functions are integrable for the whole class $\mathcal{F}$ is not ensured, but necessary for $\mathcal{F}$-identification functions.

\begin{prop}
Let $T : \mathcal{F} \rightarrow \mathsf{A}$ be conditionally identifiable of order $k$. Moreover, suppose that for each subclass $\mathcal{F}_{(x_1, \ldots, x_{k-1}) }$ as defined in (\ref{Eqn:DefRestrictClass}) there exists an $\mathcal{F}$-integrable strict identification function. Then there is an identifiable functional $T' : \mathcal{F} \rightarrow \mathsf{A}' \subseteq \mathbb{R}^k$ such that $ T_k' = T$.
\end{prop}

\begin{proof}
Let $T$ be conditionally identifiable with corresponding functionals $T_i : \mathcal{F} \rightarrow \mathsf{A}_i \subseteq \mathbb{R}$ and  strict $\mathcal{F}$-identification functions $V_i$ for $T_i$, $i=1, \ldots, k-1$. Moreover, let $\mathcal{F}_{ (x_1, \ldots, x_{k-1}) } $ be defined as in Equation~(\ref{Eqn:DefRestrictClass}). Then for any $(x_1, \ldots, x_{k-1})^\top \in \mathsf{A}_1 \times \ldots \times \mathsf{A}_{k-1} $ there exists a strict $\mathcal{F}_{ (x_1, \ldots, x_{k-1}) }$-identification function 
\begin{equation*}
V: \mathsf{A} \times \mathsf{O} \rightarrow \mathbb{R}, \quad (z,y) \mapsto V(x_1, \ldots, x_{k-1} ; z, y)
\end{equation*}
for $T$. By assumption, it is even possible to choose $V$ such that it is $\mathcal{F}$-integrable. We thus define the functional
\begin{equation*}
T' : \mathcal{F} \rightarrow  \mathsf{A}_1 \times \ldots \times \mathsf{A}_{k-1} \times \mathsf{A} , \quad F \mapsto T'(F) := (T_1(F), \ldots , T_{k-1} (F) , T(F) )^\top ,
\end{equation*}
which has a strict $\mathcal{F}$-identification function given by
\begin{align*}
V' &: \mathsf{A}_1 \times \ldots \times \mathsf{A}_{k-1} \times \mathsf{A} \times \mathsf{O} \rightarrow \mathbb{R}^k , \\
(x,y) &\mapsto V'(x,y) := (V_1(x_1,y), \ldots , V_{k-1}(x_{k-1}, y) , V(x_1, \ldots, x_{k-1} ; x_k , y) )^\top .
\end{align*}
To see this, observe that the first $k-1$ components of $\bar{V}'(x,F)$ are zero if and only if $F \in \mathcal{F}_{ (x_1, \ldots, x_{k-1}) } $ is satisfied. But for such $F$, $ \bar{V}(x_1, \ldots, x_{k-1} ; x_k, F)$ is zero if and only if $x_k = T(F)$, showing that $T'$ is as desired.
\end{proof}

The examples of variance and ES raise the question whether conditional elicitability implies joint elicitability. Unfortunately, it is not possible to apply the technique of the previous proof to scoring functions. To see this, let $S(x_1, \ldots, x_{k-1} ; \cdot, \cdot)$ be a strictly $\mathcal{F}_{ (x_1, \ldots, x_{k-1}) }$-consistent scoring function for $T$. When concerned with identification, it is irrelevant how $\bar{V}(x_1, \ldots, x_{k-1}; \cdot , F )$ behaves for $F \notin \mathcal{F}_{ (x_1, \ldots, x_{k-1}) } $. For scoring functions however, the values of $\bar{S}(x_1, \ldots, x_{k-1}, \cdot, F)$ for $F \notin \mathcal{F}_{ (x_1, \ldots, x_{k-1}) }$ might be high or low, making it difficult (or even impossible) to establish consistency. Nevertheless, there is also no known functional which disproves this conjecture, hence, as stated in \cite{FissZieg}, this is an open question.

\subsection{Non-elicitable functionals with special structure} 
\label{Sec:GeneralizedFZ}

This section continues to discuss variance and Expected Shortfall and presents a unified approach which can be used to show joint elicitability in both cases. Recall that Remark~\ref{Rem:PlugInFunctionals} discusses functionals which can be represented as $T(F) = \bar{H} (T_1(F), \ldots, T_{k-1}(F),F)$ for elicitable functionals $T_1, \ldots, T_{k-1}$ and an integrable function $H$. This subsection begins by showing that such functionals are not only conditionally elicitable, but even part of an elicitable vector of functionals if $H$ is a strictly consistent scoring function for $(T_1, \ldots, T_{k-1})^\top$. This fact is then applied to prove a generalization of Theorem~\ref{Thm:SpectralRiskElicitable}.

\begin{remark}
A result similar to the one in the following proposition is stated by Frongillo and Kash~\cite{FrongilloKashON}. The author of this thesis found this reference a few days before submission. The formulation of the proposition, the stated proof, and all following results were developed independently of~\cite{FrongilloKashON}.
\end{remark}

\begin{prop}  \label{Thm:FuncPlusMinScore}
Let $T_1 : \mathcal{F} \rightarrow \mathsf{A}_1 \subseteq \mathbb{R}^{k-1}$ be an elicitable functional with $\mathcal{F}$-consistent scoring functions $S_1, \ldots, S_n$. Define a second functional $T_2$ via
\begin{equation*}
T_2 : \mathcal{F} \rightarrow \mathsf{A}_2 \subseteq \mathbb{R}^n,  \quad F \mapsto (\bar{S}_1 (T_1(F) , F), \ldots, \bar{S}_n(T_1(F), F))^\top .
\end{equation*}
If $S_1, \ldots, S_n$ are all strictly $\mathcal{F}$-consistent, the functional $T : \mathcal{F} \rightarrow \mathsf{A}_1 \times \mathsf{A}_2$, $F \mapsto (T_1(F), T_2(F))^\top$ is elicitable. Moreover,  $\mathcal{F}$-consistent scoring functions for $T$ are given by
\begin{align*}
S (x, y) &=  - \sum_{i=1}^{n} \partial_i f(x_k, \ldots, x_{k+n-1}) (S_i(x_1, \ldots, x_{k-1}, y) - x_{k+i-1}) \\
&\quad-f(x_k, \ldots, x_{k+n-1}) + \sum_{i=1}^{n} c_i S_i(x_1, \ldots, x_{k-1}, y)  , 
\end{align*}
where $c \in \mathbb{R}^n$ and $f : \mathsf{A}_2 \rightarrow \mathbb{R}$ is a differentiable and convex function such that $ \partial_i f \leq c_i$ holds for $i=1, \ldots, n$. $S$ is strictly consistent if $S_1, \ldots, S_n$ are strictly consistent, $f$ is strictly convex, and $\partial_i f < c_i$ is satisfied.
\end{prop}

\begin{remark}
Note that it is always possible to find a function $f$ which satisfies the requirements of Proposition~\ref{Thm:FuncPlusMinScore}. In the one-dimensional case we pick any $c \geq 0$, and a strictly convex function $f$ satisfying $f' < c$ is given by $ x \mapsto \exp(-x)$. Another choice for $c > 0$ is given by $ g(x) = c x^2 /(1 + \vert x \vert)$, see also~\cite[Corollary 2.16]{FissZiegArxiv}. Similarly, if $n>1$, we select $c \in \mathbb{R}^n$ with positive components and choose $x \mapsto \exp( - \sum_{i=1}^{n} x_i )$. If $c$ has strictly positive components, we again use $g$ to construct $x \mapsto g(\Vert x \Vert)$. This mapping is strictly convex due to the strict convexity of $g$ and $\Vert \cdot \Vert$ and the fact that $g$ is strictly increasing on $(0, \infty)$.
\end{remark}

\begin{proof}
Firstly, observe that $\mathcal{F}$-integrability of $S$ follows from $\mathcal{F}$-integrability of $S_1, \ldots, S_n$. To show consistency, fix $F \in \mathcal{F}$, define $t:= T(F) = (T_1(F), T_2(F))^\top$, and choose $ x \in  \mathsf{A}_1 \times \mathsf{A}_2$. Moreover, fix $c \in \mathbb{R}^n$ and let $f$ be a convex function such that $\partial_i f \leq c_i$ for $i \in I_n := \lbrace 1, \ldots,n \rbrace$ holds. Note that the definition of $t$ implies  $\bar{S}_i ( t_1, \ldots, t_{k-1}, F) = t_{k+i-1}$ for any $i \in I_n$. This gives
\begin{equation*}
\bar{S}(x,F) - \bar{S}(t,F) = f(t_k, \ldots, t_{k+n-1}) - f(x_k, \ldots, x_{k+n-1}) + \sum_{i=1}^{n} \mathsf{R}_i , 
\end{equation*}
where for any $i \in I_n$ it holds that
\begin{align}  \label{Eqn:MinScoreEst} 
\mathsf{R}_i &= - \partial_i f(x_k, \ldots, x_{k+n-1})(  \bar{S}_i(x_1, \ldots, x_{k-1}, F) - x_{k+i-1}) \nonumber \\
&\quad+ \partial_i f(t_k, \ldots, t_{k+n-1}) ( \bar{S}_i (t_1, \ldots, t_{k-1}, F) - t_{k+i-1}) \nonumber \\
&\quad+ c_i ( \bar{S}_i (x_1, \ldots, x_{k-1},F) - \bar{S}_i (t_1, \ldots, t_{k-1}, F) ) \nonumber \\
&= - \partial_i f(x_k, \ldots, x_{k+n-1}) ( t_{k+i-1} - x_{k+i-1}) \nonumber \\
&\quad+ (c_i - \partial_i f(x_k, \ldots, x_{k+n-1}))( \bar{S}_i (x_1, \ldots, x_{k-1},F) - \bar{S}_i (t_1, \ldots, t_{k-1}, F) ) \nonumber \\
&\geq - \partial_i f(x_k, \ldots, x_{k+n-1}) ( t_{k+i-1} - x_{k+i-1}) .
\end{align}
Notice that Inequality~(\ref{Eqn:MinScoreEst}) follows from $\partial_i f \leq c_i$ and the fact that $S_i$ is a scoring function for $T_1$ for $i \in I_n$. The convexity of $f$ gives $\sum_{i=1}^{n} \mathsf{R}_i  \geq f(x_k, \ldots, x_{k+n-1}) -f(t_k, \ldots, t_{k+n-1})$, proving that $S$ is an $\mathcal{F}$-consistent scoring function for $T$. If we assume that $S_1, \ldots, S_n$ are strictly $\mathcal{F}$-consistent and $\partial_i f < c_i$, $i\in I_n$, is satisfied, we obtain a strict inequality in (\ref{Eqn:MinScoreEst}). If we additionally assume that $f$ is strictly convex, then also the estimate for $\sum_{i=1}^{n} \mathsf{R}_i$ is strict, hence, these three conditions ensure strict consistency of $S$.
\end{proof}

In the following we only consider the interesting special case $n=1$. In this situation, the functional $T$ takes the simpler form
\begin{equation*}
T : \mathcal{F} \rightarrow  \mathsf{A}_1 \times \mathsf{A}_2 , \quad F \mapsto T(F) = (T_1 (F) , \bar{S}_1 (T_1(F), F) )^\top
\end{equation*}
and the (strictly) consistent scoring functions $S$ for $T$ simplify to
\begin{equation*}
S(x,y) = - f(x_k) - f'(x_k) (S_1 (x_1, \ldots, x_{k-1}, y) - x_k) + cS_1 (x_1, \ldots, x_{k-1}, y) ,
\end{equation*}
where $c \in \mathbb{R}$ and $f : \mathsf{A}_2 \rightarrow \mathbb{R}$ is a differentiable and convex function such that $ f' \leq c$ holds. The rest of this section applies Proposition~\ref{Thm:FuncPlusMinScore} in order to show the joint elicitability of the two functionals variance and Expected Shortfall.

\begin{example}
Let $\mathcal{F}$ be a class of distribution functions with finite second moments. If we define $T_1 : \mathcal{F} \rightarrow \mathbb{R}$, $ F \mapsto \int y \, \mathrm{d}F(y)$, a strictly $\mathcal{F}$-consistent scoring function for $T_1$ is given by $S(x,y) := (x-y)^2$. The functional $T_2 (F) := \bar{S}(T_1(F), F)$ is then the variance functional, which is not elicitable (see Example~\ref{Thm:ExVariNotElicitable} and Remark~\ref{Rem:FnotElicitable}). However, the functional $(T_1, T_2)^\top$ is elicitable, which is shown in Example~\ref{Thm:ExMAndVElicitable}, but also follows directly from Proposition~\ref{Thm:FuncPlusMinScore}. 
\end{example}

\begin{remark}
It is possible to consider an equivalent formulation of the consistent scoring functions of Proposition~\ref{Thm:FuncPlusMinScore}. For any choice of strictly convex $f : \mathsf{A}_2 \rightarrow \mathbb{R}$ and $c \in \mathbb{R}$ with $f^\prime < c$ we define $x \mapsto \tilde{f}(x) := f(x) - c x$ and this function is again strictly convex and also strictly decreasing. Moreover, we may add any consistent scoring function for $T_1$ to the scoring function $S$ without affecting its consistency. Hence, the (strictly) consistent scoring functions for $T$ given in Proposition~\ref{Thm:FuncPlusMinScore} take the form
\begin{equation*}
S(x, y) = -f(x_k) - f^\prime (x_k) ( S_1(x_1, \ldots, x_{k-1}, y) - x_k) + S_* (x_1, \ldots, x_{k-1}, y) ,
\end{equation*}
where $f : \mathsf{A}_2 \rightarrow \mathbb{R}$ is differentiable, strictly convex, and strictly decreasing and $S_* : \mathsf{A}_1 \times \mathsf{O} \rightarrow \mathbb{R}$ is another consistent scoring function for $T_1$. If $T_1$ is also identifiable with strict identification function $V_1$, it is clear that $V : \mathsf{A}_1 \times \mathsf{A}_2 \times \mathsf{O} \rightarrow \mathbb{R}^2$ given by
\begin{equation}  \label{Eqn:IdentFuncPlusMin}
V(x, y) = ( V_1(x_1, \ldots, x_{k-1}, y) , S_1(x_1, \ldots, x_{k-1}, y) - x_k )^\top
\end{equation}
is a strict $\mathcal{F}$-identification function for $T$. Note that this construction does not require $S$ to be a scoring function. 
\end{remark}

Proposition~\ref{Thm:FuncPlusMinScore} is not only of theoretical interest, since it can also be used to extend the joint elicitability of spectral risk measure functionals to classes $\mathcal{F}$ containing discontinuous distribution functions. This extension is based on the fact that the Expected Shortfall at level $\alpha \in (0,1)$ of $F$ is the minimizer of $\bar{S}(\cdot, F)$, where $S$ is the (strictly) consistent scoring function for the $\alpha$-quantile determined by the choice $g(x) = x$ in Theorem~\ref{Thm:QuantilesElicitable}. This argument is also used in the proof of Theorem~\ref{Thm:EScoherent} and is well-known in risk management (see for instance Rockafellar and Uryasev~\cite[Sec. 4]{Rockafellar}).

\begin{cor}  \label{Thm:SpectralRiskElicitableII}
Let $\mathcal{F}$ be a class of distribution functions having finite first moments. Define the functionals $T_1, \ldots, T_{k-1}, T_k$ and $T$ as in Theorem~\ref{Thm:SpectralRiskElicitable}, but on the larger class $\mathcal{F}$. If all members of $\mathcal{F}$ have unique $q_i$-quantiles, then $T$ is elicitable. Moreover, an $\mathcal{F}$-consistent scoring function $S : \mathsf{A} \times \mathsf{O} \rightarrow \mathbb{R}$ for $T$ is given by
\begin{align*}
S(x,y) = &-f(x_k) + f'(x_k) \left( x_k + \sum_{i=1}^{k-1} \frac{p_i}{q_i} \big( (\mathbbm{1}_{ \lbrace y \leq x_i \rbrace } - q_i) x_i -  \mathbbm{1}_{ \lbrace y \leq x_i \rbrace } y \big) \right) \\
 &+ c \sum_{i=1}^{k-1} \frac{p_i}{q_i} \big( (\mathbbm{1}_{ \lbrace y \leq x_i \rbrace } - q_i) x_i -  \mathbbm{1}_{ \lbrace y \leq x_i \rbrace } y \big) ,
\end{align*}
where $c \in \mathbb{R}$ and $f$ is convex and differentiable and satisfies $-f' \leq c$. If all $q_i$-quantiles are unique and $f$ is strictly convex with $-f' < c$, then $S$ is strictly $\mathcal{F}$-consistent.
\end{cor}

\begin{proof}
For any $i \in I_{k-1} := \lbrace 1, \ldots, k-1 \rbrace$ we have $q_i \in (0,1)$ and $T_i(F) = F^{\leftarrow}(q_i)$ and set $\mathsf{A}_i := T_i (\mathcal{F})$. Theorem~\ref{Thm:QuantilesElicitable} and the fact that all members of $\mathcal{F}$ have finite first moments imply that the function $\tilde{S}_i(x,y) = (\mathbbm{1}_{ \lbrace y \leq x \rbrace } - q_i) (x-y)$ is an $\mathcal{F}$-consistent scoring function for $T_i$. By Lemma~\ref{Thm:TransformOfScoring} (i), it is possible to first subtract the function $y \mapsto q_i y$ from $\tilde{S}_i$ and then scale it with $p_i / q_i$ and the result remains $\mathcal{F}$-consistent. Consequently, for any $i \in I_{k-1}$, the function 
\begin{equation*}
S_i : \mathsf{A}_i \times \mathsf{O} \rightarrow \mathbb{R}, \quad (x,y) \mapsto S_i(x,y) = \frac{p_i}{q_i} (\mathbbm{1}_{ \lbrace y \leq x \rbrace } - q_i) x -  \frac{p_i}{q_i} \mathbbm{1}_{ \lbrace y \leq x \rbrace } y
\end{equation*}
is $\mathcal{F}$-consistent for $T_i$. It is strictly consistent if the $q_i$-quantile is unique for all $F \in \mathcal{F}$. It follows from Lemma~\ref{Thm:AllCompElicitable} that the function $S' (x,y) := \sum_{i=1}^{k-1} S_i(x_i, y)$ is an $\mathcal{F}$-consistent scoring function for the functional $(T_1, \ldots, T_{k-1})^\top$. The next step is to represent the spectral risk measure functional $T_k$ by using $S'$. Equation~(\ref{Eqn:ESAcerbiDefinition}) leads to the representation
\begin{align*}
-T_k (F) &= \sum_{i=1}^{k-1} \mathrm{ES}_{q_i} (F) p_i \\
&= \sum_{i=1}^{k-1} \frac{p_i}{q_i} \int_{\mathsf{O}} - \mathbbm{1}_{ \lbrace y \leq F^{\leftarrow}(q_i) \rbrace } y + (\mathbbm{1}_{ \lbrace y \leq F^{\leftarrow}(q_i) \rbrace } - q_i) F^{\leftarrow}(q_i)  \, \mathrm{d}F(y) \\ 
&= \int_{\mathsf{O}} \sum_{i=1}^{k-1} \frac{p_i}{q_i} (\mathbbm{1}_{ \lbrace y \leq F^{\leftarrow}(q_i) \rbrace } - q_i) F^{\leftarrow}(q_i) - \frac{p_i}{q_i} \mathbbm{1}_{ \lbrace y \leq F^{\leftarrow}(q_i) \rbrace } y \, \mathrm{d}F(y) \\
&= \int_{\mathsf{O}} \sum_{i=1}^{k-1} S_i ( F^{\leftarrow}(q_i), y) \, \mathrm{d}F(y) \\
&= \bar{S}'( F^{\leftarrow}(q_1), \ldots, F^{\leftarrow}(q_{k-1}) , F ) ,
\end{align*}
which shows that Proposition~\ref{Thm:FuncPlusMinScore} is applicable with $n=1$ and gives $\mathcal{F}$-consistent scoring functions for the functional $T' := (T_1, \ldots, T_{k-1} , -T_k)^\top$. Choosing a convex and differentiable function $\tilde{f}$ and a constant $c \in \mathbb{R}$ such that $\tilde{f}' \leq c$ holds, implies that
\begin{align*}
S''(x,y) :=  - \tilde{f}(x_k) - \tilde{f}'(x_k) ( S'(x_1, \ldots, x_{k-1}, y) - x_k) + c S'(x_1, \ldots, x_{k-1}, y)
\end{align*}
is $\mathcal{F}$-consistent for $T'$. An application of the revelation principle as stated in Proposition~\ref{Thm:RevelationPrinciple} leads to an $\mathcal{F}$-consistent scoring function for $(T_1, \ldots, T_{k-1}, T_k)^\top$ given by
\begin{equation*}
S(x,y) :=  -\tilde{f}(-x_k) - \tilde{f}'(-x_k) ( S'(x_1, \ldots, x_{k-1}, y) + x_k) + c S'(x_1, \ldots, x_{k-1}, y) . 
\end{equation*}
If for all $F \in \mathcal{F}$ all $q_i$-quantiles are unique, strict $\mathcal{F}$-consistency carries over from $S_i$, $i \in I_{k-1}$ to $S'$ by Lemma~\ref{Thm:AllCompElicitable}. If we additionally assume that $\tilde{f}$ is strictly convex and satisfies $ \tilde{f}' < c$, strict $\mathcal{F}$-consistency carries over from $S'$ to $S''$ by Proposition~\ref{Thm:FuncPlusMinScore} and from $S''$ to $S$ by the revelation principle. Finally, if we define the (strictly) convex function $f$ via $f(x) = \tilde{f}(-x)$ and require $-f'  \leq c$ or $-f' < c$ instead of $\tilde{f}' \leq c$ or $ \tilde{f}' < c$, respectively, the function $S$ remains (strictly) $\mathcal{F}$-consistent and has the desired representation.
\end{proof}

\begin{remark}
We compare the strictly consistent scoring function of Theorem~\ref{Thm:SpectralRiskElicitable} to the one of Corollary~\ref{Thm:SpectralRiskElicitableII} and call them $S_1$ and $S_2$, respectively. At first we argue that the strict consistency of $S_2$ can be shown using Theorem~\ref{Thm:SpectralRiskElicitable}. To this end, let a strictly convex $f$ and $c \in \mathbb{R}$ be given and set $G_k = f$ as well as $g_r (v) = c (p_r / q_r ) v$ for $r=1, \ldots, k-1$. Since $g_k = f' > -c$ is fulfilled, we obtain that the function
\begin{equation*}
H_{r,u} (v) = v \frac{p_r}{q_r} g_k (u) + g_r (v) = v \frac{p_r}{q_r} ( g_k (u) + c )
\end{equation*}
is strictly increasing in $v$, hence Theorem~\ref{Thm:SpectralRiskElicitable} is applicable.

Conversely, we consider if strict consistency of $S_1$ can be shown using Corollary~\ref{Thm:SpectralRiskElicitableII}. To this end, let strictly increasing functions $g_1, \ldots, g_{k-1}$ and a strictly convex $G_k$ be given. In order to obtain the same representation as in Corollary~\ref{Thm:SpectralRiskElicitableII}, it is necessary to set $c=0$, $f=G_k$ and add the function
\begin{equation*}
S'(x_1, \ldots, x_{k-1},y) := \sum_{i=1}^{k-1}  ( \mathbbm{1}_{\lbrace y \leq x_i \rbrace} - q_i ) g_i (x_i) - \mathbbm{1}_{\lbrace y \leq x_i \rbrace} g_i (y) ,
\end{equation*}
which is a strictly consistent scoring function for the functional $(T_1, \ldots, T_{k-1})^\top$ as defined in Theorem~\ref{Thm:SpectralRiskElicitable}. Due to $c=0$ it is also needed that $g_k = f' > 0$ is fulfilled. However, the requirement that $H_{r,u}$ is strictly increasing does not suffice to guarantee this inequality. Hence, Corollary~\ref{Thm:SpectralRiskElicitableII} cannot be used in general to show the strict consistency of $S_1$, showing that the classes of strictly consistent scoring functions obtained from both results are different. This is intuitive, since Theorem~\ref{Thm:SpectralRiskElicitable} is designed to handle spectral risk measure functionals and uses the condition on $H_{r,u}$ to exploit the structure of strictly consistent scoring functions for quantiles. In contrast, the scoring functions of Corollary~\ref{Thm:SpectralRiskElicitableII} are derived by using Proposition~\ref{Thm:FuncPlusMinScore}, a result which holds for a larger class of functionals.
\end{remark}

It would be desirable to conclude this section with a characterization of all strictly consistent scoring functions for the functional $T$ given in Proposition~\ref{Thm:FuncPlusMinScore}. One idea to do this is to use the strict identification function given in (\ref{Eqn:IdentFuncPlusMin}) together with Osband's principle, similar to Section~\ref{Sec:FunctionalsComponents} or \cite[Thm. 5.2 (iii)]{FissZieg}. Unfortunately, a straightforward adaptation of arguments used in the proofs of these results is not fruitful, since we neither have a `simple' identification function as in Section~\ref{Sec:FunctionalsComponents}, nor can we exploit special properties of $S$ as in \cite[Thm. 5.2 (iii)]{FissZieg}. A suitable approach to tackle this problem is yet to be found.

\section{Discussion} 

In this thesis we reviewed results concerning higher order elicitability and identifiability and studied the connection of these concepts to quantitative risk management. We discussed the properties of the three risk measures Value at Risk, Expected Shortfall and Expectile Value at Risk and reached the conclusion that Expected Shortfall, or more general spectral measures of risk, possess the most satisfying properties. Taking elicitability into account, we showed that only the other two risk measures have this property. Intuitively, the reason for this is that the two properties of coherence and elicitability are conflicting, a fact already recognized by Weber~\cite{WeberConsistency}. Interestingly, Expectile Value at Risk is the only risk measure which manages the balancing act between both requirements. Nevertheless, Corollary~\ref{Thm:JointElicitVaRES} showed that moving from one-dimensional elicitability to higher order elicitability solves this problem, in the sense that Expected Shortfall is jointly elicitable with Value at Risk. More generally, Theorem~\ref{Thm:SpectralRiskElicitable} proves that all spectral measures of risk with discrete spectrum are part of an elicitable vector. The motivation for these results is to be able to perform comparative backtesting for Expected Shortfall estimates. However, joint elicitability implies that only joint estimates of Expected Shortfall and Value at Risk can be compared. The differences of such a test to a comparison of one-dimensional estimates demand further attention. Closely connected is the problem of choosing a suitable strictly consistent scoring function in applications.
\par
On the theoretical side of this thesis, we began by discussing new as well as classical results related to elicitability and identifiability. The most important one is the multivariate version of Osband's principle which is due to Fissler and Ziegel~\cite{FissZieg}. We presented a version on the level of expectations in Theorem~\ref{Thm:OsbandPrinciple1} and a pointwise version in Theorem~\ref{Thm:OsbandPrinciple2}. By discussing and proving the latter, we drew attention to the difficulties appearing when moving from an integrated to a pointwise version. We proposed additional assumptions on the identification function $V$ as well as on the domain $\mathsf{A}$ and illustrated why they are reasonable. Nevertheless, we applied both versions of Osband's principle in order to show a characterization of functionals with elicitable components which was stated in \cite{FissZieg}. We continued by discussing further ideas related to non-elicitable functionals, one of which is conditional elicitability. While we are able to prove that conditional identifiability implies joint identifiability, we also recognize that a similar statement for elicitability requires more work. It is thus an interesting open question which requirements are needed such that conditionally elicitable functionals are also jointly elicitable. Moreover, we proved Proposition~\ref{Thm:FuncPlusMinScore}, which allowed us to argue that variance as well as Expected Shortfall are jointly elicitable because they are mean scores of elicitable functionals, namely mean and Value at Risk. In Corollary~\ref{Thm:SpectralRiskElicitable} we used this result to extend the joint elicitability of spectral risk measure functionals to classes containing discontinuous distribution functions. \par

Finally, there are many other topics as well as open problems which were not discussed in this thesis. One of the most important ones is definitely the problem of characterizing all elicitable functionals, possibly subject to some regularity conditions. The one-dimensional version of such a characterization was mentioned in Remark~\ref{Rem:LevelSetSufficient}, but not discussed in detail, since the proof of this result (as done in Steinwart et al.~\cite{SteinwartPasinWilliam}) is extensive enough to fill a separate thesis. In the multi-dimensional case, little is known on how elicitable functionals can be characterized, since the conditions of the one-dimensional case are not sufficient as shown by Frongillo and Kash~\cite[Example 1]{FrongilloKash}. Another important problem is to find necessary and sufficient conditions for non-elicitable functionals to be jointly elicitable. We mentioned a class of functionals having this property in Subsection~\ref{Sec:GeneralizedFZ}, but a thorough study of this topic does not exist yet. Maybe a starting point to study this problem consists of establishing general conditions such that the functionals defined in Remark~\ref{Rem:PlugInFunctionals} are part of an elicitable vector.


\begin{appendix}
\chapter{Auxiliary results}

In this chapter, we provide some results which are needed throughout the thesis, but may interfere with the flow of reading. If not stated otherwise, all proofs and examples are our own contribution.

\begin{theorem}[Skorohod representation]  \label{Thm:SkorohodRep}
Let $(X_n)_{n \in \mathbb{N}}$ be a sequence of real-valued random variables such that $X_n \rightarrow X$ in distribution as $n \rightarrow \infty$. Then there exist a probability space $(\Omega, \mathscr{A}, \mathbb{P})$ and real-valued random variables $Y, Y_1, Y_2, \ldots$ defined on it such that under $\mathbb{P}$
\begin{enumerate}[label=(\roman*)]
	\item $Y =^d X$, $Y_n =^d X_n$ for all $n \in \mathbb{N}$, and
	\item $Y_n \rightarrow Y$ almost surely.
\end{enumerate}
\end{theorem}

\begin{proof}
See for example van der Vaart~\cite[Thm. 2.19]{vanderVaart} or the more general version in Kallenberg~\cite[Thm. 4.30]{KallenbergProb}.
\end{proof}

\begin{lemma}  \label{Thm:AppendixConvexHull}
Let $0 \in \intr ( \conv (x_1, \ldots , x_{k+1} ) )$ for $x_i \in \mathbb{R}^k$, $i= 1, \ldots, k+1$ and fix an arbitrary $x_{k+2} \in \mathbb{R}^k$. Then there exist coefficients $\lambda_i >0$, $i=1, \ldots, k+2$ such that $0 = \sum_{i=1}^{k+2} \lambda_i x_i$ holds.
\end{lemma}

\begin{proof}
Due to $0 \in \intr ( \conv (x_1, \ldots , x_{k+1} ) )$ there are $\gamma_i > 0$, $i \in I_{k+1} := \lbrace 1, \ldots, k+1 \rbrace$ such that $\sum_{i=1}^{k+1} \gamma_i x_i = 0$ is fulfilled. Simultaneously, this condition implies that the set $\lbrace x_1, \ldots, x_{k+1} \rbrace$ contains $k$ linearly independent vectors and hence spans $\mathbb{R}^k$. We thus find some $\beta_i$, $i \in I_{k+1}$ such that the representation $x_{k+2} = \sum_{i=1}^{k+1} \beta_i x_i$ holds and use its coefficients to define $\lambda_{k+2} := \min \lbrace \gamma_i / \vert \beta_i \vert \mid \beta_i \neq 0, i \in I_{k+1} \rbrace / 2$. Moreover, we define $\lambda_i := \gamma_i - \lambda_{k+2} \beta_i$ for $i \in I_{k+1}$. For any $i \in I_{k+1}$ such that $\beta_i > 0$ we thus have $\lambda_i = \gamma_i - \lambda_{k+2} \beta_i \geq \gamma_i - \frac{\gamma_i \beta_i}{2 \vert \beta_i \vert} > 0$ and hence $\lambda_i > 0$ for all $i \in I_{k+1}$. Finally we calculate
\begin{equation*}
\sum_{i=1}^{k+2} \lambda_i x_i = \lambda_{k+2} x_{k+2} + \sum_{i=1}^{k+1} \gamma_i x_i - \lambda_{k+2} \sum_{i=1}^{k+1} \beta_i x_i = 0 ,
\end{equation*}
which proves that $\lambda_i$, $i \in I_{k+2}$ is as desired.
\end{proof}

For the following two results we remark that all matrix norms are equivalent, hence we do not explicitly choose one.

\begin{lemma}  \label{Thm:AppendixLinearInd}
For $k,d \in \mathbb{N}$, let $\mathbb{V} : \mathbb{R}^k \rightarrow \mathbb{R}^{d \times d}$ be a mapping which is continuous at a point $x \in \mathbb{R}^k$ and for which $\mathbb{V}(x)$ is an invertible matrix. Then there is an open set $U_x$ containing $x$ and such that $\mathbb{V}(z)$ is an invertible matrix for all $z \in U_x$.
\end{lemma}

\begin{proof}
Let $k,d \in \mathbb{N}$ and $\mathbb{V}$ be given and define the mapping $D : \mathbb{R}^k \rightarrow \mathbb{R}, z \mapsto \det ( \mathbb{V} (z) )$. Since $\mathbb{V}(x)$ is invertible we have $D(x) \neq 0$. Moreover, $D$ is a continuous mapping, because $\mathbb{V}$ is continuous and $A \mapsto \det(A)$ is a polynomial in the components of $A$. Now assume by means of contradiction that for any open set $U_x$ containing $x$ there is a $z \in U_x$ for which $\mathbb{V}(z)$ is singular. This implies that there is a sequence $(z_n)_{n \in \mathbb{N}}$ which satisfies $\det ( \mathbb{V}(z_n) ) = 0$ for all $n \in \mathbb{N}$ as well as $z_n \rightarrow x$. Using the continuity of $D$, we obtain $ D(x) = \lim_{n \rightarrow \infty} D(z_n) = 0$, which is a contradiction, and thus the claim is proved.
\end{proof}

\begin{lemma}  \label{Thm:AppendixInverseCont}
Let $k \in \mathbb{N}$ and $\mathrm{GL}(k) \subset \mathbb{R}^{k \times k}$ be the space of regular $k \times k$-matrices. The mapping $I: \mathrm{GL}(k) \rightarrow \mathrm{GL}(k)$ defined via $A \mapsto A^{-1}$ has continuously differentiable components.
\end{lemma}

\begin{proof}
For fixed $k \geq 2$ and $A \in \mathrm{GL}(k)$ the components of $A^{-1}$ can be represented using Cramer's rule (see for example Liesen and Mehrmann~\cite[Corollary 7.23]{LiesenMehrmann}), which sates that
\begin{equation*}
(A^{-1})_{ij} = \frac{ (-1)^{i+j} \det( A_{ -(i,j) } ) }{\det(A)}
\end{equation*}
where $A_{ -(i,j) } \in \mathbb{R}^{k-1 \times k-1}$ is the matrix obtained from $A$ by deleting the $i$-th row and $j$-th column. Since $\det(A) \neq 0$ holds for all $A \in \mathrm{GL}(k)$ we see that each component of $A^{-1}$ is a continuously differentiable mapping.
\end{proof}

The following result states the relation between the lower and upper quantile function.

\begin{lemma}  \label{Thm:AppendixQuantileLR}
Let $X$ be a real-valued random variable and let the lower and upper quantile functions $F^{\leftarrow}$ and $F^{\rightarrow}$ be defined as in Definition~\ref{Def:quantilefunction}. Then for any $\alpha \in (0,1)$ we have $F^{\leftarrow}_{-X} (\alpha) = - F^{\rightarrow}_X (1- \alpha)$.
\end{lemma}

\begin{proof}
For any $\alpha \in (0,1)$ we calculate
\begin{align*}
F^{\leftarrow}_{-X} (\alpha) &= \inf \lbrace z \mid F_{-X} (z) \geq \alpha \rbrace \\
&= \inf \lbrace  z \mid \mathbb{P}(X < -z) \leq 1 - \alpha \rbrace \\
&= - \sup \lbrace z \mid \mathbb{P}(X < z) \leq 1 - \alpha \rbrace \\
&= - \inf \lbrace z \mid \mathbb{P}(X \leq z) > 1 - \alpha \rbrace  = - F^{\rightarrow}_X (1- \alpha) 
\end{align*}
and hence it remains to be shown that 
\begin{equation*}
s := \sup \lbrace z \mid \mathbb{P}(X < z) \leq 1 - \alpha \rbrace = \inf \lbrace z \mid \mathbb{P}(X \leq z) > 1 - \alpha \rbrace  =: i
\end{equation*}
is true. To prove this, suppose by way of contradiction that $s < i$. Then there is a $z \in (s,i)$ such that $\mathbb{P}( X \leq z) \leq 1-\alpha$ and $\mathbb{P}(X < z) > 1-\alpha$ is satisfied, which is a contradiction. If on the other hand we have $i < s$, then there is a $z \in (i,s)$  such that $\mathbb{P}(X \leq z) > 1-\alpha$ is fulfilled. But then we have for any $\varepsilon >0$ the inequality $\mathbb{P}(X < z + \varepsilon) \geq \mathbb{P}(X \leq z) > 1- \alpha$, which contradicts the definition of $s$.
\end{proof}

\begin{theorem}[Rademacher]  \label{Thm:AppendixRademacher}
If $U \subseteq \mathbb{R}^n$ is open and $f : U \rightarrow \mathbb{R}^k$ is Lipschitz continuous, then $f$ is differentiable almost everywhere in $ U$.
\end{theorem}

\begin{proof}
See for example Federer~\cite[Thm. 3.1.6]{Federer}.
\end{proof}

The following theorem is a version of Klenke~\cite[Thm. 6.28]{KlenkeProb} which gives sufficient conditions for interchanging integration and differentiation.

\begin{theorem}  \label{Thm:AppendixDiff}
For $k,d,m \in \mathbb{N}$, take sets $\mathsf{A} \subset \mathbb{R}^k$ and $\mathsf{O} \subset \mathbb{R}^d$ and equip $\mathsf{O}$ with the Borel $\sigma$-algebra $\mathcal{O}$. Fix a Borel measure $\mu$ on $(\mathsf{O}, \mathcal{O})$ and let $f : \mathsf{A} \times \mathsf{O} \rightarrow \mathbb{R}^m$ be a mapping with the following properties.
\begin{enumerate}[label=(\roman*)]
	\item For all $x \in \mathsf{A}$ the mapping $y \mapsto f(x,y)$ is $\mu$-integrable.
	\item For $\mu$-almost all $y \in \mathsf{O}$ the mapping $f(\cdot, y) : \intr (\mathsf{A}) \rightarrow \mathbb{R}^m$, $x \mapsto f(x,y)$ is continuously differentiable with partial derivatives $\frac{\partial f}{\partial x_i}$, $i= 1, \ldots, k$.
	\item There is a $\mu$-integrable function $h$ such that for all $x \in \mathsf{A}$ and all $i=1, \ldots, k$ we have $\vert \frac{\partial f}{\partial x_i}(x, \cdot) \vert \leq h$ $\mu$-almost everywhere.
\end{enumerate}
Then for any $x \in \intr (\mathsf{A})$ and $i \in \lbrace 1, \ldots, k \rbrace$ the mapping $y \mapsto \frac{\partial f}{\partial x_i} (x, y)$ is $\mu$-integrable and the function $F: \mathsf{A} \rightarrow \mathbb{R}^m , x \mapsto \int_{\mathsf{O}} f(x,y) \,\mathrm{d}  \mu ( y )$ is continuously differentiable with partial derivatives
\begin{equation*}
\frac{\partial F}{\partial x_i} (x) = \int_{\mathsf{O}} \frac{\partial f}{\partial x_i} (x,y) \, \mathrm{d} \mu ( y ).
\end{equation*}
\end{theorem}

\begin{proof}
Proceed as in the proof of \cite[Thm. 6.28]{KlenkeProb} for every partial derivative, using the $k$-dimensional mean value theorem of calculus. Moreover, use dominated convergence for the continuity of the partial derivatives.
\end{proof}

The following example shows that Assumption (iii) of Theorem~\ref{Thm:AppendixDiff} is indeed necessary to ensure differentiability of $F$.

\begin{example} \label{Thm:AppendixExCounterDiff}
This example shows that integration over a family of continuously differentiable functions does not necessarily deliver a continuously differentiable function. More precisely, define $f$ via
\begin{equation*}
f: \mathbb{R} \times (0,1] \rightarrow \mathbb{R}, \quad (x,y) \mapsto f(x,y) := \left( 1 - \frac{x^2}{y^2} \right)^3  \, \mathbbm{1}_{[-y,y]} (x).
\end{equation*}
For fixed $y$, and up to normalization, $x \mapsto f(x,y)$ is known as the triweight kernel used in kernel density estimation. The function $f$ is displayed in Figure~\ref{Fig:ExCounterDiff} for three choices of $y$. For any $y \in (0,1]$, the map $x \mapsto f(x,y)$ is continuously differentiable and for any $x \in \mathbb{R}$, the map $y \mapsto f(x,y)$ is integrable with respect to the Lebesgue measure. If we now define the new function $F$ via

\begin{figure}[ht]
\label{Fig:ExCounterDiff}
\centering
\includegraphics[width= 0.8\textwidth]{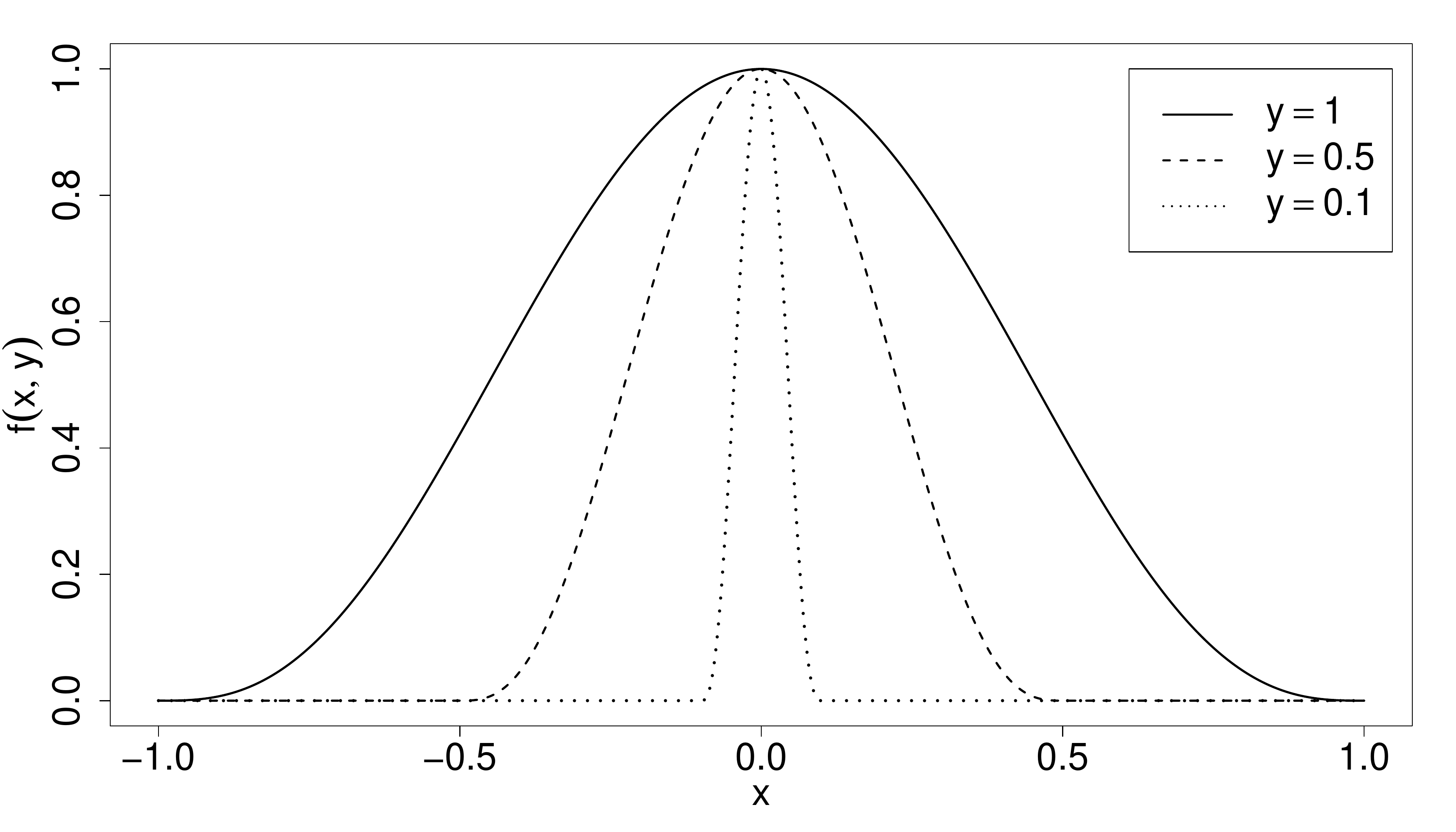}
\caption{Plot of the function $f(x,y)$ from Example~\ref{Thm:AppendixExCounterDiff} for different values of $y$.}
\end{figure}

\begin{equation*}
F: \mathbb{R} \rightarrow \mathbb{R}, \quad x \mapsto F(x) := \int_{0}^{1} f(x,y) \, \text{d} y,
\end{equation*}
then it fails to be differentiable in $x = 0$. In order to prove this, we start by explicitly computing $F$. Taking the indicator function into account, we obtain for any $0<x\leq 1$
\begin{align*}
F(x) = \int_{x}^{1} \left( 1 - \frac{x^2}{y^2} \right)^3 \, \text{d} y &= \left. y + 3 \frac{x^2}{y} - \frac{x^4}{y^3} + \frac{x^6}{5y^5} \, \right|_{y=x}^{y=1} \\
&= 1 - \frac{16}{5} x + 3x^2 - x^4 + \frac{1}{5} x^6 ,
\end{align*}
and analogously for any $-1 \leq x <0$ we calculate
\begin{equation*}
F(x) =  1 + \frac{16}{5} x + 3x^2 - x^4 + \frac{1}{5} x^6.
\end{equation*}
Finally, we observe that for any $x \notin [-1,1]$ it holds that $F(x) = 0$. All in all, we have the representation
\begin{equation*}
F(x) = \left\lbrace
\begin{array}{ll}
1 - \frac{16}{5} x + 3x^2 - x^4 + \frac{1}{5} x^6, & \, \text{if } x \in [0,1] \\
1 + \frac{16}{5} x + 3x^2 - x^4 + \frac{1}{5} x^6, & \, \text{if } x \in [-1,0) \\
0, & \, \text{else } 
\end{array}
\right.
\end{equation*}
and this shows, that $F$ is not differentiable in $x=0$. This happens although for any $x \in \mathbb{R}$ the first partial derivative is integrable with respect to $y$, and it holds that
\begin{equation*}
\int_{0}^{1}  \frac{\partial f}{\partial x} (x,y) \, \mathrm{d} y = - \text{sign}(x) \frac{16}{5}  + 6x - 4x^3 + \frac{6}{5}x^5
\end{equation*}
for $x \in [-1,1] \backslash \{0\}$. But the map $y \mapsto \sup_{x \in \mathbb{R}} \vert \partial f / \partial x \,(x,y) \vert$ is not integrable on $(0,1]$, and therefore differentiability cannot be carried over from $f$ to $F$ using dominated convergence.
\end{example}

\end{appendix}

\backmatter

\chapter{Notation}

\begin{longtable}{p{3cm}l}

$\mathbb{E}X$		& Expectation of the random variable $X$ \\
$\Var (X)$			& Variance of the random variable $X$ \\
$\mathbb{P},P, Q$	& Probability measures \\
$F_P$				& Cumulative distribution function $x \mapsto P( (-\infty, x])$ \\
$\intr(M)$			& Interior of the set $M$ \\
$\conv(M)$			& Convex hull of the set M (see below) \\
$\imge(A)$			& Image of linear mapping $A$ \\
$\ker(A)$			& Kernel of linear mapping $A$ \\
$\dim (A)$			& Dimension of the linear subspace $A \subseteq \mathbb{R}^k$ \\
$\det(M) $			& Determinant of the matrix $M \in \mathbb{R}^{k \times k}$ \\
$I_n$				& Index set $\lbrace 1, 2, \ldots, n \rbrace$ \\
$f(x,\cdot)$		& Mapping $y \mapsto f(x,y)$ \\
$f_{\vert M}$		& Restriction of the mapping $f$ to the set $M$ \\
$\mathbbm{1}_M$		& Characteristic function of the set $M$ \\
$\Vert x \Vert$		& Euclidean norm of $x \in \mathbb{R}^n$, i.e. $\sqrt{x_1^2 + \ldots +  x_n^2}$. \\
$\mathbb{S}^{n-1}$	& $(n-1)$-Sphere in $\mathbb{R}^n$, that is $\lbrace x \in \mathbb{R}^n \mid \Vert x \Vert = 1 \rbrace$. \\
$\ldots \,\mathrm{d}F(x)$ 	& Lebesgue-Stieltjes-Integral with respect to $F$ \\
$\bar{g}(F)$		& Lebesgue-Stieltjes-Integral of $g$ with respect to $F$  \\
$\nabla f$			& Gradient of the function $f$ \\
a.s.				& almost sure \\
a.e.				& almost every \\
w.l.o.g.			& without loss of generality \\
w.r.t.				& with respect to \\
$A \uplus B$		& Union of the sets $A$ and $B$ with $A \cap B = \emptyset$ \\
$F^{\leftarrow}$	& Lower quantile function of distribution function $F$, p. \pageref{Def:quantilefunction} \\
$F^{\rightarrow}$	& Upper quantile function of distribution function $F$, p. \pageref{Def:quantilefunction} \\
$e_\tau (F)$		& $\tau$-expectile of $F$, p. \pageref{Def:Expectiles} \\
$\sigma( \ldots )$	& $\sigma$-algebra generated by collection of mappings or sets \\
$\mathcal{B}(\mathbb{R})$	& Borel $\sigma$-algebra on $\mathbb{R}$ \\
$\mathcal{B}(A)$	& Borel $\sigma$-algebra on $A \subset \mathbb{R}$ \\
$\mathcal{U}([a,b])$	& Continuous uniform distribution on the interval $[a,b]$ \\
$\mathcal{N}(\mu, \sigma^2)$	& Normal distribution with expected value $\mu$ and variance $\sigma^2$\\
$\Phi$				& Distribution function of $\mathcal{N}(0,1)$ \\
$\delta_x$			& Dirac measure at the point $x$ \\
$f^+$				& Positive part of $f$, i.e. $f^+(x) = \max (f(x), 0)$ \\
$f^-$				& Negative part of $f$, i.e. $f^-(x) = \min (f(x), 0)$ \\
$P \otimes Q$		& Product measure of $P$ and $Q$ \\
$f(x+)$				& Right-hand limit of $f$ at $x$ \\
$f(x-)$				& Left-hand limit of $f$ at $x$ \\
$X =^d F$ 			& The random variable $X$ has distribution function $F$ \\
$X =^d Y$			& $X$ and $Y$ have the same distribution \\
$\mathscr{P}(A)$	& Power set of the set $A$ \\
$\mathrm{VaR}_\alpha (X)$ 	& Value at Risk at level $\alpha$ of $X$, p. \pageref{Def:ValueatRisk} \\
$\mathrm{EVaR}_\tau (X)$	& Expectile Value at Risk at level $\tau$ of $X$, p. \pageref{Def:ExpectileVaR} \\
$\mathrm{ES}_\alpha (X)$	& Expected Shortfall at level $\alpha$ of $X$, p. \pageref{Def:ExpectedShortfall} 

\end{longtable}

The \textit{convex hull} of a set $M$ is defined as
\begin{equation*}
\conv(M) := \left\lbrace \sum_{i=1}^{n} \lambda_i x_i \, \Big| \, n \in \mathbb{N}, \, x_1, \ldots, x_n \in M, \, \lambda_1, \ldots, \lambda_n \geq 0, \, \sum_{i=1}^{n} \lambda_i = 1 \right\rbrace .
\end{equation*}

\phantomsection
\addcontentsline{toc}{chapter}{Bibliography}
\bibliographystyle{plain}
\bibliography{ma.bbl}

\end{document}